\documentclass[a4paper,12pt]{amsart}

\pdfoutput=1 

\usepackage[top=3cm,bottom=3cm,outer=3cm,inner=2cm,marginpar=2.45cm]{geometry}
\usepackage[abbrev]{amsrefs}

\usepackage[french,ngerman,english]{babel}
\usepackage[utf8]{inputenc}
\usepackage[T1]{fontenc}

\usepackage[all,pdf]{xy}

\usepackage{enumitem}
\usepackage{mathtools}  
\usepackage[usenames,dvipsnames]{xcolor}
\usepackage{calc} 

\babeltags{de = ngerman}
\babeltags{fr = french}

\usepackage[no-mario]{marionotations}

\usepackage{bbm}     

\usepackage{circledsteps}  

\usepackage[final,colorlinks=true]{hyperref}
\usepackage{breakurl}  

\usepackage{subfiles}

\newcommand{\termKcolor}{NavyBlue}
\newcommand{\termJcolor}{VioletRed}

\newcommand{\McKernan}{M${}^{\text{c}}$Kernan}

\newcommand{\alert}[2][RoyalBlue]{{\color{#1}#2}}

\NewDocumentCommand{\Ltwo}{ 
  m    
  O{X} 
  s    
  m    
}{L^{#1}_{(2)}\paren{\IfBooleanF{#3}{#2;} #4}}

\NewDocumentCommand{\Ltwosp}{
  D//{q}             
  t{M}               
  D||{\vphi}         
  o 
  D<>{\clomega}      
}{\Ltwo{n,#1}*{D \otimes F \IfBooleanT{#2}{\otimes M}}_{#3, \IfNoValueF{#4}{#4,} #5}}

\NewDocumentCommand{\Harm}{ 
  D//{q}             
  D||{\vphi}         
  o 
  D<>{\clomega}      
}{\mathcal{H}^{n,#1}_{#2, \IfNoValueF{#3}{#3,} #4}}

\newcommand{\spHbase}{\mathcal{H}}
\NewDocumentCommand{\spH}{ 
  D//{q}       
  t{M}         
  s            
  O{\IfBooleanTF{#5}{#6 +1}{\sigma}}
  t{|}         
  D<>{\sigma}  
  d()          
}{\spHbase^{n,#1}\IfNoValueTF{#7}{
    \begingroup%
    \newcommand{\upidl}{\IfBooleanTF{#3}{\holo_X}{\defidlof{\lcc[#4]}}}%
    \paren{\IfBooleanT{#2}{M\otimes}
      \IfBooleanTF{#5}{\frac{\upidl}{\defidlof{\lcc[#6]}}}{\upidl}}\endgroup%
  }{\paren{\IfBooleanT{#2}{M\otimes}#7}}}

\DeclareMathOperator{\lc}{lc} 
\NewDocumentCommand{\lcc}{ 
  O{\sigma}                
  D<>{X}                   
  D(){S}
}{\lc_{#2}^{#1}\paren{#3}}

\NewDocumentCommand{\lcS}{  
  s            
  D//{S}       
  D||{\sigma}  
  O{p}         
}{\mathtt{\IfBooleanT{#1}{\rs} #2}^{#3}_{#4}}

\NewDocumentCommand{\PRes}{ 
  O{\lcS}  
  d()      
}{\mathcal R_{#1}\IfNoValueF{#2}{\paren{#2}}}

\newcommand{\defidlof}[1]{\mathcal{I}_{#1}}  
\NewDocumentCommand{\mtidlof}{   
  O{}      
  D<>{#1}  
  m        
}{\multidl_{#2}\paren{#3}} 

\NewDocumentCommand{\residlof}{  
  D||{\sigma}   
  d<>           
  m             
  s             
  O{RoyalBlue}  
}{\sheaf R_{\IfNoValueTF{#2}{}{#2,} #1}\IfBooleanT{#4}{^{{\color{#5}\infty}}}\paren{#3}}

\NewDocumentCommand{\Adjidlof}{
  D||{\sigma}  
  O{X}         
  D<>{S}       
  m            
}{\operatorname{\mathit{Adj}}^{#1}_{\paren{#2,#3}}\paren{#4}}

\def\ELAdjidlof{\Adjidlof|\text{EL}|} 

\def\HMAdjidlof{\Adjidlof|\text{HM}|} 

\def\GAdjidlof{\Adjidlof|\text{G}|} 

\NewDocumentCommand{\aidlof}{
  D||{\sigma}  
  d<>          
  m            
  D//{}        
  O{\psi}      
  s            
}{\sheaf{J}_{\!\IfNoValueTF{#2}{}{#2,} #1}\IfBooleanF{#6}{\paren{#3; #4#5}}}

\NewDocumentCommand{\lcV}{ 
  D||{\sigma}    
  D//{\vphi_L}   
  d??            
  O{\psi}        
}{\:d\operatorname{lcv}^{#1}_{\IfNoValueF{#3}{#3,}#2}\left[#4\right]}

\newcommand{\dvol}{\:d\vol}

\newcommand{\RTFsym}{\mathfrak{F}} 
\NewDocumentCommand{\RTF}{ 
  s    
  G{\RTFsym} 
  d//  
  o    
  >{\SplitArgument{1}{,}} d<> 
  d||  
  d()  
  o    
}{%
  \begingroup%
    \newif\ifboolup%
    \booluptrue%
    \IfNoValueT{#4}{\IfNoValueT{#5}{\IfNoValueT{#6}{\boolupfalse}}}%
    \IfNoValueT{#7}{\boolupfalse}%
    \newcommand{\srptstr}{\cramped{{}^{\IfNoValueF{#4}{#4}\IfNoValueF{#5}{\inner#5}\IfNoValueF{#6}{\abs{#6}^2}}%
      \ifboolup _
      \fi{\ifboolup\displaystyle\fi\IfNoValueF{#7}{\paren{#7}}\IfNoValueF{#8}{%
          \ifboolup {\scriptstyle #8} \else _{#8} \fi%
        }}}}%
    \ifboolup%
      \IfBooleanTF{#1}{
        \smash[t]{
          \IfNoValueF{#3}{{}^{#3}}#2\raisebox{\depthof{$\srptstr$} * \real{0.3}}{$\srptstr$}%
        }%
      }{\IfNoValueF{#3}{{}^{#3}}#2\raisebox{\depthof{$\srptstr$} * \real{0.3}}{$\srptstr$}}%
    \else%
      \IfNoValueF{#3}{{}^{#3}}#2\srptstr%
    \fi%
  \endgroup%
} 

\def\RTI{\RTF{\mathfrak{I}}}

\NewDocumentCommand{\mtlog}{O{e} d() D||{\psi}}{\log\!#1^{\paren{#2}}\abs{#3}}
\NewDocumentCommand{\slog}{O{e} D||{\psi}}{\log\abs{#1 #2}}
\NewDocumentCommand{\dlog}{O{e} D||{\psi}}{\mtlog[#1](2)|#2|}

\NewDocumentCommand{\logpole}{ 
  D||{\psi}     
  D//{\sigma}   
  O{e}          
  D<>{1+\eps}   
}{\abs{#1}^{#2} \paren{\slog[#3]|#1|}^{#4}}

\DeclareFontFamily{OMX}{MnSymbolE}{}
\DeclareSymbolFont{MnLargeSymbols}{OMX}{MnSymbolE}{m}{n}
\SetSymbolFont{MnLargeSymbols}{bold}{OMX}{MnSymbolE}{b}{n}
\DeclareFontShape{OMX}{MnSymbolE}{m}{n}{
    <-6>  MnSymbolE5
   <6-7>  MnSymbolE6
   <7-8>  MnSymbolE7
   <8-9>  MnSymbolE8
   <9-10> MnSymbolE9
  <10-12> MnSymbolE10
  <12->   MnSymbolE12
}{}
\DeclareFontShape{OMX}{MnSymbolE}{b}{n}{
    <-6>  MnSymbolE-Bold5
   <6-7>  MnSymbolE-Bold6
   <7-8>  MnSymbolE-Bold7
   <8-9>  MnSymbolE-Bold8
   <9-10> MnSymbolE-Bold9
  <10-12> MnSymbolE-Bold10
  <12->   MnSymbolE-Bold12
}{}
\DeclareMathDelimiter{\llangle}{\mathopen}%
{MnLargeSymbols}{'164}{MnLargeSymbols}{'164}
\DeclareMathDelimiter{\rrangle}{\mathclose}%
{MnLargeSymbols}{'171}{MnLargeSymbols}{'171}

\newcommand{\eqcls}[1]{\left[#1\right]}

\NewDocumentCommand{\idxup}{ 
  m          
  O{\omega}  
}{\paren{#1}^{\mathrlap{\!#2}\makebox[\maxof{\widthof{$#2$}-\widthof{$\!\omega$}}{0pt}]{}}}

\NewDocumentCommand{\dep}{t{;} d<> O{\nu} m}{#4\IfBooleanTF{#1}{_}{^}{\IfNoValueF{#2}{#2\:}(#3)}}

\NewDocumentCommand{\sm}{s m}{{#2}\IfBooleanTF{#1}{_}{^}\text{sm}}

\newcommand{\tlog}{{\text{log}}}

\newcommand{\charfct}{\mathbbm 1} 

\NewDocumentCommand{\lelong}{m O{x}}{\operatorname{\boldsymbol{\nu}}\paren{#1,#2}}

\newcommand{\currInt}[1]{\left[#1\right]} 

\NewDocumentCommand{\rs}{ 
  s  
  m  
}{\IfBooleanTF{#1}{\smash[t]{\widetilde{#2}}}{\widetilde{#2}}}

\DeclareMathOperator{\Ann}{Ann}  
\DeclareMathOperator{\mlc}{mlc} 
\newcommand{\Diff}{\operatorname{Diff}^*} 
\newcommand{\coef}{\operatorname{coef}}

\newcommand{\sect}[1][s]{\mathtt{#1}} 

\newcommand{\bphi}{\boldsymbol{\vphi}}
\newcommand{\bphip}[1][p]{\res\bphi_{#1}} 

\newcommand{\divP}{\sheaf P}  
\NewDocumentCommand{\cbn}{  
  D//{\sigma_V}
  O{\sigma}
}{\mathfrak{C}^{#1}_{#2}} 
\NewDocumentCommand{\Iset}{  
  t{.}             
  D//{V}           
  O{\sigma}        
}{\IfBooleanTF{#1}{I}{\rs I}^{#2}_{#3}} 

\newtheorem{prop}{Proposition}[subsection]
\newtheorem{lemma}[prop]{Lemma}
\newtheorem{thm}[prop]{Theorem}
\newtheorem{cor}[prop]{Corollary}
\newtheorem{SNCassumption}[prop]{Snc assumption}
\newtheorem{SNCassumptionx}{Snc assumption}

\newtheorem{definition-thm}[prop]{Definition-Theorem}

\theoremstyle{remark}
\newtheorem{remark}[prop]{Remark}

\theoremstyle{definition}
\newtheorem{definition}[prop]{Definition}
\newtheorem{example}[prop]{Example}
\newtheorem{notation}[prop]{Notation}

\numberwithin{equation}{subsection}

\allowdisplaybreaks  

\begin{document}

\newcommand{\titlestr}{%
  A new definition of analytic adjoint ideal sheaves via
  the residue functions of log-canonical measures I%
}

\newcommand{\shorttitlestr}{%
  A new definition of analytic adjoint ideal sheaves I%
}

\newcommand{\MCname}{Tsz On Mario Chan}
\newcommand{\MCnameshort}{Mario Chan}
\newcommand{\MCemail}{mariochan@pusan.ac.kr}

\newcommand{\addressstr}{%
  Dept.~of Mathematics, Pusan National
  University, Busan 46241, South Korea%
}

\newcommand{\subjclassstr}[1][,]{%
  32J25 (primary)#1  
  14B05 (secondary)
}

\newcommand{\keywordstr}[1][,]{%
  adjoint ideal sheaf#1
  multiplier ideal sheaf#1
  lc centre#1
  $L^2$ extension%
}

\newcommand{\dedicatorystr}{%
  In the memory of Prof.~Jean-Pierre Demailly%
}

\newcommand{\thankstr}{%
  The author would like to thank Young-Jun Choi for his support and
  encouragement on publishing this work.
  His thanks also go to Chen-Yu Chi for drawing his attention to this
  topic during his stay at the NCTS, and to an anonymous referee for
  correcting a reference in an early version of this paper.
  The theoretical basis of this work started to take shape when the
  author visited Jean-Pierre Demailly in \textfr{Institut Fourier},
  whom the author feels so indebted to and will never be able to repay.
  This work was partly supported by the National Research Foundation (NRF) of
  Korea grant funded by the Korea government (No.~2018R1C1B3005963 and
  No.~2023R1A2C1007227)%
}


\title[\shorttitlestr]{\titlestr}
 
\author[\MCnameshort]{\MCname}
\email{\MCemail}
\address{\addressstr}


\thanks{\thankstr}
 
\subjclass[2020]{\subjclassstr}

\keywords{\keywordstr}

\dedicatory{\dedicatorystr}

\begin{abstract}

A new definition of analytic adjoint ideal sheaves for
quasi-plurisubharmonic (quasi-psh) functions with only neat analytic
singularities is studied and shown to admit some residue short exact
sequences which are obtained by restricting sections of the newly
defined adjoint ideal sheaves to some unions of $\sigma$-log-canonical
($\sigma$-lc) centres.
The newly defined adjoint ideal sheaves induce naturally some residue $L^2$ norms on the
unions of $\sigma$-lc centres which are invariant under log-resolutions.
They can also describe unions of $\sigma$-lc centres without the need of
log-resolutions even if the quasi-psh functions in question are not in
a simple-normal-crossing configuration.
This is hinting their potential use in discussing the $\sigma$-lc
centres even when the quasi-psh functions in question have more
general singularities.
Furthermore, their relations between the algebraic adjoint ideal
sheaves of Ein--Lazarsfeld as well as those of Hacon--\McKernan\ are
described in order to illustrate their role as a (potentially finer)
measurement of singularities in the minimal model program.
In the course of the study, a local $L^2$ extension theorem is proven,
which shows that holomorphic sections on any unions of $\sigma$-lc
centres can be extended holomorphically to some neighbourhood of the
unions of $\sigma$-lc centres with some $L^2$ estimates.
The proof does not rely on the techniques in the Ohsawa--Takegoshi-type
$L^2$ extension theorems.


\end{abstract} 

\date{\today} 

\maketitle



\section{Introduction}
\label{sec:introduction}


\subsection{A brief account on the background of adjoint ideal
  sheaves}
\label{sec:brief-account-on-background}

The algebraic version of the adjoint ideal sheaf $\Adjidlof||<S>{D}$
arises when one tries to restrict the multiplier ideal sheaf
$\mtidlof[X]{D}$ of some $\fieldQ$-divisor $D$ on an ambient
projective manifold $X$ (or, more generally, normal variety) to a
(possibly singular) hypersurface $S$,
which does not lie in the support of $D$, and compares it with the
multiplier ideal sheaf $\mtidlof[S]{\res D_S}$ of the restriction of
that divisor $\res D_S$ on $S$ (see \cite{Lazarsfeld_book-II}*{\S
  9.3.E} and \cite{Takayama_adj-ideal}*{Prop.~2.4}).
According to Lazarsfeld (\cite{Lazarsfeld_book-II}*{\S 9.3.E}), it
first appeared in the work of \textfr{Vaquié} (\cite{Vaquie_adjIdl})
and was later rediscovered in the work of Ein and Lazarsfeld
(\cite{Ein&Lazarsfeld_adjIdl}).
That it fits into the short exact sequence induced by the restriction
map, namely,
\begin{equation*}
  0 \to \mtidlof[X]{D} \otimes \holo_{X}\paren{-S}
  \to \Adjidlof||<S>{D} \to \iota_*\nu_*\mtidlof[\rs S]{\nu^*\iota^*D}
  \to 0 \; ,  
\end{equation*}
where $\nu \colon \rs S \to S$ is any resolution of singularities and
$\iota \colon S \hookrightarrow X$ is the natural inclusion,
is originally a defining property of such kind of ideal sheaves.
In \cite{Ein&Lazarsfeld_adjIdl}, it is used as a measurement of the
singularities on $S$ (when $D=0$). 
Its another use can be found in the study of pluricanonical systems in
the works of Takayama (\cite{Takayama_adj-ideal}) and
Hacon and M${}^{\text{c}}$Kernan (\cite{Hacon&Mckernan}\footnote{
  The adjoint ideal sheaf in this paper is indeed named as
  ``multiplier ideal sheaf'' (see \cite{Hacon&Mckernan}*{Def.-Lemma
    4.2}).
  It is eventually called ``adjoint ideal sheaf'' in
  \cite{Ein-Popa}*{Def.~2.4}.
}).
In these works, its role is more or less an auxiliary object to
facilitate the use of the associated long exact sequence and Nadel
vanishing theorem to guarantee that global sections of
$\mtidlof[S]{\res D_S}$ (after tensored with some line bundle with
sufficient positivity) can be extended to some global sections of
$\mtidlof[X]{D}$.

The adjoint ideal sheaves introduced in both
\cite{Ein&Lazarsfeld_adjIdl} and \cite{Hacon&Mckernan} are defined
also for $S$ being a reduced divisor with simple normal crossings
(snc)\footnote{
  Indeed Ein and Lazarsfeld do not require $S$ to have only snc in
  \cite{Ein&Lazarsfeld_adjIdl}*{Prop.~3.1}.
} such that $\supp D$ contains no lc centres of $(X,S)$, i.e.~no
intersection of any number of irreducible components of $S$ is
contained in $\supp D$.
While the two are the same when $S$ is a disjoint union of smooth
hypersurfaces (and $X$ is smooth), \mmark{they are different for more
  general $S$ (see \eqref{eq:definition-alg-ELAdjidl} and
  \eqref{eq:definition-alg-HMAdjidl} for their definitions).}{Add reference
  to the definitions and equality in the smooth
  case. \alert{Done without proof of the equality.}} 
Indeed, it follows from their respective definitions that the
Ein--Lazarsfeld adjoint ideal sheaf, $\ELAdjidlof{D}$, is trivial if
and only if $(X, S+D)$ is purely log-terminal (plt) with $\floor D =0$
(cf.~\cite{Takagi_alg-adjoint-ideal}*{Remark 1.3 (2)}), while the
Hacon--M${}^{\text{c}}$Kernan adjoint ideal sheaf, $\HMAdjidlof{D}$,
is trivial if and only if $(X, S+D)$ is divisorially log-terminal
(dlt) with $\floor D = 0$ (see \cite{Hacon&Mckernan}*{Lemma 4.3 (1)}).
(Readers are referred to \cite{Kollar_Sing-of-MMP}*{Def.~2.8} for
various notions of singularities in the minimal model program.
Note that the definition of ``plt'' here follows that in
\cite{Kollar_Sing-of-MMP}*{Def.~2.8} (in which coefficients of $D$ are
assumed to be $\leq 1$ and only discrepancies of \emph{exceptional
  divisors} have to be $>-1$) instead of
\cite{Takagi_alg-adjoint-ideal}*{Def.~1.1 (ii)} (in which
discrepancies of \emph{all} divisors not dominating components of $S$
have to be $>-1$).
In this sense, $(X,S+D)$ being plt alone is not sufficient to guarantee
$\ELAdjidlof{D} =\holo_X$.)

Further research and applications on these algebraic versions of
adjoint ideal sheaves include works of Takagi
(\cite{Takagi_alg-adjoint-ideal}) and Eisenstein
(\cite{Eisenstein_thesis}), in which the generalisation of
$\ELAdjidlof{D}$ with the reduced divisor $S$ replaced by a reduced
closed subscheme of pure codimension (possibly $> 1$) is studied.
Following the spirit of \cite{Takayama_adj-ideal} and
\cite{Hacon&Mckernan}, Ein and Popa (\cite{Ein-Popa}) consider
$\HMAdjidlof<S>{\ideal a^c}$, where $S$ is a reduced snc divisor and
$\ideal a^c$ is the formally exponentiated ideal sheaf of a coherent ideal sheaf
$\ideal a$ by a real number $c \geq 0$,
and, among other things, improve the vanishing and extension statements in
\cite{Hacon&Mckernan}*{Lemma 4.4 (2) and (3)} via the use of the corresponding
short exact sequence of $\HMAdjidlof<S>{\ideal a^c}$ induced from a
restriction map (in a slightly different form from the one shown
above) and an induction on the number of components of $S$ as 
well as the dimension (see \cite{Ein-Popa}*{Thm.~2.9 and Prop.~4.1}).

For the analytic version, Guenancia proposes an analytic definition of
the adjoint ideal sheaves in \cite{Guenancia}, which is later
generalised by Dano Kim (\cite{KimDano-adjIdl}).
In the case where the relevant subvariety $S$ is a \emph{smooth
hypersurface} whose associated line bundle has a canonical section
$\sect$, for any plurisubharmonic (psh) function $\bphi$ on a complex
manifold $X$ such that $S \not\subset \bphi^{-1}(-\infty)$, their
adjoint ideal sheaf can be given as
\begin{equation} \label{eq:Guenancia-adj-ideal}
  \Adjidlof<S>{\bphi} =\bigcup_{\lambda > 1} \mtidlof[X]{\lambda \bphi
  +\phi_S +\log\abs{\psi_S}^{\sigma +1}} \; ,
\end{equation}
where $\phi_S :=\log\abs{\sect}^2$ and $\psi_S :=\phi_S -\sm\vphi_S$
for some smooth potential $\sm\vphi_S$ of $S$ (see Notations
\ref{notation:potential-def} and
\ref{notation:potentials}) and $\mtidlof<X>{\vphi}$ denotes the multiplier
ideal sheaf of a potential $\vphi$ on $X$ (see Section
\ref{sec:setup}).\footnote{ \label{fn:Guenancia-Adjidl} 
  The index $\sigma$ in the current notation $\Adjidlof{\bphi}$ is off
  by $1$ from the one given in \cite{KimDano-adjIdl} so that it is
  easier to compare with the new version of analytic adjoint ideal
  sheaves introduced in next section.
  For $S$ being a smooth hypersurface, Guenancia (\cite{Guenancia})
  gives the definition with $\sigma =1$ while Dano Kim
  (\cite{KimDano-adjIdl}) suggests to allow any $\sigma > 0$.
  For $S$ being a reduced snc divisor, the adjoint ideal sheaf considered 
  by Guenancia in \cite{Guenancia}, denoted by
  $\GAdjidlof{\bphi}$, is $\Adjidlof|1|{\bphi}$ but having the
  function $\psi_S$ replaced by $\Psi_G$, which is locally of the form
  $\Psi_G =\prod_{j=1}^{\sigma} \log\abs{z_j}^2$ (while $\psi_S$
  is locally of the form $\sum_{j=1}^\sigma \log\abs{z_j}^2
  +\BigO(1)$).
  It is easy to check that $c\abs{\psi_S}^2 \lesssim_\tlog \abs{\Psi_G}^2
  \lesssim_\tlog C\abs{\psi_S}^{\sigma+1}$ (see Notation
  \ref{notation:ineq-up-to-constants}) on a small polydisc centred at
  the origin for some constants $c, C > 0$, so $\Adjidlof|1|{\bphi}
  \subset \GAdjidlof{\bphi} \subset \Adjidlof{\bphi}$.
}\mmark{}{Add reference to computation of $\GAdjidlof{\bphi}$ in the
  footnote. \alert{Reference to the counter example added to next footnote.}}
It is proved that when $\bphi$ is obtained from some algebraic data
(like the $\fieldQ$-divisor $D$ or the formally exponentiated ideal
sheaf $\ideal a^c$ above), which has only neat analytic 
singularities (see Section \ref{sec:setup} item
\eqref{item:neat-analytic-singularities} for the definition) such that
the polar set does not contain $S$, their analytic version of adjoint
ideal sheaves (with $\sigma \in (0,1]$) coincides with the algebraic
version.
(Note that the Ein--Lazarsfeld and Hacon--M${}^{\text{c}}$Kernan
versions of adjoint ideal sheaves coincide when $X$ is smooth and $S$
is a disjoint union of smooth hypersurfaces.)\footnote{
  Guenancia also claims in \cite{Guenancia}*{Prop.~2.11} that his
  analytic adjoint ideal sheaf coincides with the Ein--Lazarsfeld
  algebraic adjoint ideal sheaf even when $S$ is a reduced snc divisor.
  See Example \ref{example:Guenancia-neq-EL} for a counter example to
  this claim. See also the Erratum \cite{Guenancia_AdjIdl-Erratum} by
  Guenancia.
}\mmark{}{Better reference to the counter example to Guenancia's claim
  in the footnote. \alert{Done.}}
Moreover, when $e^{\bphi}$ is locally \textde{Hölder} continuous (thus
including the case when $\bphi$ has only neat analytic singularities)
and when the solution to the strong openness conjecture on multiplier
ideal sheaves (see \cite{Guan&Zhou_openness} or \cite{Hiep_openness})
is taken into account, their adjoint ideal sheaves also fit into the
corresponding short exact sequence induced from restriction as in the
algebraic case (with $\mtidlof[X]{D}$ replaced by $\mtidlof[X]{\bphi}$
and $\mtidlof[\rs S]{\nu^*\iota^*D} =\mtidlof[S]{\res D_S}$ by
$\mtidlof[S]{\res\bphi_S}$), which in turn implies the coherence of
the analytic adjoint ideal sheaves.
While the well-definedness of the residue morphism
$\Adjidlof<S>{\bphi} \to \mtidlof[S]{\res\bphi_S}$ is guaranteed by
the assumption that $e^{\bphi}$ being locally \textde{Hölder}
continuous, the surjectivity of the residue morphism relies on the
$L^2$ estimates from the Ohsawa--Takegoshi--Manivel $L^2$ extension
theorem (see \cite{Manivel} or \cite{Demailly_on_OTM-extension}) or
its variants.
Even at the moment of writing, no Ohsawa--Takegoshi-type $L^2$ extension
theorem is applicable for proving the surjectivity of the
residue morphism when the subvariety $S$ is an snc divisor with
intersecting components (due to the non-integrable singularities on
the Ohsawa measure).
That is why $S$ is mostly assumed to be a smooth hypersurface in the
expositions in \cite{Guenancia} and \cite{KimDano-adjIdl}.

More recent research on the analytic adjoint ideal sheaves that the
author is aware of includes the works of Qi'an Guan and Zhenqian Li
(\cite{Guan&Li_adjIdl-not-coherent}, which shows that existence of
$\bphi$ such that the corresponding Guenancia's version of adjoint
ideal sheaf is not coherent), Zhenqian Li (\cite{Li_adj-idl-II}, which
generalises Guenancia's adjoint ideal sheaf, in a way different from
the current article, so that $S$ can be locally a complete
intersection, and obtains a short exact sequence induced from
restriction via the Ohsawa--Takegoshi--Manivel $L^2$ extension
theorem) and Dano Kim and Hoseob Seo (\cite{KimDano&Seo_adj-idl},
which considers Guenancia's adjoint ideal sheaf with $S$ being a
singular hypersurface).

\subsection{The main results of this article}
\label{sec:main-results}

Let $X$ be a complex manifold, $L$ a holomorphic line bundle equipped
with a singular metric $e^{-\vphi_L}$, and $\psi \leq -1$ a global function
on $X$ such that $\vphi_L+m\psi$ are locally differences of quasi-psh
functions with neat analytic singularities for all $m \in \fieldR$.
Let $m=m_1$ be the only jumping number of the family
$\set{\mtidlof[X]{\vphi_L+m\psi}}_{m\in[m_0,m_1]}$ on the interval
$(m_0,m_1]$, where $0 \leq m_0 < m_1$ (be warned that $m_1$ can be any
jumping number of the family in general and \emph{need not be} the
smallest one).
After a suitable normalisation, assume that $m_1=1$ and, for most
cases, $m_0 =0$ (see Section \ref{sec:setup} item
\eqref{item:jumping-no} and footnote
\ref{fn:renormalising-jumping-no}).\footnote{
  Note that the existence of jumping numbers is guaranteed when $X$ is
  compact or is a relatively compact domain in some ambient manifold. 
  Note also that the set of jumping numbers has no accumulation points
  when $\vphi_L$ and $\psi$ have only neat analytic singularities.}
Assume also that $\vphi_L+m_1\psi =\vphi_L+\psi$ is quasi-psh on $X$.
The subvariety
$S$ playing the role in previous section 
is now given as the \emph{scheme-theoretic difference between the
subvarieties associated to the multiplier ideal sheaves
$\mtidlof[X]{\vphi_L}$ and $\mtidlof[X]{\vphi_L+\psi}$} (see Section
\ref{sec:setup} item \eqref{item:def-S}).
The subvariety $S$ such defined is reduced (see
\cite{Demailly_extension}*{Lemma 4.2}).
In \cite{Chan&Choi_ext-with-lcv-codim-1} and
\cite{Chan_on-L2-ext-with-lc-measures}, the author introduces a
sequence of lc-measures $\seq{\!\lcV}_{\sigma \in\Nnum}$ and their
corresponding residue functions $\RTF|f|(\eps)[\sigma]$ (see Section
\ref{sec:residue-fcts-with-nas}) which provide, at least in the case
when the reduced subvariety $S$ is an snc divisor, a natural way to define the
residue morphism when restriction from $X$ to $\lcc$, the union of lc centres
of $(X,S)$ of codimension $\sigma$ (or \emph{$\sigma$-lc centres} for
short; see Definition \ref{def:sigma-lc-centres} and Remark
\ref{rem:sigma-lcc-explain} for the usage of the terminology when $S$
is not an snc divisor), is considered for each natural number $\sigma
\geq 1$ (see Theorem \ref{thm:results-on-residue-fct}).
This gives a good reason to consider the following new version of
analytic adjoint ideal sheaves.
\begin{definition} \label{def:adjoint-ideal-sheaves}
  Given any integer $\sigma \geq 0$ and the family
  $\set{\mtidlof[X]{\vphi_L+m\psi}}_{m\in[0,1]}$ with a jumping number
  at $m=1$, the \emph{(analytic) adjoint ideal sheaf $\aidlof{\vphi_L}
    := \aidlof<X>{\vphi_L}$ of index $\sigma$} of $(X,\vphi_L,\psi)$
  is the sheaf associated to the presheaf over $X$ given by 
  \begin{equation*}
    \bigcap_{\eps > 0} \mtidlof{\vphi_L+\psi +\log\paren{\logpole}}\paren{V}
  \end{equation*}
  for every open set $V \subset X$.
  Then, its stalk at each $x \in X$ can be described as
  \begin{equation*}
    \aidlof{\vphi_L}_x
    =\setd{f\in \holo_{X,x}}{\exists~\text{open set } V_x \ni x \: , \;
      \forall~\eps > 0 \: , \; \frac{\abs f^2
        e^{-\vphi_L-\psi}}{\logpole} \in L^1\paren{V_x} } \; . \:\footnotemark
  \end{equation*}%
  \footnotetext{The notation $\aidlof{\cdot}[\cdot]$ is chosen so as
    to be in line with the current choice of the typeface of the
    notation of multiplier ideal sheaves $\mtidlof{\cdot}$ and to
    signify that it denotes an ad\underline{J}oint ideal sheaf.
    The author also prefers to use a different notation from the
    algebraic and Guenancia's versions of adjoint ideal sheaves because
    of the subtle difference indicated in Theorem
    \ref{thm:comparison-alg-and-analytic-aidl}.}%
\end{definition}

In studies involving restriction of a system (for example, $(X,\bphi)$
for some psh potential $\bphi$, $(X,\ideal a^c)$ for some ideal $\ideal
a$, or $(X,D)$ for some $\fieldQ$-divisor $D$) to some lc centres, it
is usual to assume that the support of the system
(i.e.~$\bphi^{-1}(-\infty)$, the zero locus of $\ideal a$ or $\supp
D$ respectively in regard to the examples above) does not contain any
of the lc centres (see, for example, \cite{Fujino_log-MMP}*{Thm.~6.1}
or \cite{Ein-Popa}*{\S 2.2}).
In the current setup, thanks to the
assumption on the jumping number, 
after passing to some log-resolution of $(X,\vphi_L,\psi)$ (possible
under the assumption that $\vphi_L$ and $\psi$ having only neat
analytic singularities),
there is an effective ($\Znum$-)divisor $S_0$
with $\supp S_0 \subset S$ such that
\begin{equation*}
  \vphi_L+\psi = \bphi +\phi_{S_0} +\phi_S \; ,
\end{equation*}
where $\phi_{S_0} := \log\abs{\sect_0}^2$ ($\phi_S
:=\log\abs{\sect}^2$), a potential defined from a canonical section
$\sect_0$ of $S_0$ ($\sect$ of $S$), and $\bphi$ is a quasi-psh
potential which is locally integrable around general points on $S$
(see \eqref{eq:define-bphi} in Section \ref{sec:snc-assumption}).
It follows automatically from the snc assumption on $\vphi_L$ (hence
$\bphi$; it essentially means that the polar set of $\vphi_L$ as well
as $\bphi$ is a divisor having snc with $S$; see Snc assumption
\ref{assum:snc-nas}) that $\bphi^{-1}(-\infty)$ contains no lc centres
of $(X,S)$.
The system $(X,\vphi_L,\psi)$ indeed describes a slightly more general
setting than in other studies.

One can compare the definitions of $\aidlof{\vphi_L}[\psi_S]$ and
$\Adjidlof{\vphi_L}$, which are apparently similar, when $S$ is a
disjoint union of smooth hypersurfaces.
The union over the parameter $\lambda$ in the definition of
$\Adjidlof{\vphi_L}$ in \eqref{eq:Guenancia-adj-ideal} is needed to
compensate for the lack of the openness property for adjoint ideal
sheaves, in the sense that, for any holomorphic function $f$,
$\frac{\abs f^2 \:e^{-\vphi_L-\phi_S}}{\abs{\psi_S}^{\sigma+1}} \in
\Lloc$ \emph{does not} imply that there exists some constant $\lambda
> 1$ such that $\frac{\abs f^2
  \:e^{-\lambda\vphi_L-\phi_S}}{\abs{\psi_S}^{\sigma+1}} \in \Lloc$ for
any $\sigma > 0$ (especially when $(X,\vphi_L,\psi_S)$ is not yet in
an snc configuration; see \cite{Guenancia}*{Counterexample on
  pp.1023}; also cf.~Example \ref{example:compare-alg-ana-adj}).
With the introduction of the parameter $\lambda$, Guenancia shows
that his version of the adjoint ideal sheaves coincides with the
algebraic version when $S$ is smooth \mmark{(see
  \cite{Guenancia}*{Prop.~2.11} and its erratum
  \cite{Guenancia_AdjIdl-Erratum}; see also Example 
  \ref{example:Guenancia-neq-EL} and Remark
  \ref{rem:flaw-in-Guenancia})}{Add reference to the counter example
  \alert{Done.}}.

Concerning the openness property under the Snc assumption
\ref{assum:snc-nas}, since $\bphi^{-1}(-\infty)$ contains no lc
centres of $(X,S)$, it turns out that $\aidlof{\vphi_L}$ (one may
put $\vphi_L=\bphi$ and $\psi=\psi_S$ for a fair comparison) retains the
openness property of $\Adjidlof{\bphi}$ (see Section
\ref{sec:openness}) even without the direct limit over $\lambda >1$ in
the definition (although it is different from the openness property of
$\Adjidlof{\vphi_L}$ if $\vphi_L^{-1}(-\infty)$ contains some
components of $S$).

Assume that $S$ given above is a reduced divisor (which need not have
snc and may have more than one components).
Without the parameter $\lambda$ in the definition, the new adjoint
ideal sheaves introduced above coincide with the Ein--Lazarsfeld or
Hacon--\McKernan\ algebraic adjoint ideal sheaves according to the
their indices $\sigma$ in ``almost all'' cases in the following sense.
\begin{thm}[see Theorem \ref{thm:comparison-alg-and-analytic-aidl}]
  \label{thm:intro-comparison-alg-and-analytic-aidl}
  Given an lc pair $(X,S)$ (in which $S$ is a reduced divisor but need
  not have snc), 
  let $\vphi_{\ideal a^c}$ be a quasi-psh potential induced from
  $\ideal a^c$, where $\ideal a$ is a coherent ideal sheaf on $X$
  whose zero locus contains no lc centres of $(X,S)$ in the sense of
  \cite{Kollar_Sing-of-MMP}*{Def.~4.15}, and $c \geq 0$ is a real
  number. 
  Set also $\psi := \psi_S$ (see Notation \ref{notation:potentials}).
  Then, one has, for any $c \geq 0$,
  \begin{equation*}
    \ELAdjidlof{\ideal a^c} \subset
    \aidlof|1|<X>{\vphi_{\ideal a^c}}
    \quad\text{ and }\quad
    \HMAdjidlof{\ideal a^c} \;
    \subset
    \aidlof|\sigma_{\mlc}|<X>{\vphi_{\ideal a^c}} \; ,
  \end{equation*}
  and there exists a countable discrete set $N := N(\ideal a, S)
  \subset \fieldR_{> 0}$ (excluding $0$) such that equalities hold for
  both inclusions for any $c \in \fieldR_{\geq 0} \setminus N$.
  Note that the integer $\sigma_{\mlc}$ depends on $c$ and is the
  smallest integer such that $\aidlof|\sigma_{\mlc}|<X>{\vphi_{\ideal
      a^c}} =\mtidlof<X>{\vphi_{\ideal a^c} +m^c \psi}$, where $m^c \in
  [0,1)$ such that $\mtidlof{\vphi_{\ideal a^c} +m\psi}$ is unchanged
  as $m$ varies within $[m^c,1)$.

  More specifically, if, for a given $c \geq 0$, the zero locus of
  $\ideal a$ contains no lc centres of $(X,\vphi_{\ideal a^c},\psi)$
  defined in Definition-Theorem \ref{thm:main-results-log-resoln},
  then $c \not\in N$ and therefore, in such case,
  \begin{equation*}
    \ELAdjidlof{\ideal a^c}
    =\aidlof|1|<X>{\vphi_{\ideal a^c}}
    \quad\text{ and }\quad
    \HMAdjidlof{\ideal a^c} \;
    =\aidlof|\sigma_{\mlc}|<X>{\vphi_{\ideal a^c}} \; .
  \end{equation*}

\end{thm}

See Example \ref{example:adjidl-neq-in-non-lc-case} for an example of
the different versions of adjoint ideal sheaves when $(X,S)$ is not an
lc pair.

Despite being slightly different from the algebraic version of the
adjoint ideal sheaves, the ease of analysis and nice properties of
$\aidlof{\vphi_L}$ strongly suggest that the advantage has outweighed
the drawback of adopting the new definition (see also Remark
\ref{rem:why-no-lambda}).
Among those nice properties, the sheaves $\aidlof{\vphi_L}$ for
various $\sigma$ give a finer structure via a filtration between
$\mtidlof{\vphi_L} =\mtidlof{\vphi_L +m_0 \psi}$ and
$\mtidlof{\vphi_L+\psi}$ and fit into the expected residue short exact
sequence, which are stated as follows (the statements are stated with
the normalisation $m_1 =1$ but for general $m_0 \geq 0$ for the ease
of applications in practice).
\begin{thm}[Theorem \ref{thm:results-on-residue-fct}, Remark
  \ref{rem:residue-fct-result-jumping-no}, Theorem
  \ref{thm:short-exact-seq} and Corollary
  \ref{cor:local-L2-estimates}]
  \label{thm:main-results}
  Given the notation above (or as in Section \ref{sec:setup}) and under
  the Snc assumption \ref{assum:snc-nas} (thus $\vphi_L$ and $\psi$
  having only neat analytic singularities with snc in particular), it
  follows that
  \begin{enumerate}[labelindent=0em, leftmargin=*]
  \item one has
    \begin{equation*}
      \aidlof{\vphi_L} = \mtidlof[X]{\vphi_L +m_0 \psi} \cdot
      \defidlof{\lcc[\sigma+1]}
    \end{equation*}
    for all integers $\sigma \geq 0$, where
    $\defidlof{\lcc[\sigma+1]}$ is the defining ideal sheaf of
    the union of $(\sigma+1)$-lc centres $\lcc[\sigma+1]$ (with the
    reduced structure), and, therefore, obtains a filtration
    \begin{equation*}
      \mtidlof{\vphi_L +\psi} =\aidlof|0|{\vphi_L}
      \subset \aidlof|1|{\vphi_L} \subset \dotsm
      \subset \aidlof|\sigma_{\mlc}|{\vphi_L}
      =\mtidlof{\vphi_L +m_0 \psi} \; ,
    \end{equation*}
    where $\sigma_{\mlc}$ is the codimension of a minimal lc centre
    (mlc) of $(X,S)$;

  \item \label{item:thm:residue-exact-seq}
    the ideal sheaves $\aidlof{\vphi_L}$ fit into the residue
    short exact sequence
    \begin{equation*}
      \xymatrix{
        0 \ar[r]
        & {\aidlof|\sigma-1|{\vphi_L}} \ar[r]
        & {\aidlof{\vphi_L}} \ar[r]^-{\Res}
        & {\residlof{\bphi}} \ar[r]
        & 0
      } \; ,
    \end{equation*}
    where $\residlof{\bphi}$ is a coherent sheaf supported on $\lcc$
    such that, on an open set $V$ with $\lcc \cap V = \bigcup_{p\in\cbn}
    \lcS$, where $\lcS$'s are the $\sigma$-lc centres in $V$ indexed
    by $p \in \cbn$, one has
    \begin{equation*}
      K_X \otimes \residlof{\bphi}(V)
      =\prod_{p\in\cbn} K_{\lcS} \otimes \parres{S_0^{-1} \otimes
        S^{-1}}_{\lcS} 
      \otimes \mtidlof[\lcS]{\bphi}\paren{\lcS} \; ;
    \end{equation*}

  \item \label{item:thm:local-L2-estimate}
    for any $\paren{g_p}_{p\in\cbn} \in 
    \residlof{\bphi}\paren{\cl V}$ on a sufficiently small open set
    $V$ with $\lcc \cap V =\bigcup_{p\in\cbn} \lcS$ as above, there
    exists a section $f \in \aidlof{\vphi_L}\paren{\cl V}$ such that
    $\Res(f) =\paren{g_p}_{p\in\cbn}$, and, for any $\eps > 0$, 
    \begin{equation*}
      \eps \int_V \frac{\abs f^2 \:e^{-\vphi_L-\psi} \dvol_V}{\logpole}
      \leq C\sum_{p\in\cbn} \frac{\pi^\sigma}{(\sigma -1)! \:\vect\nu_p}
      \int_{\lcS} \abs{g_p}^2 \:e^{-\bphi} \dvol_{\lcS}
      \; ,
    \end{equation*}
    where $\vect\nu_p$ is the product of the generic Lelong numbers
    $\lelong{\psi}[D_i]$ of $\psi$ along the irreducible components
    $D_i$ of $S$ containing $\lcS$, and $C > 0$ is a constant
    depending only ``mildly'' on $\bphi$.
  \end{enumerate} 
\end{thm}
Readers are referred to Theorem \ref{thm:results-on-residue-fct} (see
also Remark \ref{rem:residue-fct-result-jumping-no}), Theorem
\ref{thm:short-exact-seq} and Corollary \ref{cor:local-L2-estimates}
for the precise statements.
The residue morphism $\Res$ in \eqref{item:thm:residue-exact-seq}, if
tensored by $K_X$, is given by the \textfr{Poincaré} residue map
(see \eqref{eq:g_p-from-Poincare-residue} or
\cite{Kollar_Sing-of-MMP}*{\S 4.18}).
The surjectivity of $\Res$ in Theorem \ref{thm:main-results}
\eqref{item:thm:residue-exact-seq} (or Theorem
\ref{thm:short-exact-seq}) and the estimate in Theorem
\ref{thm:main-results} \eqref{item:thm:local-L2-estimate} (or
Corollary \ref{cor:local-L2-estimates}) can indeed be regarded as a
\emph{local $L^2$ extension theorem with estimates}, which does
\emph{not} rely on any Ohsawa--Takegoshi-type $L^2$ extension
theorems.
Note that both sides of the estimate in Theorem \ref{thm:main-results}
\eqref{item:thm:local-L2-estimate} can be expressed as the residue
functions (see Section \ref{sec:residue-fcts-with-nas}) so that the
estimate is read as 
\begin{equation*}
  \RTF|f|(\eps)[V,\sigma] \leq C \:\RTF|f|(0)[V,\sigma] \; ,
\end{equation*}
where $\RTF*|f|(0)[V,\sigma]$ is the residue norm of $f$ given respect
to the $\sigma$-lc measure $\lcV$ (see
\eqref{eq:lc-measure-for-non-cpt-supp} for a precise definition).
The residue norm $\RTF|\cdot|(0)[V,\sigma]$ indeed induces an $L^2$ norm
$\norm\cdot_{\lcc<V>}$ on $\residlof{\bphi}(\cl V)$ for any open set
$V \Subset X$.
The latter norm is also referred to as the \emph{residue norm on
  $\residlof{\bphi}(\cl V)$} for convenience.

The residue short exact sequence in Theorem
\ref{thm:main-results} \eqref{item:thm:residue-exact-seq} can
facilitate inductive arguments.
One example can be found in the study of injectivity theorem in
\cite{Chan&Choi_injectivity-I}.

Further results show that the adjoint ideal sheaves
$\aidlof{\vphi_L}$ and the residue norms $\RTF|\cdot|(0)[V,\sigma]$
(hence $\norm\cdot_{\lcc<V>}$) exhibit good invariant properties with
respect to log-resolutions such that it makes sense to define the
residue morphisms $\Res$ and the residue norms even when
$(X,\vphi_L,\psi)$ is not in an snc configuration, i.e.~not satisfying
the Snc assumption \ref{assum:snc-nas}.
To state the results, suppose $m=1$ is again a jumping number of the
family $\set{\mtidlof<X>{\vphi +m\psi}}_{m \in [0,1]}$ and $S$ is
given as before (which is named as the \emph{lc locus\footnote{The use of
    the term ``lc locus'' differs from the use in some literature. See
    item \eqref{item:def-S} in Section \ref{sec:setup} as well as
    footnote \ref{fn:lc-locus-original-meaning}.} of the family at
  jumping number $m=1$} for convenience). 
Let $\pi \colon \rs X \to X$ be any log-resolution of
$(X,\vphi_L,\psi)$, which gives rise to the effective divisors $E$,
$R$ (where $K_{\rs X / X} \sim E+R$) and the potentials
$\pi^*_\ominus\vphi_L$.
Then let $\rs S$, $\rs S_0$
and $\rs\bphi$ (the counterparts of $S$, $S_0$ and $\bphi$) be defined
from the family $\set{\mtidlof<\rs X>{\pi^*_\ominus\vphi_L
    +m\pi^*\psi}}_{m\in[0,1]}$ at jumping number $m=1$ (see Section
\ref{sec:snc-assumption}, Proposition
\ref{prop:log-resoln-on-jump-subvar} and the beginning of Section
\ref{sec:non-snc}).
The lc locus $\rs S$ such defined is a reduced snc divisor in $\rs X$.
Note that $S \subset \pi(\rs S)$ but the equality need not hold in
general, as seen from Example \ref{example:S-neq-pi-rs-S}.
(An incorrect claim is made in
\cite{Chan&Choi_ext-with-lcv-codim-1}*{\S 2.1} and an earlier version
of this writing.
See Proposition \ref{prop:log-resoln-on-jump-subvar} for the correct
relation between $S$ and $\pi(\rs S)$ and for some sufficient
conditions for the two being equal.)

As $\rs S$ is reduced and has only snc, one can easily describe
$\lcc<\rs X>(\rs* S)$, the union of lc centres of $(\rs X, \rs S)$ of
codimension $\sigma$ in $\rs X$ in the sense of
\cite{Kollar_Sing-of-MMP}*{Def.~4.15}.
Notice the $\holo_X$-isomorphism $\aidlof<X>{\vphi_L} \isom
\pi_*\paren{E \otimes \aidlof<\rs
  X>{\pi^*_\ominus\vphi_L}[\pi^*\psi]}$ in \eqref{eq:log-resoln-aidl}
(or $K_X \otimes \aidlof<X>{\vphi_L} \isom \pi_*\paren{K_{\rs X}
  \otimes R^{-1} \otimes \aidlof<\rs
  X>{\pi^*_\ominus\vphi_L}[\pi^*\psi]}$; see Remark
\ref{rem:residue-norm-on-hol-n-forms-intrinsic}),
and set $\residlof<\rs X>{\pi^*_\ominus\vphi_L;\pi^*\psi}
:=\residlof<\rs X>{\rs\bphi}$ for the consistency of notations on $\rs
X$ and $X$.
The results on $\aidlof<X>{\vphi_L}$ with respect to a log-resolution
$\pi$ can then be stated.
\begin{definition-thm}[Theorem \ref{thm:direct-image-residl},
  Definition \ref{def:non-snc-residl-residue-norm}, Theorem
  \ref{thm:sigma-lc-centres}, Definition
  \ref{def:sigma-lc-centres} and Proposition \ref{prop:relation-of-S-and-lcc}]
  \label{thm:main-results-log-resoln}
  Suppose that $(X,\vphi_L,\psi)$ is given as described in Section
  \ref{sec:setup} (so, in particular,
  $\vphi_L$ has only neat analytic singularities but
  $(X,\vphi_L,\psi)$ may \emph{not} satisfy the Snc assumption
  \ref{assum:snc-nas}).
  \begin{enumerate}[labelindent=0pt, leftmargin=*, itemsep=1ex]
  \item There exists a coherent $\holo_X$-sheaf
    $\residlof<X>{\vphi_L;\psi}$ and a residue morphism $\Res$, both
    are unique up to isomorphisms, such that the sequence
    \begin{equation*}
      \xymatrix{
        0 \ar[r]
        & {\aidlof|\sigma-1|<X>{\vphi_L}} \ar[r]
        & {\aidlof<X>{\vphi_L}} \ar[r]^-{\Res}
        & {\residlof<X>{\vphi_L;\psi}} \ar[r]
        & 0
      }
    \end{equation*}
    is exact for all integers $\sigma \geq 1$.
    The residue norm $\RTF|\cdot|(0)[V,\sigma]$ induces an $L^2$ norm
    $\norm\cdot_{\lcc<V>(\vphi_L;\psi)}$ on 
    $\residlof<X>{\vphi_L;\psi}(\cl V)$ for any open set $V \Subset
    X$ such that the monomorphism 
    \begin{equation*}
      \tau \colon \residlof<X>{\vphi_L;\psi}\paren{\cl V}
      \hookrightarrow
      E \otimes \residlof<\rs X>{\pi^*_\ominus\vphi_L;
        \pi^*\psi}\paren{\pi^{-1}(\cl V)}
    \end{equation*}
    induced via the residue morphisms is an isometric embedding, where
    the target space is equipped with the residue norm
    $\norm\cdot_{\lcc<\pi^{-1}(V)>(\rs* S)}$.

  \item A \emph{$\sigma$-lc centre} of $(X,\vphi_L,\psi)$ is defined to be an
    irreducible component of the reduced closed analytic subset
    $\lcc(\vphi_L;\psi)$ of $X$ given by the ideal sheaf 
    \begin{equation*}
      \defidlof{\lcc(\vphi_L;\psi)}
      :=\Ann_{\holo_X}\paren{
        \frac{\aidlof<X>{\vphi_L}}{\aidlof|\sigma-1|<X>{\vphi_L}}
      } \; .
    \end{equation*}
    It follows that
    \begin{equation*}
      \Ann_{\holo_X}\paren{\frac{
          \aidlof<X>{\vphi_L} 
        }{
          \aidlof|\sigma-1|<X>{\vphi_L}
        }}
      =\Ann_{\holo_X}\paren{\pi_*\paren{ E \otimes
          \frac{
            \aidlof<\rs X>{\pi^*_\ominus\vphi_L}[\pi^*\psi]
          }{
            \aidlof|\sigma-1|<\rs X>{\pi^*_\ominus\vphi_L}[\pi^*\psi]
          }
        }} \; ,
    \end{equation*}
    which implies that the sheaves $\residlof<X>{\vphi_L;
      \psi}$ and $\pi_*\paren{E \otimes \residlof<\rs X>{\pi^*_\ominus
        \vphi_L; \pi^*\psi}}$ (being different in general) have the same
    support. 
    Moreover, in general, one has
    \begin{equation*}
      \lcc(\vphi_L;\psi) \subset \pi\paren{\lcc<\rs X>(\rs* S)} \; .
    \end{equation*}
    If $\lcS*$ is a $\sigma$-lc centre in $\lcc<\rs X>(\rs* S)$
    such that $\pi\paren{\lcS*} \not\subset \lcc(\vphi_L;\psi)$,
    then $\res{\frac{\pi^*f}{\rs \sect_0}}_{\lcS*} \equiv 0$ (or,
    equivalently, $\res{\Res(f)}_{\pi(\lcS*)} \equiv 0$) for all $f \in
    \aidlof<X>{\vphi_L}(V)$ and open set $V \Subset X$ such that
    $\pi^{-1}(V) \cap \lcS* \neq \emptyset$, where $\rs \sect_0$ is a
    holomorphic canonical section of $\rs S_0$.
    A $\sigma$-lc centre $\lcS* \subset \lcc<\rs X>(\rs* S)$ such that
    $\res{\frac{\pi^*f}{\rs\sect_0}}_{\lcS*} \not\equiv 0$ for some $f
    \in \aidlof<X>{\vphi_L}(V)$ is said to be \emph{$\pi$-supportive}.
    
  \item The \emph{index of the mlc of $(X,\vphi_L,\psi)$},
    denoted by $\sigma_{\mlc} := \sigma_{\mlc}\paren{X,\vphi_L,\psi}$,
    is defined to be the smallest non-negative integer such that $\lcc[\sigma
    +1]<X>(\vphi_L;\psi) =\emptyset$ for all $\sigma \geq
    \sigma_{\mlc}$.
    One then has $S = \lcc[1]<X>(\vphi_L;\psi) \cup \dotsm \cup
    \lcc[\sigma_{\mlc}]<X>(\vphi_L;\psi)$. 
    
  \end{enumerate}
  When $(X,\vphi_L,\psi)$ itself satisfies the Snc assumption
  \ref{assum:snc-nas}, the sheaf $\residlof<X>{\vphi_L;\psi}
  =\residlof<X>{\bphi}$, the residue morphism $\Res$ and the union of
  $\sigma$-lc centres $\lcc(\vphi_L;\psi) =\lcc$ all reduce back to
  the corresponding entities considered in Theorem \ref{thm:main-results}. 
\end{definition-thm}

Under the definition of lc centres given above, even when the lc locus
$S$ of $\set{\mtidlof<X>{\vphi_L+m\psi}}_{m \in [0,1]}$ at jumping
number $m=1$ is a reduced snc divisor satisfying $S =\pi(\rs S)$,
there may be more lc centres of $(X,\vphi_L,\psi)$ than those of
$(X,S)$, as illustrated in Example \ref{example:compare-alg-ana-adj}.
On the other hand, when $\Delta$ is an effective $\fieldQ$-divisor on
$X$ (which is $\fieldQ$-Cartier as $X$ is smooth) such that
$(X,\Delta)$ is an lc pair and $\psi_\Delta
:= \phi_\Delta -\sm\vphi_{\Delta} \leq -1$ a global function on
$X$ given according to Notation \ref{notation:potentials}, the
irreducible components in the disjoint union
$\bigsqcup_{\sigma \geq 1}\lcc(0;\psi_\Delta)$ are exactly those
lc centres of $(X,\Delta)$ on $X$ in the sense of
\cite{Kollar_Sing-of-MMP}*{Def.~4.15 and Thm.~4.16} (see Theorem
\ref{thm:aidl-and-singularities-in-MMP-intro} and Proposition
\ref{prop:mlc=rs_mlc-in-lc-system}).
These reflect the fact that the definition of lc centres of
$(X,\vphi_L,\psi)$ via adjoint ideal sheaves given above incorporates
not only the singularity information in $S$ ($\subset
\psi^{-1}(-\infty)$), but also those in $\vphi_L$,
which give some reasons for the difference between the algebraic and
analytic versions of adjoint ideal sheaves in Theorem
\ref{thm:intro-comparison-alg-and-analytic-aidl}.

The residue exact sequence in Definition-Theorem
\ref{thm:main-results-log-resoln} for $\sigma = 1$, together with the
vanishing of the higher direct images of multiplier ideal sheaves via
a log-resolution of $(X,\vphi_L,\psi)$, leads to the following
statement.

\begin{thm}[Theorem \ref{thm:pi-supportive-1-lc-centres}]
  \label{thm:connected-1-lc-centres-intro} 
  Under the notation and assumptions in Definition-Theorem
  \ref{thm:main-results-log-resoln},
  given any $1$-lc centre $\lcS|1|$ of
  $(X,\vphi_L,\psi)$ (which may not be a divisor),
  there is one and only one $\pi$-supportive $1$-lc centre $\rs D
  \subset \lcc[1]<\rs X>(\rs* S)$ of $(\rs X, \rs S)$ such that $\rs
  D$ is mapped into, and thus onto, $\lcS|1|$ by $\pi$.
\end{thm}

This theorem is used in the proofs of Theorems
\ref{thm:intro-comparison-alg-and-analytic-aidl} and
\ref{thm:aidl-and-singularities-in-MMP-intro}.
Although it is a different statement from Theorem
\ref{thm:connected-1-lc-centres-intro}, readers are referred to the
connectedness lemma (see Shokurov
\cite{Shokurov_3-fold-log-flips}*{Connectedness Lemma 5.7} and
\textfr{Kollár} \cite{Kollar_AST_1992}*{Thm.~17.4}) for a comparison.
A brief survey and the recent research regarding the connectedness can
be found in the article of Birkar (\cite{Birkar_non-klt-connectedness}).

Note that the sequence $\seq{\aidlof<X>{\vphi_L}}_{\sigma \geq 0}$
stabilises when $\sigma \geq \sigma_{\mlc}$, i.e.~$\aidlof<X>{\vphi_L}
=\mtidlof<X>{\vphi_L}$ for all $\sigma \geq \sigma_{\mlc}$, by the
definition of $\sigma_{\mlc}$.
This exhibits the role of the adjoint ideal sheaves
$\aidlof<X>{\vphi_L}$ as a measurement of the singularities of the lc
locus $S =\bigcup_{\sigma \geq 1} \lcc(\vphi_L;\psi)$ (and the
corresponding lc centres) of $\set{\mtidlof<X>{\vphi_L+m\psi}}_{m \in
  [0,1]}$ at the jumping number $m=1$, which distantly resembles the
first use of adjoint ideal sheaves in \cite{Ein&Lazarsfeld_adjIdl}.
The following theorem serves as an illustration of this aspect of the
adjoint ideal sheaves.
\begin{thm}[Theorem \ref{thm:aidl-and-singularities-in-MMP}]
  \label{thm:aidl-and-singularities-in-MMP-intro}
  Let $\Delta$ be an effective $\fieldQ$-divisor on a complex manifold
  $X$ and consider the function $\psi_\Delta :=\phi_\Delta
  -\sm\vphi_\Delta$ as described in Notation \ref{notation:potentials}.
  Then,
  \begin{enumerate}
  \item the pair $(X,\Delta)$ is klt if and only if
    $\aidlof|0|<X>{0}[\psi_\Delta] =\holo_X$;

  \item the pair $(X,\Delta)$ is plt if and only if
    $\aidlof|1|<X>{0}[\psi_\Delta] =\holo_X$ and
    $\aidlof|0|<X>{0}[\psi_\Delta] =\defidlof{\floor\Delta}$, the
    defining ideal sheaf of $\floor\Delta$;
    
  \item the pair $(X,\Delta)$ is lc if and only if
    $\aidlof<X>{0}[\psi_\Delta] =\holo_X$ for some (sufficiently
    large) integer $\sigma \geq 0$ (and it suffices to consider only
    $\sigma \in [0,n]$).
  \end{enumerate}
\end{thm}

The residue exact sequence together with Theorem
\ref{thm:aidl-and-singularities-in-MMP-intro} facilitates another
proof of \textfr{Kollár}'s theorem on the inversion of adjunction for
the case when the base space $X$ is smooth (see
\cite{Kollar_AST_1992}*{Thm.~17.6}; also
cf.~\cite{KimDano&Seo_adj-idl}*{Thm.~1.7}).
\begin{thm}[Inversion of adjunction (Theorem
  \ref{thm:inversion-of-adjunction});
  see \cite{Kollar_AST_1992}*{Thm.~17.6} and
  cf.~\cite{KimDano&Seo_adj-idl}*{Thm.~1.7}]
  \label{thm:inversion-of-adjunction-intro}
  On a complex manifold $X$, let $\Delta$ be an effective
  $\fieldQ$-divisor on $X$ such that $S :=\floor\Delta$ is a reduced
  divisor.
  Also let $\nu \colon S^\nu \to S$ be the normalisation of $S$ and
  $\Diff_{S^\nu}\Delta$ the general different (see
  \cite{Kollar_Sing-of-MMP}*{\S 4.2}) such that $K_{S^\nu}
  +\Diff_{S^\nu}\Delta \sim_\fieldQ \nu^*\parres{K_X +\Delta}_{S}$
  (where $\sim_\fieldQ$ means ``$\fieldQ$-linearly equivalent to'').
  Then, $(X, \Delta)$ is plt near $S$ if and only if $(S^\nu,
  \Diff_{S^\nu}\Delta)$ is klt.
\end{thm}

The proof can be found in Theorem \ref{thm:inversion-of-adjunction}.
This proof illustrates that the residue short exact sequences indeed
provide information regarding to the inversion of adjunction in the
more general situations.

\subsection{Further questions}
\label{sec:further-questions}

There are several follow-up questions which, the author believes,
worth further research.
\begin{enumerate}
\item It is natural to ask whether the results here can be carried
  over to the case where the potential $\vphi_L$ (or maybe even
  the function $\psi$) has more general singularities.
  If the result in Theorem \ref{thm:main-results-log-resoln} can be
  carried over, one can then discuss about the $\sigma$-lc centres of
  $(X,\vphi_L,\psi)$ without the need of any log-resolutions.
  The original adjoint ideal sheaf (\cite{Guenancia} and
  \cite{KimDano-adjIdl}) in \eqref{eq:Guenancia-adj-ideal} fits into
  the corresponding residue short exact sequence (and therefore coherent)
  when $e^{\bphi}$ is locally \textde{Hölder} continuous, and
  \cite{Guenancia}*{Remark 2.17} provides an example such that the
  residue map $\Res$ is not well-defined.
  Although it is easy to follow \cite{Guenancia}*{Remark 2.17} to
  construct examples of $\vphi_L$ such that the residue map $\Res$ on
  $\aidlof{\vphi_L}$ is not well-defined, it is still legitimate to
  ask if $\aidlof{\vphi_L}$ could still fit in the residue short exact
  sequence for more general singularities on $\vphi_L$ than those
  restricted by the local \textde{Hölder} continuity condition.
  This, among other things, will be discussed in the subsequent paper.

\item In \cite{Guan&Li_adjIdl-not-coherent}, Guan and Li construct an
  example of a psh potential $\bphi$ such that the original adjoint
  ideal sheaf in \eqref{eq:Guenancia-adj-ideal} is not coherent.
  By setting $\vphi_L := \bphi$ (the example in
  \cite{Guan&Li_adjIdl-not-coherent}*{Proof of Thm.~1.1}) and $\psi
  :=\log\abs{z_1}^2$, it is easy to check that $\set{\mtidlof{\vphi_L
      +m\psi}}_{m \in [0,1]}$ has a sequence of jumping numbers in
  $[0,1]$ accumulating at $m=1$ around the origin and the adjoint
  ideal sheaves $\aidlof{\vphi_L}$ are also not coherent at the origin
  for all $\sigma \geq 1$.
  A natural follow-up question is whether the existence of
  accumulating jumping numbers has any relation with the incoherence
  of $\aidlof{\vphi_L}$'s.

\item In view of the isometric property of the residue norms
  $\norm\cdot_{\lcc(\vphi_L;\psi)}$ and the results in Theorem
  \ref{thm:main-results-log-resoln}, it is tempting to ask whether the
  adjoint ideal sheaves $\aidlof{\vphi_L}$ and the residue norms
  $\norm\cdot_{\lcc(\vphi_L;\psi)}$ (or even the residue function
  $\RTF|\cdot|(\eps)[X,\sigma]$) can be defined on a complex space $X$
  which comes with singularities.
  
\item If one accepts the ``mild'' dependence of the constant $C$ on
  the potential $\bphi$ on the local $L^2$ extension theorem in
  Theorem \ref{thm:main-results} \eqref{item:thm:local-L2-estimate}
  (or Corollary \ref{cor:local-L2-estimates}), the subsequent pressing
  question is whether the local holomorphic extensions with $L^2$
  estimates can be used to obtain the desired global extensions with
  estimates as stated in
  \cite{Chan_on-L2-ext-with-lc-measures}*{Conj.~1.1.3}. 
\end{enumerate}

This paper is organised as follows.
Preliminaries are given in Section \ref{sec:preliminaries}, in which
Sections \ref{sec:notation} and \ref{sec:setup} explain some less
commonly used notations as well as the basic setup and assumptions in
this paper, while Section \ref{sec:snc-assumption} describes the snc
assumptions on $\vphi_L$ and $\psi$ by passing to a suitable
log-resolution.
First properties of the new version of adjoint ideal sheaves are
studied in Section \ref{sec:adjoint-ideal-sheaves}.
Section \ref{sec:residue-short-exact-sequence} exclusively discusses
about the residue short exact sequences that the adjoint ideal sheaves
satisfy, which leads to Theorem \ref{thm:main-results}.
Section \ref{sec:non-snc} studies the adjoint ideal sheaves under
non-snc scenarios which leads to Definition-Theorem
\ref{thm:main-results-log-resoln}.
Theorem \ref{thm:connected-1-lc-centres-intro} is proved there as an
application. 
Section \ref{sec:algebraic-and-analytic-comparison} compares the
current version of adjoint ideal sheaves with the algebraic versions
and proofs of Theorems
\ref{thm:intro-comparison-alg-and-analytic-aidl},
\ref{thm:aidl-and-singularities-in-MMP-intro} and
\ref{thm:inversion-of-adjunction-intro} are given.


\section{Preliminaries}
\label{sec:preliminaries}

\subsection{Notation}
\label{sec:notation}


In this paper, the following notations are used throughout.

\begin{notation}
  Set $\ibar := \ibardefn \;$. \ibarfootnote
\end{notation}

\begin{notation} \label{notation:potential-def}
  Each potential $\vphi$ (of the curvature of
  a metric) on a holomorphic line bundle $L$ in the following
  represents a collection of local functions
  $\set{\vphi_\gamma}_\gamma$ with respect to some fixed local
  coordinates and trivialisation of $L$ on each open set $V_\gamma$ in
  a fixed open cover $\set{V_\gamma}_\gamma$ of $X$.  The functions
  are related by the rule
  $\vphi_\gamma = \vphi_{\gamma'} + 2\Re h_{\gamma \gamma'}$ on
  $V_\gamma \cap V_{\gamma'}$ where $e^{h_{\gamma \gamma'}}$ is a
  (holomorphic) transition function of $L$ on
  $V_\gamma \cap V_{\gamma'}$ (such that
  $s_\gamma = s_{\gamma'}e^{h_{\gamma \gamma'}}$, where $s_\gamma$ and
  $s_{\gamma'}$ are the local representatives of a section $s$ of $L$
  under the trivialisations on $V_\gamma$ and $V_{\gamma'}$
  respectively).
  Inequalities between potentials is meant to be the inequalities
  under the chosen trivialisations over open sets in the fixed open
  cover $\set{V_\gamma}_\gamma$.
\end{notation}

\begin{notation} \label{notation:potentials}
  For any prime (Cartier) divisor $E$, let
  \begin{itemize}
  \item $\phi_E := \log\abs{s_E}^2$, representing the collection
    $\set{\log\abs{s_{E,\gamma}}^2}_{\gamma}$, denote a potential (of
    the curvature of the metric) on the line bundle associated to $E$
    given by the collection of local representations
    $\set{s_{E,\gamma}}_{\gamma}$ of some canonical section $s_E$
    (thus $\phi_E$ is uniquely defined up to an additive constant);
    
  \item $\sm\vphi_E$ denote a smooth potential on the line
    bundle associated to $E$;
    
    
  \item $\psi_E := \phi_E - \sm\vphi_E$, which is a global function
    on $X$, when both $\phi_E$ and $\sm\vphi_E$ are fixed.
  \end{itemize}
  All the above definitions are extended to any $\fieldR$-divisor $E$
  by linearity.
  For notational convenience, the notations for a $\fieldR$-divisor
  and its associated $\fieldR$-line bundle are used interchangeably.
  The notation of a line bundle is also abused to mean its associated
  invertible sheaf.
\end{notation}

\begin{notation} \label{notation:norm-of-n0-form}
  For any $(n,0)$-form (or $K_X$-valued section) $f$, define $\abs f^2
  := c_n f \wedge \conj f$, where $c_n :=
  (-1)^{\frac{n(n-1)}{2}}\paren{\pi\ibar}^n$.
  For any hermitian metric $\omega =\pi\ibar \sum_{1\leq j,k\leq
    n} h_{j\conj k} \:dz^j \wedge d\conj{z^k}$ on $X$, set $d\vol_{X,\omega} :=
  \frac{\omega^{n}}{n!}$.
  Set also $\abs f_\omega^2 \dvol_{X,\omega} = \abs f^2$.
  When the integrand in question is supported in some complex
  coordinate chart of $X$, $\dvol_X$ is used to mean the
  volume form on $X$ given by the Euclidean metric in that coordinate
  chart.
\end{notation}

\begin{notation} \label{notation:ineq-up-to-constants}
  For any two non-negative functions $u$ and $v$,
  write $u \lesssim v$ (equivalently, $v \gtrsim u$) to mean that there
  exists some constant $C > 0$ such that $u \leq C v$, and $u
  \sim v$ to mean that both $u \lesssim v$ and $u \gtrsim v$ hold
  true.
  For any functions $\eta$ and $\phi$, write $\eta \lesssim_\tlog \phi$
  if $e^\eta \lesssim e^\phi$.
  Define $\gtrsim_\tlog$ and $\sim_\tlog$ accordingly.
\end{notation}


\subsection{Basic setup}
\label{sec:setup}



Let $(X,\omega)$ be a \emph{compact} hermitian manifold or a
\emph{relatively compact domain} of some ambient hermitian manifold of
complex dimension $n$.
This is served as the background manifold in this article.
While only the local properties of the analytic adjoint ideal sheaves
are considered in this paper, (relative) compactness is assumed here just to
guarantee that the \emph{jumping number} given below (see
item \eqref{item:jumping-no} below) is well-defined
and there are only finitely many lc centres of $\paren{\vphi_L, \psi}$
described below to be considered.

Let $\mtidlof{\vphi} := \mtidlof[X]{\vphi}$
be the multiplier ideal sheaf of the potential $\vphi$ on $X$ given at
each $x \in X$ by
\begin{equation*}
  \mtidlof{\vphi}_x := \mtidlof[X]{\vphi}_x
  :=\setd{f \in \holo_{X,x}}{
    \begin{aligned}
      &f \text{ is defined on a coord.~neighbourhood } V_x \ni x \vphantom{f^{f^f}} \\
      &\text{and }\int_{V_x} \abs f^2 e^{-\vphi} \dvol_{V_x} < +\infty
    \end{aligned}
  } \; .
\end{equation*}
A potential $\vphi$ is said to have \emph{Kawamata log-terminal (klt)
  singularities} on $X$ if $\mtidlof{\vphi} =\holo_X$ on $X$, and
\emph{log-canonical (lc) singularities} on $X$ if
$\mtidlof{(1-\eps)\vphi} =\holo_X$ for all $\eps > 0$ on $X$.

Throughout this paper, the following are assumed on $X$:

\begin{enumerate}[itemsep=8pt]

\item \label{item:neat-analytic-singularities}
  $(L, e^{-\vphi_L})$ is a hermitian line bundle on $X$ with a singular
  metric $e^{-\vphi_L}$ such that $\vphi_L$ is locally equal
  to $\vphi_1 - \vphi_2$, where each of the
  $\vphi_i$'s is a quasi-plurisubharmonic (quasi-psh) local function
  with \emph{neat analytic singularities}, i.e.~locally
  \begin{equation*}
    \vphi_i \equiv c_i \log\paren{\sum_{j=1}^N \abs{g_{ij}}^2} \mod
    \smooth \; ,
  \end{equation*}
  where $c_i \in \fieldR_{\geq 0}$ is a constant and $g_{ij} \in
  \holo_X$ is a local holomorphic function (which
  comes with its \emph{polar ideal sheaf} $\divP_{\vphi_L}$, the ideal
  sheaf generated by all $g_{ij}$'s) for each $i=1,2$ and $j=1,\dots,N$;

\item \label{item:def-psi}
  $\psi$ is a global function on $X$ which is also locally a
  difference of quasi-psh functions  with the associated \emph{polar
    ideal sheaf} denoted by $\divP_\psi$;

\item \label{item:psi-bounded}
  $\psi \leq -1$ on $X$ and $\vphi_L+m\psi$ is locally bounded
  above on $X$ for each $m \in [0,1]$ (which implies that $\psi$ and
  $\vphi_L+m\psi$ are quasi-psh after some blow-ups as they have only
  neat analytic singularities);


\item \label{item:jumping-no}
  $1$ is the only jumping number of the family $\set{\mtidlof{\vphi_L
      + m {\psi}}}_{m \in [0,1]} \;$, i.e.
  \begin{equation*}
    \mtidlof{\vphi_L}
    = \mtidlof{\vphi_L + m {\psi}}
    \supsetneq \mtidlof{\vphi_L + \psi}
    \quad \text{ for all }m \in [ 0 ,  1) 
  \end{equation*}
  (the jumping numbers exist on (relatively) compact $X$ by the
  openness property of multiplier ideal sheaves of (quasi-)psh
  functions, and accumulation of them as in
  \cite{Guan&Li_cluster-jumping-numbers} and
  \cite{KimDano&Seo_jumping-numbers} does not occur since $\vphi_L$
  and $\psi$ have only neat analytic singularities);%
  \footnote{\label{fn:renormalising-jumping-no}
    Systems $(\vphi_L', \psi')$ with more general jumping numbers $0 \leq m_0 < m_1$ can
    be reduced to the current one by setting $\paren{\vphi_L , \psi}
    := \paren{\vphi_L'+m_0\psi' \: , \: \paren{m_1-m_0} \psi' }$.
    Sometimes it may also be convenient to set $(\vphi_L, \psi) :=
    (\vphi_L' +(m_1-1)\psi', \psi')$, but then one should have 
    $\mtidlof{\vphi_L +\rs m_0\psi} =\mtidlof{\vphi_L +m\psi}
    \supsetneq \mtidlof{\vphi_L +\psi}$ for all $m \in [\rs m_0, 1)$
    with $\rs m_0 := \max\set{m_0 -m_1 +1, 0}$, and
    ``$\mtidlof{\vphi_L}$'' in most part of this paper should be
    replaced by ``$\mtidlof{\vphi_L +\rs m_0\psi}$''.
  }

\item \label{item:def-S} 
  $S \subset \psi^{-1}\paren{-\infty}$ is a \emph{reduced}
  subvariety defined by the annihilator
  \begin{equation*}
    \defidlof{S} := \Ann_{\holo_X} \paren{ \dfrac{\multidl\paren{\vphi_L}}
      {\multidl\paren{\vphi_L +  \psi}} } 
  \end{equation*}
  (see \cite{Demailly_extension}*{Lemma 4.2} for the proof that
  $\defidlof{S}$ is reduced).
  Call $S$ to be the \emph{lc locus}\footnote{
    \label{fn:lc-locus-original-meaning}
    The use of the terminology ``lc locus'' is different from its use
    in some literature like \cite{Ambro_lcc} (in which, when adapted
    to the current setup, the ``lc locus'' of $(X,\vphi_L,\psi)$
    should mean the subset of $X$ on which
    $\mtidlof{(1-\eps)\paren{\vphi_L+\psi}} = \holo_X$ for all $\eps >
    0$).
    In more recent literatures, like \cite{Ambro_injectivity} and
    \cite{Fujino&Gongyo_abundance}, mostly only the complement ``non-lc
    locus'' or its sibling ``non-klt locus'' is considered.
    The author feels safe to use the term ``lc locus'' in the current
    context for a more intuitive description.
  } \emph{of the family $\set{\mtidlof{\vphi_L +m\psi}}_{m\in [0,1]}$
    at jumping number $m=1$} for easy reference.
\end{enumerate}

Note that, since both $\vphi_L$ and $\psi$ are locally differences of
quasi-psh functions with \emph{neat analytic singularities}, it is
easy to prove, via a log-resolution of the polar sets of $\vphi_L$ and
$\psi$ and an application of Fubini's theorem, that $\mtidlof{\vphi_L
  +m\psi}$ is coherent for every $m \in\fieldR$.


\subsection{Effects of log-resolution and the snc assumption}
\label{sec:snc-assumption}


When $\vphi_L$ and $\psi$ have only neat analytic singularities, the
discussion in \cite{Chan&Choi_ext-with-lcv-codim-1}*{\S 2.1} shows how
the setting can be reduced to the snc scenario (the subvariety $S$ and
the polar sets of $\vphi_L$ and $\psi$ have only \emph{simple normal
  crossings (snc)} against one another) via some suitable
log-resolution according to \cite{Hironaka}.
For the convenience of discussion in the subsequent papers, assume,
only in this section and unless otherwise stated, that
\begin{itemize}
\item the potential $\vphi_L$ may have singularities \emph{worse than} neat
  analytic singularities and 
\item $\vphi_L +m\psi$ is quasi-psh for each $m \in [\rs m_0, 1]$ for some
  number $\rs m_0 \in [0,1)$ so that $\mtidlof[X]{\vphi_L+m\psi}$ (as well
  as $\mtidlof[\rs X]{\pi^*\vphi_L+m\pi^*\psi}$ for some
  log-resolution $\pi$) is coherent.
\end{itemize}
While resolving the singularities of $\vphi_L$ as before is not
possible when it has more general singularities, Lemma
\ref{lem:principal-annihilator-criterion} below suggests an snc
assumption on $\mtidlof[X]{\vphi_L+\psi}$ which may be useful for the
study in subsequent papers.

Let $\pi\colon \rs X \to X$ be any modification (proper
generically $1$-to-$1$ holomorphic map) of $X$ such that $\rs X$ is
smooth.
Assume in what follows that the inverse image $\divP_\psi \cdot
\holo_{\rs X}$ of the polar ideal sheaf of $\psi$ is $\holo_{\rs
  X}(-P_\psi)$ for some effective snc divisor $P_\psi$ on $\rs X$.
Moreover, the exceptional locus $\Exc(\pi)$ of $\pi$ is an snc divisor
and $\Exc(\pi) +P_\psi$ also has only snc.
Such modification is referred to as a log-resolution of $(X,\psi)$ in
what follows.
(When $\vphi_L$ has only neat analytic singularities, one can also
consider a log-resolution of $(X,\vphi_L,\psi)$ such that
$\Exc(\pi)+P_{\vphi_L}+P_\psi$ has only snc.)
Note that the function $\pi^*\psi$ is then quasi-psh on $\rs X$. 

Let $E_{d\pi}$ be the exceptional divisor defined scheme-theoretically
by the ideal sheaf generated by the holomorphic Jacobian of $\pi$
(which is linearly equivalent to the relative canonical divisor
$K_{\rs X / X} :=K_{\rs X} -\pi^*K_X$).
Consider the decomposition
\begin{equation*}
  E_{d\pi} = E + R
\end{equation*}
of $E_{d\pi}$ into two effective $\Znum$-divisors $E$ and $R$ (with
the corresponding canonical holomorphic sections $\sect_E$ and
$\sect_R$ fixed) such that $R$ is the \emph{maximal} divisor
satisfying
\begin{equation*}
  \alert{\pi^*\vphi_L -\phi_R} +\pi^*\psi
  =: \alert{\pi^*_\ominus\vphi_L} +\pi^*\psi
  \quad\text{ being quasi-psh} \; ,
\end{equation*}
where $\phi_R :=\log\abs{\sect_R}^2$ (see Notation
\ref{notation:potentials}; set also $\phi_E :=\log\abs{\sect_E}^2$).
Notice that, with the divisor $R$ (hence $E$) such defined, the weight
$e^{-\pi^*_\ominus\vphi_L-\pi^*\psi}$ is locally integrable around
general points of each component of $E$ \footnote{
  \label{fn:E-and-S-no-common-comp}
  To see this, consider Siu's decomposition for the closed positive
  current $\ibddbar\paren{\pi^*\vphi_L +\pi^*\psi}$ (see,
  \cite{Siu_closed-pos-currents}*{\S 6} or \cite{Demailly}*{Ch.~III,
    (8.17)}), which assures that $R$ (hence $E$) is chosen such
  that each $\lelong{\pi^*_{\ominus}\vphi_L +\pi^*\psi}[D_i]$, the generic
  Lelong number of $\pi^*_{\ominus}\vphi_L +\pi^*\psi$ along a
  component $D_i$ of $E$, has to be $< 1$. 
  The claim then follows from Skoda's lemma (see
  \cite{Skoda_Analytic-subsets}*{\S 7} or
  \cite{Demailly_multiplier-ideal-sheaves}*{Lemma (5.6)}).
} (so $E$ contains no components of $\rs S$ defined below).

The following lemma is an analogous result of (and a correction
to\footnote{
  In \cite{Chan&Choi_ext-with-lcv-codim-1}*{\S 2.1}, it is incorrectly
  claimed that $S =\pi(\rs S)$ (or, more precisely, $\defidlof{S}
  =\pi_* \defidlof{\rs S}$ with the sheaves given by the annihilators
  as in item \eqref{item:def-S} in Section \ref{sec:setup} and in
  Proposition \ref{prop:log-resoln-on-jump-subvar}).
}) the
one in \cite{Chan&Choi_ext-with-lcv-codim-1}*{\S 2.1}.

\begin{prop} \label{prop:log-resoln-on-jump-subvar}
  For any number $m \in [0,1]$, one has
  \begin{align*} 
    \mtidlof[X]{\vphi_L+m \psi} \cdot \holo_{\rs X}
    &
      \hookrightarrow
      \mtidlof[\rs X]{\pi^*_\ominus\vphi_L -\phi_E +m \pi^*\psi}
      \xhookrightarrow{\:\otimes \sect_E\:}
      E \otimes
      \mtidlof[\rs X]{\pi^*_\ominus\vphi_L +m \pi^*\psi}
      \intertext{and} 
    \mtidlof[X]{\vphi_L+m \psi}
    &\isom \pi_* \mtidlof[\rs X]{\pi^*_\ominus\vphi_L -\phi_E
      +m \pi^*\psi}
      \isom \pi_*\paren{E \otimes \mtidlof[\rs X]{\pi^*_\ominus\vphi_L
      +m \pi^*\psi}} \; , 
  \end{align*}
  where the maps are respectively globally defined $\holo_{\rs X}$- and
  $\holo_X$-homomorphisms and both maps on the far right-hand-side
  depend on the choice of $\sect_E$.
  The family $\set{\mtidlof[\rs X]{\pi^*_\ominus\vphi_L
      +m\pi^*\psi}}_{m \in [\rs m_0, 1]}$ (as well as
  $\set{\mtidlof[\rs X]{\pi^*_\ominus\vphi_L -\phi_E +m\pi^*\psi}}_{m \in [\rs
    m_0, 1]}$) has a jumping number at $m=1$.

  Moreover, when $\vphi_L$ has only neat analytic singularities, there
  exists a number $m_0 \in [\rs m_0,1)$ sufficiently close
  to $1$ such that
  \begin{equation*}
    \mtidlof[\rs X]{\pi^*_\ominus\vphi_L +m
      \pi^*\psi} \;\;\text{ is coherent for all } m \in [m_0, 1]
  \end{equation*}
  and  
  \begin{equation*}
    \begin{aligned}
      \mtidlof[\rs X]{\pi^*_\ominus\vphi_L +m_0\pi^*\psi}
      &=\mtidlof[\rs X]{\pi^*_\ominus\vphi_L +m\pi^*\psi} \\
      &\supsetneq \mtidlof[\rs X]{\pi^*_\ominus\vphi_L +\pi^*\psi}
      \quad\text{ for all } m\in [m_0,1) \; .
    \end{aligned}
  \end{equation*} 
  Let $\rs S$ be the lc locus of the family $\set{\mtidlof[\rs X]{\pi^*_\ominus\vphi_L
      +m\pi^*\psi}}_{m \in [\rs m_0, 1]}$ at jumping number $m=1$,
  i.e.~the reduced subvariety defined by the ideal sheaf \mmark{
  \begin{equation*}
    \defidlof{\rs S}
    :=\Ann_{\holo_{\rs X}} \paren{\frac{
        \mtidlof[\rs X]{\pi^*_\ominus\vphi_L +m_0\pi^*\psi}
      }{
        \mtidlof[\rs X]{\pi^*_\ominus\vphi_L +\pi^*\psi}
      }} \; .
  \end{equation*} %
  }{Reference to show that $\rs S$ defined by the family with $\phi_E$
  is the same. \alert{Done in Remark \ref{rem:defidlof-rs_S-w-phi_E}.}}%
  Then, one has, in general,
  \begin{equation*}
    S \subset \pi\paren{\rs* S} 
  \end{equation*}
  and the equality holds when, for example, either $\psi^{-1}(-\infty) = S$ or
  $\mtidlof<X>{\vphi_L+m_0\psi} =\holo_X$.
\end{prop}

\begin{proof}
  For any $f \in \holo_X(V)$ on any open set $V \subset X$ and for any
  $m \in [0,1]$, one has 
  \begin{equation*} \tag{$*$} \label{eq:pf:pullback-integral}
    \int_V \abs f^2 \:e^{-\vphi_L -m\psi} \dvol_X
    \sim \int_{\pi^{-1}\paren{V}} \abs{\pi^*f \cdot \sect_E}^2
    \:e^{-\pi^*_\ominus\vphi_L  -m\pi^*\psi} \dvol_{\rs X} \; .
  \end{equation*}
  Recalling that $\abs{\sect_E}^2 = e^{\phi_E}$, this implies
  immediately the monomorphisms $\mtidlof[X]{\vphi_L+m\psi} 
  \cdot \holo_{\rs X}
  \hookrightarrow
  \mtidlof[\rs X]{\pi^*_\ominus\vphi_L -\phi_E +m\pi^*\psi}
  \xhookrightarrow{\:\otimes \sect_E\:}
  E \otimes \mtidlof[\rs X]{\pi^*_\ominus\vphi_L +m\pi^*\psi}$ and
  $\mtidlof[X]{\vphi_L+m\psi}
  \hookrightarrow \pi_*\mtidlof[\rs X]{\pi^*_\ominus\vphi_L -\phi_E
    +m\pi^*\psi}
  \hookrightarrow \pi_*\paren{E \otimes \mtidlof[\rs X]{\pi^*_\ominus\vphi_L
    +m\pi^*\psi}}$ given by $f \mapsto \pi^*f \mapsto \pi^*f \cdot \sect_E$.
  On the other hand, following the argument in
  \cite{Demailly_multiplier-ideal-sheaves}*{Prop.~(5.8)},
  as $\pi$ is biholomorphic on the complement of $Z :=
  \pi\paren{\Exc(\pi)}$ in $X$ and as $Z \supset \pi\paren{\supp E}$,
  any section $\rs f \in E \otimes \mtidlof[\rs
  X]{\pi^*_\ominus\vphi_L +m\pi^*\psi}\paren{\pi^{-1}(V)}$ gives rise
  to a holomorphic section $f \in
  \mtidlof[X]{\vphi_L+m\psi}\paren{V\setminus Z}$ such that $\pi^*f
  =\frac{\rs f}{\sect_E}$ on $\pi^{-1}\paren{V \setminus Z}$.
  It then follows from \eqref{eq:pf:pullback-integral} (where $\pi^*f
  \cdot \sect_E$ is replaced by $\rs f$ and $V$ by $V'\setminus Z$
  for any $V' \Subset V$), together with the fact that $\vphi_L+m\psi$
  is locally bounded from above, that $f$ is in $\Lloc[2](V)$ (in the
  unweighted norm).
  Thus $f$ can be analytically extended across $Z$ via the $L^2$
  Riemann extension theorem (\cite{Demailly_complete-Kahler}*{Lemma
    6.9}) and it follows that $\rs f = \pi^*f \cdot \sect_E$ on
  $\pi^{-1}(V)$ by the identity theorem.
  (Alternatively, one can also use the fact that $Z$ has codimension
  $\geq 2$ in $X$ to see that $f$ can be extended to $X$; see
  \cite{Grauert&Remmert-Modifikationen}*{Satz 7} or
  \cite{Grauert&Remmert-CAS}*{10.6.2}.)
  One therefore obtains the desired isomorphisms
  $\mtidlof[X]{\vphi_L+m\psi}
  \isom \pi_*\mtidlof[\rs X]{\pi^*_\ominus\vphi_L -\phi_E +m\pi^*\psi}
  \isom \pi_*\paren{E \otimes \mtidlof[\rs X]{\pi^*_\ominus\vphi_L
      +m\pi^*\psi}}$.

  If $\pi^*_\ominus\vphi_L +m\pi^*\psi$ is quasi-psh on $\rs
  X$, then its associated multiplier ideal sheaf is already known to
  be coherent.
  In the general situation, around every point $y \in \rs X$, viewing
  $\sect_{R}$ as a local defining function of $R$,
  there is a \emph{local} isomorphism
  \begin{equation*}
    \xymatrix@R=0.4cm{
      {\mtidlof[\rs X]{\pi^*\vphi_L
          -\phi_{R} +m\pi^*\psi}} \ar[r]^-{\isom}
      \ar@{}[d]|*[left]{\in}
      &
      {\mtidlof[\rs X]{\pi^*\vphi_L +m\pi^*\psi} \cap
        \holo_{\rs X}\paren{-R}} \ar@{}[d]|*[left]{\in} \\
      {\parbox{3cm}{\centering $\rs f$}} \ar@{|->}[r] 
      &
      {\parbox{4cm}{\centering $\;\rs f \: \sect_{R} \; .$}}
    }
  \end{equation*}
  As $\mtidlof[\rs X]{\pi^*\vphi_L +m\pi^*\psi}$ (hence $\mtidlof[\rs
  X]{\pi^*\vphi_L +m\pi^*\psi} \cap \holo_{\rs X}\paren{-R}$)
  is coherent for all $m \in [\rs m_0, 1]$ and coherence is a local
  property, it follows that $\mtidlof[\rs X]{\pi^*_\ominus\vphi_L
    +m\pi^*\psi}$ is coherent for all $m \in [\rs m_0, 1]$.
  (The same argument also verifies that $\mtidlof<\rs X>{\pi^*_\ominus
  \vphi_L -\phi_E +m\pi^*\psi}$ is coherent.)
  Moreover, the local isomorphism also implies that the openness
  property of multiplier ideal sheaves of quasi-psh functions still
  applies
  so that one can talk about the jumping numbers of
  the family $\set{\mtidlof[\rs X]{\pi^*_\ominus\vphi_L
      +m\pi^*\psi}}_{m\in [\rs m_0, 1]}$ (under the assumption that
  $X$ is compact or is a relatively compact domain in an ambient
  manifold).
  Therefore, as can be seen from \eqref{eq:pf:pullback-integral}, if
  \begin{equation*}
    f \in \mtidlof[X]{\vphi_L+\alert{m}\psi}_x \setminus
    \mtidlof[X]{\vphi_L+\psi}_x
  \end{equation*} 
  at some $x \in X$ for all $m < 1$ sufficiently close to $1$, then
  one has
  \begin{equation*}
    \pi^*f \cdot \sect_E \in \mtidlof[\rs X]{\pi^*_\ominus\vphi_L
      +\alert{m}\pi^*\psi}_y \setminus 
    \mtidlof[\rs X]{\pi^*_\ominus\vphi_L +\pi^*\psi}_y
  \end{equation*}
  for some $y \in \pi^{-1}(x)$ and for all $m < 1$ sufficiently close
  to $1$.
  This implies that $m=1$ is a jumping number of the family
  $\set{\mtidlof[\rs X]{\pi^*_\ominus\vphi_L +m\pi^*\psi}}_{m \in
    [\rs m_0,1]}$, as it is a jumping number of the family
  $\set{\mtidlof[X]{\vphi_L+m\psi}}_{m \in [\rs m_0,1]}$.
  The same argument goes for the family $\set{\mtidlof[\rs
    X]{\pi^*_\ominus\vphi_L -\phi_E +m\pi^*\psi}}_{m \in [\rs m_0,1]}$.

  When $\vphi_L$, as well as $\psi$, has only neat analytic
  singularities, the family $\set{\mtidlof[\rs X]{\pi^*_\ominus\vphi_L
      +m\pi^*\psi}}_{m \in [\rs m_0,1]}$ cannot have accumulating
  jumping numbers.
  The openness property of multiplier ideal sheaves of quasi-psh
  functions then guarantees that there exists a number
  $m_0 \in [\rs m_0,1)$ which satisfies the claim in this
  lemma.
  (Note that $m_0$ may not be $0$ or $\rs m_0$; see
  \cite{Chan&Choi_ext-with-lcv-codim-1}*{footnote 4}.)

  Finally, the fact that $\rs S$ is reduced follows from the argument in
  \cite{Demailly_extension}*{Lemma (4.2)}. 
  Note that one has the vanishing of the first direct image sheaf
  \begin{equation*}
    R^1\pi_*\paren{E \otimes \mtidlof[\rs X]{\pi^*_\ominus\vphi_L
        +\pi^*\psi}}
    =R^1\pi_*\paren{K_{\rs X} \otimes R^{-1} \otimes
      \mtidlof[\rs X]{\pi^*_\ominus\vphi_L +\pi^*\psi}
    } \otimes K_X^{-1}
    =0
  \end{equation*}
  by the local vanishing theorem
  (\cite{Lazarsfeld_book-II}*{Thm.~9.4.1}) in the algebraic setting
  or the generalisation of the Grauert--Riemenschneider vanishing
  theorem by Matsumura
  (\cite{Matsumura_injectivity-Kaehler}*{Cor.~1.5}) in the analytic
  setting (one has to replace $\vphi_L$ in $\mtidlof[\rs
  X]{\pi^*_\ominus\vphi_L +\pi^*\psi}$ by $\vphi_L+\sm\vphi$
  on some local open set in $X$ for some smooth local function
  $\sm\vphi$ so that $\vphi_L+\sm\vphi+\psi$ is psh on the local set
  when applying Matsumura's result).
  It then follows immediately from $\mtidlof[X]{\vphi_L+m\psi}
  \isom \pi_*\paren{E \otimes \mtidlof[\rs X]{\pi^*_\ominus\vphi_L
      +m\pi^*\psi}}$ and the exact sequence induced from the
  inclusion $\mtidlof[\rs X]{\pi^*_\ominus\vphi_L +\pi^*\psi}
  \hookrightarrow \mtidlof[\rs X]{\pi^*_\ominus\vphi_L +m_0\pi^*\psi}$
  that 
  \begin{equation*}
    \defidlof{S}
    =\Ann_{\holo_X} \paren{
      \frac{
        \mtidlof[X]{\vphi_L+m_0\psi}
      }{
        \mtidlof[X]{\vphi_L+\psi}
      }
    }
    =\Ann_{\holo_{X}} \paren{
      \pi_* \paren{E \otimes \frac{
        \mtidlof[\rs X]{\pi^*_\ominus\vphi_L +m_0\pi^*\psi}
      }{
        \mtidlof[\rs X]{\pi^*_\ominus\vphi_L +\pi^*\psi}
      }}
    } \; . \footnotemark
  \end{equation*}\footnotetext{
    The argument actually shows that the quotients inside
    $\Ann_{\holo_X} \paren{\dotsm}$ on both sides are equal.
  }%
  For any coherent sheaf $\sheaf F$ on $\rs X$, the zero locus of
  $\Ann_{\holo_{\rs X}} \!\sheaf F$ is indeed $\supp_{\rs X} \sheaf F$ (see, for
  example, \cite{Grauert&Remmert-CAS}*{A.4.5}).
  Moreover, since $\pi$ is proper with $\pi_*\holo_{\rs X} =\holo_X$
  (thus having connected fibres in particular), it is easy to check
  that $\pi_*\sqrt{\Ann_{\holo_{\rs X}} \!\sheaf F} =\defidlof{\pi\paren{\supp_{\rs X}
      \sheaf F}}$ (the defining ideal sheaf of the analytic set
  $\pi\paren{\supp_{\rs X} \sheaf F}$ in $X$).
  It is also easy to see that $\Ann_{\holo_X} \paren{\pi_*\sheaf F}
  \supset \pi_*\Ann_{\holo_{\rs X}} \sheaf F$.
  Note that $\Ann_{\holo_{\rs X}}\paren{E \otimes \frac{
      \mtidlof[\rs X]{\pi^*_\ominus\vphi_L +m_0\pi^*\psi}
    }{
      \mtidlof[\rs X]{\pi^*_\ominus\vphi_L +\pi^*\psi}
    }} =\Ann_{\holo_{\rs X}}\paren{\frac{
      \mtidlof[\rs X]{\pi^*_\ominus\vphi_L +m_0\pi^*\psi}
    }{
      \mtidlof[\rs X]{\pi^*_\ominus\vphi_L +\pi^*\psi}
    }} = \defidlof{\rs S}$ (since $E$, as a sheaf, is locally free)
  which is radical (this is why $\rs S$ is reduced).
  It follows that $\defidlof{S} \supset \pi_*\defidlof{\rs S}
  =\defidlof{\pi\paren{\rs* S}}$ and therefore $S \subset \pi(\rs S)$. 

  Since $\rs S \subset \pi^{-1}\circ \psi^{-1}(-\infty)$, the above
  result implies that $S =\pi\paren{\rs* S}$ when $\psi^{-1}(-\infty)
  = S$.
  To see that $\mtidlof<X>{\vphi_L+m_0\psi} =\holo_X$ also implies $S
  = \pi\paren{\rs* S}$, notice that the assumption implies that
  $\mtidlof<\rs X>{\pi^*_\ominus \vphi_L +m_0\pi^*\psi} =\holo_{\rs
    X}$, so (using the fact that
  $\Ann_{\holo_X}\paren{\frac{\holo_X}{\mtidlof<X>{\dotsm}}}
  =\mtidlof<X>{\dotsm}$)
  \begin{equation*}
    \defidlof S =\mtidlof<X>{\vphi_L+\psi}
    =\pi_*\paren{
      E \otimes \mtidlof<\rs X>{\pi^*_\ominus \vphi_L +\pi^*\psi}
    } =\pi_*\defidlof{\rs S} =\defidlof{\pi(\rs S)} \; ,
  \end{equation*}
  which gives $S =\pi(\rs S)$.
\end{proof}


\begin{remark} \label{rem:defidlof-rs_S-w-phi_E}
  When $\vphi_L$ has only neat analytic singularities and $\pi \colon
  \rs X \to X$ is a log-resolution of $(X,\vphi_L,\psi)$, it follows
  from the choice of $E$ (as well as $R$; see footnote
  \ref{fn:E-and-S-no-common-comp} on page
  \pageref{fn:E-and-S-no-common-comp}) and Fubini's theorem that
  $\mtidlof<\rs X>{\pi^*_\ominus \vphi_L -\phi_E +m\pi^*\psi }
  =\mtidlof<\rs X>{\pi^*_\ominus \vphi_L +m\pi^*\psi}$ for $m \geq 0$.
  In this case, $\rs S$ can also be defined by $\defidlof{\rs S}
  :=\Ann_{\holo_{\rs X}} \paren{\frac{
      \mtidlof[\rs X]{\pi^*_\ominus\vphi_L -\phi_E +m_0\pi^*\psi}
    }{
      \mtidlof[\rs X]{\pi^*_\ominus\vphi_L -\phi_E +\pi^*\psi}
    }}$.
\end{remark}

\begin{remark}
  One may follow
  \cite{Demailly_multiplier-ideal-sheaves}*{Prop.~(5.8)} to see that
  $K_X \otimes \mtidlof[X]{\vphi_L +m\psi} = \pi_*\paren{K_{\rs X}
    \otimes \mtidlof[\rs X]{\pi^*\vphi_L +m \pi^*\psi}}$ and obtain
  \begin{equation*}
    \Ann_{\holo_X} \paren{
      \frac{
        \mtidlof[X]{\vphi_L+m_0\psi}
      }{
        \mtidlof[X]{\vphi_L+\psi}
      }
    }
    =\Ann_{\holo_{\rs X}} \paren{ \pi_*\paren{
        K_{\rs X} \otimes 
        \frac{
          \mtidlof[\rs X]{\pi^*\vphi_L +m_0\pi^*\psi}
        }{
          \mtidlof[\rs X]{\pi^*\vphi_L +\pi^*\psi}
        }
      }} \; . 
  \end{equation*}
  However, the lc locus $\rs S'$ of the family $\set{\mtidlof[\rs
    X]{\pi^*\vphi_L +m\pi^*\psi}}_{m \in [m_0, 1]}$ at $m=1$ may
  contain more components than $\rs S$ defined in Proposition
  \ref{prop:log-resoln-on-jump-subvar} does.
  The advantage of considering $(\rs X, \rs S)$ over $(\rs X,
  \rs S')$ as the resolved model of $(X,\vphi_L,\psi)$ is that,
  when $\mtidlof<X>{\vphi_L +m_0\psi} =\holo_X$ and if the pair
  $(X,S)$ is \emph{lc and log-smooth}, the codimension of the minimal
  lc centres (mlc) of $(X,S)$ is preserved by $(\rs X, \rs S)$ but may
  not be so for $(\rs X, \rs S')$ (more generally, when
  $\mtidlof<X>{\vphi_L +m_0\psi} =\holo_X$, the index $\sigma =\sigma_{\mlc}$
  at which the increasing sequence $\set{\aidlof<X>{\vphi_L}}_{\sigma
    \in \Nnum}$ of adjoint ideal sheaves introduced in Definition
  \ref{def:adjoint-ideal-sheaves} stabilises is equal to the
  codimension of the mlc of $(\rs X, \rs S)$; see \mmark{Proposition
    \ref{prop:mlc=rs_mlc-in-lc-system}}{Reference needed. \alert{Done.}}).
  For example, consider the unit $2$-disc $X := \Delta^2 \subset
  \fieldC^2$ centred at the origin under the holomorphic coordinate
  system $(z_1,z_2)$ and take
  \begin{equation*}
    \vphi_L := 0 \quad\text{ and }\quad
    \psi := \log\abs{z_1}^2 -1 \; .
  \end{equation*}
  Then, $m=1$ is a jumping number of the family
  $\set{\mtidlof[X]{m\psi}}_{m\in [0,1]}$ and the corresponding lc
  locus $S$ is $\set{z_1 = 0}$.
  Let $\pi \colon \rs X \to X$ be the blow-up of the origin with
  exceptional divisor $R$.
  Then the subvariety $\rs S$ defined as in Proposition
  \ref{prop:log-resoln-on-jump-subvar} is the proper transform
  $\pi^{-1}_*S$ of $S$, while the lc locus of $\set{\mtidlof[\rs X]{m
      \pi^*\psi}}_{m \in [0,1]}$ at the jumping number $m=1$ is
  $\pi^{-1}_*S + R$.
  Note that $\pi^{-1}_*S$ and $R$ have non-empty intersection.
  Note also that there exists $f \in \mtidlof<X>{m_0\psi}$ such that
  $\res{\pi^*f}_{\pi^{-1}_*S \cap R} \not\equiv 0$ (thus also
  $\res{\pi^*f}_{R} \not\equiv 0$) for the comparison
  with Examples \ref{example:S-neq-pi-rs-S} and
  \ref{example:lcc-neq-pi_lcc-w-S=pi_S}.
\end{remark}

\begin{example} \label{example:S-neq-pi-rs-S}
  This is an example where $S \neq \pi(\rs S)$ (and the codimension of
  the mlc is not preserved in this case).
  Let $X := \Delta^3 \subset \fieldC^3$ be the unit $3$-disc centred
  at the origin under the holomorphic coordinate system
  $(z_1,z_2,z_3)$.
  Take
  \begin{equation*}
    \vphi_L := 0 \quad\text{ and }\quad
    \psi := \log\abs{z_1}^2 +\frac 32 \log\abs{z_2}^2 +\frac 12
    \log\paren{\abs{z_2}^2 +\abs{z_3}^2} -1 
  \end{equation*}
  and consider the blow-up $\pi \colon \rs X \to X$ of the line
  $\set{z_2 = z_3 = 0}$ which gives an exceptional divisor $R$.
  It follows that
  \begin{equation*}
    K_{\rs X} \sim \pi^*K_X +R \; .
  \end{equation*}
  Set $S_j := \set{z_j = 0}$ and let $\rs S_j$ be the proper transform
  of $S_j$ for $j =1,2,3$.
  Using the notation in Notation \ref{notation:potentials}, one then
  has
  \begin{equation*}
    \pi^*\psi \sim_{\tlog} \phi_{\rs S_1} +\frac 32 \phi_{\rs S_2}
    +2\phi_R
    \quad\text{ and }\quad
    \pi^*_\ominus \vphi_L =-\phi_R \; .
  \end{equation*}
  Set $m_0 := \frac 23$.
  It can be seen that the family $\set{\mtidlof<\rs
    X>{\pi^*_\ominus\vphi_L +m\pi^*\psi}}_{m \in [m_0, 1]}$ has the
  only jumping number at $m=1$ with the lc locus $\rs S = \rs S_1 +R$.

  On the other hand, one has $\mtidlof<X>{\psi} = \genby{z_1 z_2}$
  since
  \begin{equation*}
    e^{-\psi}
    \sim \frac{1}{
      \abs{z_1}^2 \:\abs{z_2}^3 \paren{\abs{z_2}^2
        +\abs{z_3}^2}^{\frac 12}
    }
    \gtrsim \frac{1}{\abs{z_1}^2 \:\abs{z_2}^3} \not\in \Lloc(X) \; ,
  \end{equation*}
  which implies that $\mtidlof<X>{\psi} \subset \genby{z_1 z_2}$, and,
  for any open set $V \Subset X$
  (note that $\frac 12 \pi^*\log \abs{z_2}^2 \sim_\tlog \frac 12
  \paren{\phi_{\rs S_2} +\phi_R}$ and $\frac 12
  \pi^*\log\paren{\abs{z_2}^2 +\abs{z_3}^2} \sim_\tlog \frac 12 \phi_R$),
  \begin{align*}
    \int_V \abs{z_1z_2}^2 \:e^{-\psi} \dvol_X
    &\sim
    \int_V \frac{\dvol_X}{
      \abs{z_2} \paren{\abs{z_2}^2 +\abs{z_3}^2}^{\frac 12}
    } \\
    &\sim \int_{\mathrlap{\pi^{-1}(V)}} \;\;\;e^{-\frac 12
      \paren{\phi_{\rs S_2} +\phi_R} -\frac 12 \phi_R}
      \:e^{\phi_R} \dvol_{\rs X}
      =\int_{\mathrlap{\pi^{-1}(V)}}
      \;\;\;e^{-\frac 12 \phi_{\rs S_2}}\dvol_{\rs X} < +\infty \; ,
  \end{align*}
  which implies $\genby{z_1z_2} \subset \mtidlof<X>{\psi}$.
  A similar analysis with $e^{-m_0\psi} \sim
  \frac{1}{\abs{z_1}^{\frac 43} \:\abs{z_2}^2 \paren{\abs{z_2}^2
      +\abs{z_3}^2}^{\frac 13}}$ yields $\mtidlof<X>{m_0\psi}
  =\genby{z_2}$ ($\neq \holo_X$).
  Therefore, the lc locus $S$ of the family $\set{\mtidlof<X>{m
      \psi}}_{m \in [m_0,1]}$ at jumping number $m=1$ is $S = S_1
  =\set{z_1 = 0}$, but $\pi(\rs S) = S_1 \cup \paren{S_2 \cap S_3}
  \supsetneq S$.
  Note that $R$ is the ``extra'' component in $\rs S$ which is not
  mapped into $S$ and $\res{\pi^*f}_R \equiv 0$ for all $f \in
  \mtidlof<X>{m_0\psi}$.
  See \mmark{Examples \ref{example:lcc-neq-pi_lcc-w-S-neq-pi_S} and
    \ref{example:lcc-neq-pi_lcc-w-S=pi_S}}{Reference
    needed. \alert{Done.}} for the similar phenomenon when $\sigma$-lc 
  centres are considered.
\end{example}

For the moment, it is not clear to the author whether there exists a
log-resolution $\pi \colon \rs X \to X$ such that the corresponding
$\rs S$ has to be a divisor when $\vphi_L$ has arbitrary
singularities.
However, there is the following lemma.


\begin{lemma} \label{lem:principal-annihilator-criterion}
  If the ideal sheaf $\mtidlof[X]{\vphi_L+\psi}$ is principal with an
  snc generator (i.e.~the generator is locally a monomial in some
  holomorphic coordinate system), then so is $\Ann_{\holo_X}
  \paren{\frac{
      \mtidlof[X]{\vphi_L+m_0\psi}
    }{
      \mtidlof[X]{\vphi_L+\psi}
    }}$.
\end{lemma}

\begin{proof}
  This is a property of unique factorisation domains.
  Suppose, at some $x \in X$, that $\mtidlof[X]{\vphi_L+\psi}_x =
  \genby{g}$ and that $\set{f_j}_{j \in J} \subset
  \mtidlof[X]{\vphi_L+m_0\psi}_x$ is a generating set of
  $\mtidlof[X]{\vphi_L+m_0\psi}_x$ (coherence assures that the index
  set $J$ is finite).
  Assume that there exist two distinct generators $h_1$ and $h_2$ in a
  \emph{minimal} generating set of the annihilator at $x$ (which is
  finitely generated by coherence; see, for example,
  \cite{Grauert&Remmert-CAS}*{Prop.~A.4.5}).
  If $\mu = \gcd\paren{h_1,h_2}$ (the \emph{greatest common divisor} of $h_1$
  and $h_2$, uniquely determined up to a multiple of unit in
  $\holo_{X,x}$) is an element of the annihilator in 
  question, then $h_1$ and $h_2$ can be replaced by $\mu$ in the
  generating set of the annihilator, contradicting minimality.
  Therefore, $\mu$ is not in the annihilator.
  By definition of an annihilator, there exist $\alpha_{ij}
  \in \holo_{X,x}$ for $i=1,2$ and $j \in J$ such that
  \begin{equation*}
    h_i f_j = g \alpha_{ij} \; .
  \end{equation*}
  Writing $h_i = h_i' \mu$
  for $i=1,2$,
  this implies that 
  \begin{equation*}
    \frac{h_1'}{h_2'} =\frac{h_1}{h_2}
    =\frac{\alpha_{1j}}{\alpha_{2j}}
    \quad\text{ for all } j \in J \; .
  \end{equation*}
  Since $h_1'$ and $h_2'$ are relatively prime, one has $h_i' |
  \alpha_{ij}$ for $i=1,2$ and $j \in J$, which implies that $\mu$ lies
  in the annihilator as
  \begin{equation*}
    \mu f_j = g \frac{\alpha_{ij}}{h_i'} \quad\text{ for all } j \in J
    \; ,
  \end{equation*}
  giving a contradiction.
  Therefore, the minimal generating set of the annihilator can have
  only one element $h$, i.e.~the annihilator is principal at $x$.
  
  To prove that $h$ is an snc generator, write $J = \set{1,\dots,r}$
  and 
  \begin{equation*}
    h f_j = g \alpha_j \quad\text{ for } j=1,\dots, r \; .
  \end{equation*}
  Note that the generator $h$ satisfies the relation
  $\gcd\paren{h,\alpha_1,\dots,\alpha_r} = 1$.
  Let $\mu_1 :=\gcd\paren{h,\alpha_1}$ and write $h = h_1 \mu_1$.
  One can see that $h_1 | g$.
  Further set $\mu_2 := \gcd\paren{\mu_1, \alpha_2}$ and write $h =
  h_1 h_2 \mu_2$.
  It follows that $h_2 | \frac{g}{h_1}$, thus $h_1 h_2 | g$.
  Proceed inductively to obtain $\mu_j :=
  \gcd\paren{\mu_{j-1},\alpha_j}$ and $h = h_1\dotsm h_j \mu_j$ for
  $j=2,\dots,r$, where $h_1 \dotsm h_j | g$ for all $j = 2, \dots, r$.
  Since $\mu_r$ is a common factor of $\alpha_1, \dots, \alpha_r$ and
  $h$, it is a unit.
  This means that $h | g$.
  As $g$ is an snc generator, $h$ is also an snc generator of the
  annihilator at $x$.
  It is a generator on a neighbourhood of $x$ by the coherence of the
  annihilator.
  This completes the proof.
\end{proof}

\begin{remark}
  The proof of Lemma \ref{lem:principal-annihilator-criterion} does
  not use the fact whether the family
  $\set{\mtidlof[X]{\vphi_L+m\psi}}_{m \in [m_0,1]}$ has any
  jumping numbers other than $m=1$ or not.
  It uses only the coherence of $\mtidlof[X]{\vphi_L+ m_0\psi}$ and
  $\mtidlof[X]{\vphi_L+\psi}$ (which implies the coherence of the
  annihilator), so the conclusion still holds true when $m_0$ in the
  statement is replaced by any $m \in [\rs m_0, 1)$.
\end{remark}

In view of Lemma \ref{lem:principal-annihilator-criterion}, the
following assumption is made to assure that $\rs S$ is an snc
divisor.

\begin{SNCassumption}\label{assum:snc}
  Given $\vphi_L$ and $\psi$ on $X$ described in Section
  \ref{sec:setup} (except that $\vphi_L$ may possibly have more general
  singularities) such that $m=1$ is a jumping number of the family
  $\set{\mtidlof[X]{\vphi_L+m\psi}}_{m\in [0,1]}$, there exists a
  log-resolution $\pi \colon \rs X \to X$ of $(X,\psi)$ such that
  \begin{equation*}
    \mtidlof[\rs X]{\pi^*_\ominus\vphi_L +\pi^*\psi}
    \text{ is principal with an snc generator.}
  \end{equation*} 
\end{SNCassumption}


When $\vphi_L$ has only neat analytic singularities, such assumption
can be achieved via a log-resolution of $(X,\vphi_L,\psi)$, as in
\cite{Chan&Choi_ext-with-lcv-codim-1}*{\S 2.1}.
Since most of the computations and constructions in this paper are
done on $\rs X$, given the Snc assumption \ref{assum:snc}, it is
assumed that the log-resolution $\pi$ provided by the assumption is
indeed $\id$, the identity map.
Moreover, the number $m_0$ provided by Proposition
\ref{prop:log-resoln-on-jump-subvar} is assumed to be $0$ by
``renormalising'' $\vphi_L$ and $\psi$ as in footnote
\ref{fn:renormalising-jumping-no}. 
This assumption is summarised as follows.
\begin{SNCassumptionx} \label{assum:snc-nas}
  Both $\vphi_L$ and $\psi$ have only neat analytic singularities such
  that the polar ideal sheaves $\divP_{\vphi_L}$ and $\divP_{\psi}$ as
  well as the product $\divP_{\vphi_L} \cdot \divP_{\psi}$ are all
  principal with snc generators.
  In particular, one has
  \begin{equation*}
    \mtidlof[X]{\vphi_L+\psi}
    \text{ being principal with an snc generator.}
  \end{equation*}
  Moreover, $m=1$ is the only jumping number of the family
  $\set{\mtidlof{\vphi_L + m {\psi}}}_{m \in [0,1]}$.
\end{SNCassumptionx}
For the sake of convenience, some consequences of the snc assumption
are collected as follows.
\begin{remark}[Consequences of Snc assumptions \ref{assum:snc-nas} or
  \ref{assum:snc} with $\pi
  =\id$] \label{rem:consequences-of-snc-assumptions} \
  \begin{itemize}[labelindent=0em, leftmargin=*]
  \item The function $\psi$ is quasi-psh and its polar set
    $\psi^{-1}(-\infty)$ is an \emph{snc divisor}.
  
  \item According to Lemma \ref{lem:principal-annihilator-criterion},
    the reduced subvariety $S$ is an \emph{snc divisor}.

  \item When $\paren{X,\vphi_L,\psi}$ satisfies the Snc assumption
    \ref{assum:snc-nas}, since $S$ is an snc divisor, one can talk
    about the \emph{lc centres} of $\paren{X,\vphi_L,\psi}$ or $(X,S)$
    as those defined in
    \cite{Kollar_Sing-of-MMP}*{Def.~4.15}.\footnote{
      The concepts of lc centres of $\paren{X,\vphi_L,\psi}$ and those
      of $(X,S)$ are used interchangeably only when
      $\paren{X,\vphi_L,\psi}$ satisfies the Snc assumption
      \ref{assum:snc-nas}.
      See Definition \ref{def:sigma-lc-centres} for the definition of
      lc centres of $\paren{X,\vphi_L,\psi}$ and Example
      \ref{example:compare-alg-ana-adj} for an example on how the two
      concepts differ from each other.
    }
    More explicitly, in this case, an \emph{lc centre of $(X,S)$ of
      codimension $\sigma$ in $X$} is an irreducible component of any
    intersections of $\sigma$ irreducible components of $S$ in $X$.
    Define $\lcc$ to be the \emph{union of all lc centres of $(X,S)$
      of codimension $\sigma$ in $X$}.

  \item There is an effective $\Znum$-divisor $F$ with snc such that
    $\mtidlof{\vphi_L+\psi} = \holo_X\paren{-F}$.  If
    $S =\sum_{i\in I} D_i$, where $D_i$'s are the irreducible
    components of $S$, then $F$ is decomposed into
    \begin{equation*}
      F = G +\underbrace{\sum_{i \in I} \mu_i D_i}_{=: \: S_0}
      +S \; ,
    \end{equation*}
    where $G$ is an effective divisor containing no components of $S$
    and the coefficients $\mu_i \geq 0$ are \emph{integers} ($F$ must
    contain $S$ as the quotient sheaf
    $\frac{\mtidlof{\vphi_L}}{\mtidlof{\vphi_L+\psi}}$ is supported
    precisely on $S$ by item \eqref{item:def-S} in Section
    \ref{sec:setup}).  Let $\sect_{0}$ be a canonical section of $S_0$
    (which is set to $1$ if $S_0$ is the zero divisor), which induces
    a potential on $S_0$ given by
    \begin{equation*}
      \phi_{S_0} :=\log\abs{\sect_0}^2 \; .
    \end{equation*}
    Note that $\sect_0$ locally divides every
    $f \in \mtidlof{\vphi_L}$, because, for every
    $h \in \defidlof{S} =\holo_X\paren{-S}$, one has
    $hf \in \mtidlof{\vphi_L+\psi} \subset \holo_X\paren{-S_0-S}$.
  \end{itemize}
\end{remark}

According to Siu's decomposition for closed positive currents (see,
\cite{Siu_closed-pos-currents}*{\S 6} or \cite{Demailly}*{Ch.~III,
  (8.17)}), given $S =\sum_{i\in I} D_i$ with $D_i$'s being the
irreducible components, one has
\begin{equation*}
  \ibddbar\paren{\vphi_L+\psi}
  =\underbrace{\sum_{i\in I} \lambda_i \currInt{D_i}}_{=: \:
    \currInt{\lambda\cdot S}}
  +\mathfrak R
  =\ibddbar \phi_{\lambda \cdot S} +\mathfrak R \; ,
\end{equation*}
where $\lambda_i = \lelong{\vphi_L+\psi}[D_i] :=\inf_{x\in D_i}
\lelong{\vphi_L+\psi}$, the generic Lelong number of $\vphi_L+\psi$
(or $\ibddbar\paren{\vphi_L+\psi}$) along $D_i$, $\currInt{D_i}$ is
the current of integration along $D_i$ and $\mathfrak R$ is a closed
$(1,1)$-current such that $\mathfrak R \geq \ibddbar\sm\vphi$ for some
smooth potential $\sm\vphi$ ($\mathfrak R \geq 0$ when $\vphi_L+\psi$
is psh) and $\lelong{\mathfrak R}[D_i] = 0$ for each $i\in I$.
Note that the last equality follows from the \textfr{Poincaré}--Lelong
formula $\ibddbar\phi_{\lambda\cdot S} =\currInt{\lambda \cdot S}$
(see Notation \ref{notation:potentials} for the definition of
$\phi_{\lambda\cdot S}$).
($S_0$ would have been denoted by $\mu \cdot S$ if it were not for
notational convenience.)
Therefore,
\begin{equation*}
  \bphi :=\vphi_L+\psi -\phi_{\lambda\cdot S}
\end{equation*}
is a quasi-psh potential (thus locally bounded from above).
Let $\sect$ be a canonical section of $S$ such that $\phi_S
=\log\abs{\sect}^2$.
\emph{The multiplier ideal sheaf $\mtidlof{\vphi_L+m\psi}$ having $m=1$ as a
jumping number implies the following lemma.}
\begin{lemma} \label{lem:Siu-decomposition-S_0-divisor}
  $\lambda \cdot S = S_0 + S \; $.
\end{lemma}
\begin{proof}
  It follows from the definition of multiplier ideal sheaves and the
  equality $\mtidlof{\vphi_L+\psi} = \holo_X\paren{-G-S_0-S}$
  that, for every $f \in \mtidlof{\vphi_L+\psi}$, written
  as $g \sect_0 \sect$ for some $g \in \holo_X\paren{-G}$, one has
  \begin{equation*}
    \abs g^2 e^{\phi_{S_0} +\phi_S -\phi_{\lambda\cdot S}}
    \lesssim \abs{g \sect_0 \sect}^2 e^{-\phi_{\lambda\cdot S} -\bphi}
    =\abs{f}^2 e^{-\vphi_L-\psi}
    \in \Lloc \; ,
  \end{equation*}
  which, as there exists $g \in \holo_X\paren{-G}$ which does not
  vanish on any of the components of $S$, implies that
  \begin{equation*} \tag{$*$} \label{eq:pf:lambda-up-bbd}
    \ceil{\mu_i +1 - \lambda_i} \geq 0 \quad\text{ for all } i\in I \; ,
  \end{equation*}
  where $\ceil\cdot$ is the ceiling function.
  Moreover, as $\lelong{\bphi}[D_i] =\lelong{\mathfrak R}[D_i] =0$ for
  all $i\in I$, $e^{-\bphi}$ is locally integrable at general points
  of $S$ (Skoda's lemma, see \cite{Skoda_Analytic-subsets}*{\S 7} or
  \cite{Demailly_multiplier-ideal-sheaves}*{Lemma (5.6)}),
  which forces the inequalities
  \begin{equation*} \tag{$**$} \label{eq:pf:lambda-low-bbd}
    \mu_i +1 -\lambda_i \leq 0  \quad\text{ for all } i \in I 
  \end{equation*}
  (otherwise, $\mtidlof{\vphi_L+\psi}$ would be strictly containing
  $\holo_X\paren{-S_0-S}$ at some point in $S$).

  Since $1$ is a jumping number of the family
  $\set{\mtidlof{\vphi_L+m\psi}}_{m\in\fieldR_{\geq 0}}$ and $S$ is
  the subvariety resulted from the jump by assumption, it follows that
  every $\lambda_i$ is a \emph{positive integer}.
  Indeed, \eqref{eq:pf:lambda-low-bbd} implies that all $\lambda_i$'s
  are positive.
  If some $\lambda_i$ is not an integer, since $\psi$ has
  logarithmic poles along $S$, there is then some $\eps > 0$ such
  that $\mtidlof{\vphi_L+(1-\eps)\psi}_x =\mtidlof{\vphi_L+\psi}_x$
  for some point $x \in D_i$, a contradiction.

  As a result, it follows from \eqref{eq:pf:lambda-up-bbd} and
  \eqref{eq:pf:lambda-low-bbd} that $\lambda_i =\mu_i +1$ for all $i
  \in I$, which concludes the proof.
\end{proof}

In summary, under the Snc assumption \ref{assum:snc-nas} and the
assumption that $m=1$ is a jumping number of
$\set{\mtidlof{\vphi_L+m\psi}}_{m\in [0,1]}$, there is a
\emph{quasi-psh} potential $\bphi$ on $L \otimes S_0^{-1} \otimes
S^{-1}$ (which is psh when $\vphi_L+\psi$ is) such that
\begin{equation} \label{eq:define-bphi}
  \bphi +\phi_{S_0} +\phi_S = \vphi_L +\psi
  \quad\text{ and }\quad
  \bphi < 0 \quad\text{ locally on } X \; ,
\end{equation}
which is defined for later use.


\section{The new definition of adjoint ideal sheaves}
\label{sec:adjoint-ideal-sheaves}

From this section onward, $\aidlof{\vphi_L}$ denotes the analytic
adjoint ideal sheaf given by Definition
\ref{def:adjoint-ideal-sheaves}. 

\subsection{Local expression of germs in $\aidlof{\vphi_L}$}
\label{sec:local-expression-of-germs}


Suppose the Snc assumption \ref{assum:snc-nas} holds in this section.
Recall the canonical section $\sect_0$ on $S_0$ and the quasi-psh
potential $\bphi$ on $L\otimes S_0^{-1} \otimes S^{-1}$ defined in
Remark \ref{rem:consequences-of-snc-assumptions} and \eqref{eq:define-bphi}.
Let $V \subset X$ be  an open set sitting inside a coordinate chart
with a holomorphic coordinate system $(z_1,\dots, z_n)$, biholomorphic
to a polydisc centred at the origin, such that $\res L_V$ is
trivialised and
\begin{equation*}
  S\cap V =\set{z_1 \dotsm z_{\sigma_V} =0} \quad\text{ for some integer }
  \sigma_V \leq n
  \; .
\end{equation*}
The goal of this section is to describe the sections of
$\aidlof{\vphi_L}(V)$ for any integers $\sigma=1,\dots,\sigma_V$ and
to prove an openness property related to $\aidlof{\vphi_L}$.

Let $\symmgp_m$ be the group of permutations on a set of $m$ elements
for any $m\in\Nnum$ and set 
\begin{equation*}
  \cbn := \symmgp_{\sigma_V} / \paren{\symmgp_{\sigma} \times
  \symmgp_{\sigma_V-\sigma}} \; ,
\end{equation*}
which is the set of choices of $\sigma$ elements in a set of
$\sigma_V$ elements.
Any element $p \in \cbn$ is abused to mean a permutation on the set of
integer $\set{1,\dots,\sigma_V}$ such that, if $p, p' \in \cbn$ and $p
\neq p'$, then $p\paren{\set{1,\dots, \sigma}} \neq
p'\paren{\set{1,\dots, \sigma}}$ (and one also has
$p\paren{\set{\sigma+1,\dots, \sigma_V}} \neq
p'\paren{\set{\sigma+1,\dots, \sigma_V}}$).
Then, the set of subvarieties
\begin{equation*}
  \lcS :=\set{z_{p(1)} = z_{p(2)} = \dotsm = z_{p(\sigma)} = 0}
  \quad\text{ for } p \in \cbn
\end{equation*}
are precisely the set of all of the \emph{lc centres of codimension
  $\sigma$} (or \emph{$\sigma$-lc centres} for short) in $V$.
Therefore,
\begin{equation*}
  \lcc<V> =\lcc \cap V = \bigcup_{p \in\cbn} \lcS \; .
\end{equation*}
When orientations of the lc centres are in concern, assume also that
$p(1) < p(2) < \dotsm < p(\sigma)$ and $p(\sigma+1) < p(\sigma+2) < \dotsm <
p(\sigma_V)$ for each $p \in \cbn$ for convenience.

For any $\sigma=1,\dots,\sigma_V$ fixed, consider any $f \in
\aidlof{\vphi_L}(V)$.
It follows that
\begin{equation*}
  \abs{\frac{f}{\sect_0}}^2 \frac{e^{-\phi_S}}{\logpole}
  \lesssim \frac{\abs{f}^2  \:e^{-\bphi-\phi_{S_0}-\phi_S}}{\logpole}
  =\frac{\abs{f}^2 \:e^{-\vphi_L-\psi}}{\logpole} \in \Lloc(V)
\end{equation*}
for $\eps > 0$.
Recall from Section \ref{sec:snc-assumption} that $\frac{f}{\sect_0}$
is holomorphic.
By considering the Taylor expansion of $\frac{f}{\sect_0}$ about the
origin, there is a \emph{minimal} integer $\sigma_f \in [0, \sigma_V]$ such that
\begin{equation*}
  \frac{f}{\sect_0}
  =\sum_{p \in \cbn[\alert{\sigma_f}]} g_p \: z_{p(\alert{\sigma_f}+1)} \dotsm
  z_{p(\sigma_V)} \; ,
  \quad\text{ where } g_p \in \holo_X(V) \; .
\end{equation*}
(Note that one has $\sigma_f =\sigma_V$ if $\frac{f}{\sect_0}$ does
not vanish identically on $\lcS|\sigma_V|[\id]$, where $\id$ is the
identity permutation, while $\sigma_f =0$ if
$\res{\frac{f}{\sect_0}}_S \equiv 0$.
The integer $\sigma_f$ is the \emph{codimension of mlc of $(V,0,\phi_S)$
with respect to $f$} defined in
\cite{Chan&Choi_ext-with-lcv-codim-1}*{Def.~2.2.5}.) 
Since
\begin{equation} \label{eq:logpole-xlogx-estimate}
  \frac{1}{\abs\psi^{\sigma+\delta}}
  \leq e^\delta \paren{\frac{1+\eps}{\delta e}}^{1+\eps} \frac{1}{\logpole}  
\end{equation}
for any $\delta > 0$ via the $x\log x$-inequality (see, for example,
\cite{Chan_on-L2-ext-with-lc-measures}*{\S 1.4.2}), the fact that
$\abs{\frac{f}{\sect_0}}^2
\frac{e^{-\phi_S}}{\abs\psi^{\sigma+\delta}} \in \Lloc(V)$,
aided by the computation in \cite{Chan&Choi_ext-with-lcv-codim-1}*{\S
  2.2}, implies that
\begin{equation*}
  \sigma_f \leq \sigma \; .
\end{equation*}

In summary, every $f \in \aidlof{\vphi_L}(V)$ can be expressed as a
finite sum
\begin{equation*}
  f = \sum_{p\in\cbn} g_p \:\sect_0 \:z_{p(\sigma+1)} \dotsm z_{p(\sigma_V)}
\end{equation*}
for some $g_p \in \holo_X(V)$.
Note that the functions $g_p$ are not uniquely determined, but the
restricted functions
\begin{equation} \label{eq:uniqueness-of-g_p}
  \res{g_p}_{\lcS} = \res{\frac{f}{\sect_0}}_{\lcS} \cdot
  \frac{1}{z_{p(\sigma+1)} \dotsm z_{p(\sigma_V)}}
\end{equation}
are (provided that the section $\sect_0$ and the coordinate system are
fixed).


\subsection{ An openness property}
\label{sec:openness}


Assume in this section the Snc assumption \ref{assum:snc-nas} in Section
\ref{sec:snc-assumption} and recall the notation for the canonical
section $\sect_0$ on $S_0$ and the quasi-psh potential $\bphi$ defined
there.

Let $\sect$ be the canonical section of $S$ such that $\phi_S =
\log\abs\sect^2$, which therefore belongs to $\defidlof{S}$ locally.
For every $f \in \aidlof{\vphi_L}(V) \subset \mtidlof{\vphi_L}(V)$ on
some open set $V \subset X$, one has
$\sect f \in \mtidlof{\vphi_L+\psi}$ by the definition of $S$ (see
Section \ref{sec:setup}) and $g :=\frac f{\sect_0} \in \holo_X$ (see
Section \ref{sec:snc-assumption}).
It follows that
\begin{equation*}
  \abs g^2 \:e^{-\bphi}
  =\abs{\frac{\sect f}{\sect_0}}^2 \:e^{-\bphi-\phi_S}
  =\abs{\sect f}^2 \:e^{-\bphi-\phi_{S_0}-\phi_S} \in \Lloc(V) \; .
\end{equation*}
After shrinking $V$ if necessary, there exists some number $\lambda >
1$ such that
\begin{equation*}
  \abs g^2 e^{-\lambda \bphi} \in \Lloc(V)
\end{equation*}
by the strong openness property.
From the facts that
\begin{equation*}
  \frac{\abs{f}^2 e^{-\bphi -\phi_{S_0} -\phi_S}}{\logpole}
  \in \Lloc(V) \quad\text{ for all } \eps > 0
\end{equation*}
and that $\bphi$ has only \emph{neat analytic singularities} with snc
singular loci which contain no components of $S$, the computation in
\cite{Chan_on-L2-ext-with-lc-measures}*{Prop.~2.2.1} shows that
\begin{equation*}
  \frac{\abs{f}^2 e^{-\lambda \bphi -\phi_{S_0} -\phi_S}}{\logpole}
  =\frac{\abs g^2 e^{-\lambda \bphi -\phi_S}}{\logpole}
  \in \Lloc(V)
  \quad\text{ for all } \eps > 0 \; .
\end{equation*}

The existence of such $\lambda > 1$ for each section $f$ is required
in the definitions of analytic adjoint ideal sheaves by Guenancia
(\cite{Guenancia}) and Dano Kim (\cite{KimDano-adjIdl}).
Yet their definitions allow the involving quasi-psh function ($\bphi$
in the current notation) to have more general singularities, namely,
$e^{\bphi}$ need only be locally \textde{Hölder} continuous in their
expositions.
The strong openness property of multiplier ideal sheaves of quasi-psh
functions does not imply directly that such $\lambda > 1$ exists for
each section $f$.


\section{The residue  short exact sequences}
\label{sec:residue-short-exact-sequence}



In this section, assume that $\vphi_L$ (as well as $\psi$) has as only
neat analytic singularities and satisfies the Snc assumption
\ref{assum:snc-nas} in Section \ref{sec:snc-assumption}.
Recall that the potentials $\phi_{S_0}$ and $\bphi$ are defined such
that $\bphi$ is a quasi-psh potential on $L \otimes S_0^{-1} \otimes
S^{-1}$ satisfying
\begin{equation*}
  \bphi +\phi_{S_0} +\phi_S =\vphi_L +\psi 
\end{equation*}
with the polar set of $\bphi +\phi_S$ being an snc divisor and
$\bphi^{-1}(-\infty)$ containing no irreducible components of $S$.
Having only neat analytic singularities, this also implies that
$\bphi^{-1}(-\infty)$ contains no lc centres of $(X,S)$ (see
the local expression of $\bphi$ in
\eqref{eq:local-expression-bphi-psi} below).

\subsection{Admissible open sets and results on residue functions in
  \cite{Chan_on-L2-ext-with-lc-measures}}
\label{sec:residue-fcts-with-nas}


Consider an open set $V \subset X$, sitting inside a coordinate chart
with a holomorphic coordinate system $(z_1,\dots, z_n)$, biholomorphic
to a polydisc centred at the origin such that
$L$ is trivialised on $V$,
\begin{equation} \label{eq:local-expression-bphi-psi}
  \begin{gathered}
    S\cap V =\set{z_1 \dotsm z_{\sigma_V} =0} \quad\text{ for some
      integer }
    \sigma_V \leq n \; , \\
    \res{\phi_S}_V =\sum_{j=1}^{\sigma_V} \log\abs{z_j}^2 \; , \qquad
    \res\bphi_V = \smashoperator[r]{\sum_{k=\sigma_V+1}^n} b_k
    \log\abs{z_k}^2 +\beta
    \quad\text{ and} \\
    \res\psi_V =\underbrace{\sum_{j = 1}^{\sigma_V} \nu_j
      \log\abs{z_j}^2}_{=: \: \phi_{\nu\cdot S}} ~
    +\underbrace{\smashoperator[r]{\sum_{k=\sigma_V+1}^n} a_k
      \log\abs{z_k}^2}_{=: \: \phi_{T}} ~ +\ \alpha \; ,
  \end{gathered}
\end{equation} 
where, after shrinking $V$ if necessary, 
\begin{itemize}
\item $\sup_V \log\abs{z_j}^2 < 0$ for $j=1,\dots, n \;$,
  
\item $\alpha$ and $\beta$ are smooth functions such that
  $\sup_V\alpha \leq -1$ and $\sup_V\beta \leq 0$,

\item $\nu_j$'s are constants such that $\nu_j > 0$ for $j=1,\dots,
  \sigma_V \;$,
  
\item $a_k$'s and $b_k$'s are constants such that $a_k, b_k \geq 0$
  for $k=\sigma_V+1,\dots,n \;$, and
  
\item $\sup_V r_j \fdiff{r_j} \psi =2\nu_j +\sup_V r_j \fdiff{r_j} \alpha >
  0$ for $j=1,\dots, \sigma_V$, where $r_j = \abs{z_j}$ is the radial
  component of the polar coordinates.
\end{itemize} 
Call such open set $V$ as an \emph{admissible open set with respect
  to the data $(\vphi_L,\psi)$ in the coordinate system
  $(z_1,\dots,z_n)$} for the ease of reference.
It is simply said to be \emph{admissible} if the data $(\vphi_L,\psi)$
are understood and the above criteria are satisfied in some
holomorphic coordinate system.
Note that such admissible open set is the kind of open sets on which
the computations in \cite{Chan_on-L2-ext-with-lc-measures}*{\S 2.2}
are valid.
The family of all such admissible open sets forms a basis of the
topology of $X$.



For any $f \in \aidlof{\vphi_L}(V)$, $f$ can be decomposed into a
finite sum (see Section \ref{sec:local-expression-of-germs})
\begin{equation} \label{eq:decomposition-of-f}
  f = \sum_{p\in \cbn} g_p \:\sect_0 \: z_{p(\sigma+1)} \dotsm z_{p(\sigma_V)}
  \; .
\end{equation}
Note that, while the holomorphic functions $g_p$ are not uniquely
determined, $\res{g_p}_{\lcS}$ is uniquely determined (under
the current coordinate system) by $f$ and
$\sect_0$ for each $p\in\cbn$.
An integrability property of the $g_p$'s is collected as follows.
\begin{lemma} \label{lem:g_p-regularity}
  Given the Snc assumption \ref{assum:snc-nas}, every $f \in
  \mtidlof{\vphi_L}$ admits a decomposition of the form in
  \eqref{eq:decomposition-of-f} for some integer $\sigma \geq 0$ on an
  admissible open set $V$ such that the $g_p$'s satisfy
  \begin{equation*}
    g_p \in \mtidlof{\bphi}
    \quad\text{ and }\quad
    g_p \sect_0 \in \mtidlof{\bphi +\phi_{S_0}}
    = \mtidlof{\vphi_L}
    \quad\text{ on } V \text{ for every } p \in \cbn \; ,
  \end{equation*}
  and thus
  \begin{equation*}
    f \in \mtidlof{\bphi +\phi_{S_0}} \cdot \defidlof{\lcc[\sigma+1]}
    =
    \mtidlof{\vphi_L} \cdot \defidlof{\lcc[\sigma+1]}
    \quad\text{ on } V \; ,
  \end{equation*}
  where $\defidlof{\lcc[\sigma+1]}$ is the defining ideal sheaf of
  $\lcc[\sigma+1]$, which is generated as an ideal by the set
  $\setd{z_{p(\sigma+1)} \dotsm z_{p(\sigma_V)}}{p \in \cbn}$ on $V$.
\end{lemma}

\begin{proof}
  The discussion in Section \ref{sec:openness} indeed shows that, for
  every $f \in \mtidlof{\vphi_L}$, one has
  \begin{equation*}
    \abs{\frac{f}{\sect_0}}^2 e^{-\bphi} \in \Lloc(V) \; .
  \end{equation*}
  Thanks to the Snc assumption \ref{assum:snc-nas} on $\vphi_L+\psi$,
  one can see from Fubini's theorem and the above integrability that
  $\frac{f}{\sect_0}$ is divisible by some monomial
  $\prod_{k=\sigma_V+1}^n z_{k}^{m_k}$ such that $m_k -b_k >-1$ for
  $k=\sigma_V+1, \dots, n$ (where $b_k$'s are the coefficients in the
  local expression of $\res\bphi_V$ in
  \eqref{eq:local-expression-bphi-psi}).
  Since $\frac{f}{\prod_{k=\sigma_V+1}^n z_{k}^{m_k}}$ admits a
  decomposition of the form in \eqref{eq:decomposition-of-f} for some
  integer $\sigma \in [0, \sigma_V]$,
  one can assume that $f$ admits the decomposition
  \eqref{eq:decomposition-of-f} (with the same $\sigma$) such that
  $g_p$ is divisible by $\prod_{k=\sigma_V+1}^n z_{k}^{m_k}$ for each
  $p\in\cbn$, and thus 
  \begin{equation*}
    \abs{g_p}^2 e^{-\bphi} \in\Lloc(V) \; .
  \end{equation*}
  This also implies that $g_p \sect_0 \in \mtidlof{\bphi +\phi_{S_0}}$
  and $f \in \mtidlof{\bphi+\phi_{S_0}} \cdot
  \defidlof{\lcc[\sigma+1]}$, and therefore $\mtidlof{\vphi_L} \subset
  \mtidlof{\bphi +\phi_{S_0}}$.
  Note that every germ in $\mtidlof{\bphi +\phi_{S_0}}$ can be
  decomposed into the form $g \sect_0$ such that $g \in \mtidlof{\bphi}$.
  Since 
  \begin{equation*}
    \vphi_L = \bphi +\phi_{S_0} +\phi_S -\psi
    =\phi_{S_0} +\phi_S-\phi_{\nu\cdot S}
    +\underbrace{\smashoperator[r]{\sum_{k=\sigma_V+1}^n} \paren{b_k -a_k}
      \log\abs{z_k}^2
      +\beta}_{= \:\bphi -\phi_T}
    -\alpha \quad\text{ on }V 
  \end{equation*}
  and $\sect_0$ is locally $L^2$ with respect to
  $e^{-\phi_{S_0}-\phi_S+\phi_{\nu\cdot S}}$ while $g$ is locally
  $L^2$ with respect to $e^{-\bphi +\phi_T}$,
  one sees, in view of Fubini's theorem, that
  \begin{equation*}
    \abs{g \sect_0}^2 e^{-\vphi_L} \in \Lloc(V) \quad\text{ and thus }\quad
    \mtidlof{\bphi +\phi_{S_0}} \subset \mtidlof{\vphi_L} \; .
  \end{equation*}
  This completes the proof.
\end{proof}

Lemma \ref{lem:g_p-regularity} also implies that $\aidlof{\vphi_L} \subset
\mtidlof{\vphi_L} \cdot \defidlof{\lcc[\sigma+1]}$ under the Snc
assumption \ref{assum:snc-nas}.


Recall that the \emph{residue function} $\eps \mapsto
\RTF[G](\eps)[V,\sigma]$ defined in
\cite{Chan_on-L2-ext-with-lc-measures} is given by
\begin{equation*}
  \RTF[G](\eps)[V,\sigma]
  := \RTF[G](\eps)[\sigma]
  := \eps \int_V \frac{G \:e^{-\vphi_L-\psi}
    \dvol_{V}}{\logpole}
  \quad\text{ for any } G \in \smooth_X(V) \text{ and } \eps > 0 \; .
\end{equation*}
The computations in
\cite{Chan_on-L2-ext-with-lc-measures}*{Prop.~2.2.1, Thm.~2.3.1 and
  Cor.~2.3.3} assure the following statements.
\begin{thm} \label{thm:results-on-residue-fct}
  Under the Snc assumption \ref{assum:snc-nas}, for every $f \in
  \mtidlof{\vphi_L}(V'')$ and for every admissible open set $V \Subset V''$
  described above, there exists a unique integer $\sigma_f \in
  [0,\sigma_V]$ such that
  \begin{enumerate}[labelindent=0em, leftmargin=*, itemsep=1ex]
  \item \label{item:thm:res-fct}
    one has $f \in \mtidlof{\bphi +\phi_{S_0}} \cdot
    \defidlof{\lcc[\alert{\sigma_f}+1]}
    =\mtidlof{\vphi_L} \cdot
    \defidlof{\lcc[\alert{\sigma_f}+1]}$ on $V$ and, for every integer
    $\sigma$ and real number $\eps > 0$, 
    \begin{equation*}
      \RTF|f|(\eps)[V,\sigma] \;\;
      \begin{dcases}
        < +\infty & \text{when } \sigma \geq \sigma_f \; , \\
        =+\infty & \text{when } \sigma < \sigma_f \; ,
      \end{dcases} 
    \end{equation*}
    and it therefore follows that
    \begin{equation*}
      \aidlof{\vphi_L}
      =\mtidlof{\vphi_L} \cdot \defidlof{\lcc[\sigma+1]}
      \quad\text{ for all integers } \sigma \geq 0 \; ,
    \end{equation*}
    which implies, in particular, that
    \begin{alignat*}{2}
      \aidlof|0|{\vphi_L}
      &=\mtidlof{\vphi_L}\cdot \defidlof{S}
        =\mtidlof{\vphi_L+\psi}
      & \quad &\text{and} \\
      \aidlof{\vphi_L}
      &=\mtidlof{\vphi_L} \cdot
      \defidlof{\emptyset} =\mtidlof{\vphi_L}
      &&\text{for all } \sigma \geq \sigma_{\mlc} \; ,
    \end{alignat*}
    where $\sigma_{\mlc} $ is the codimension of the minimal lc
    centres (mlc) of $(X,S)$;

  \item \label{item:thm:res-norm}
    for any smooth cut-off function $\rho \colon V' \to [0,1]$
    which is compactly supported in an admissible open set $V'$
    in the same coordinate system on $V$ such that $V
    \Subset V' \Subset V''$ and that $\rho \equiv 1$ on $V$,
    the function $\eps \mapsto \RTF[\rho]|f|(\eps)[V',\sigma]$
    can be analytically continued across $\eps =0$ when $\sigma \geq
    \sigma_f$ and one has the \emph{residue (squared) norm} given by
    \begin{equation*}
      \RTF[\rho]|f|(0)[V',\sigma]
      =\lim_{\eps \tendsto 0^+} \RTF[\rho]|f|(\eps)[V',\sigma]
      =
      \begin{dcases}
        \int_{\lcc<V'>} \rho \abs f^2 \lcV
        <+\infty
        & \text{when } \sigma \geq \sigma_f \; , \\
        +\infty & \text{when } \sigma < \sigma_f \; ,
      \end{dcases} 
    \end{equation*}
    in which, when $f$ admits the decomposition in
    \eqref{eq:decomposition-of-f}, the integral with respect to the
    $\sigma$-lc-measure is given by
    \begin{equation*}
      \int_{\lcc<V'>} \rho \abs f^2 \lcV
      =\sum_{p \in\cbn/\sigma_{V'}/} \frac{\pi^\sigma}{(\sigma-1)! \:
        \vect\nu_{p}}
      \int_{\lcS} \rho \abs{g_p}^2 \:e^{-\bphi}
      \dvol_{\lcS} \; ,
    \end{equation*}
    where $\vect\nu_p :=\prod_{j=1}^\sigma \nu_{p(j)}$ (product of the
    coefficients in the local expression of $\psi$ in
    \eqref{eq:local-expression-bphi-psi}, or, equivalently, product of
    the generic Lelong numbers $\lelong{\psi}[D_i]$ of $\psi$ along the
    irreducible components $D_i$ of $S$ which contain $\lcS$), and it
    thus follows that $\RTF[\rho]|f|(0)[V',\sigma] = 0$ when $\sigma >
    \sigma_f$.
  \end{enumerate} 
\end{thm}

\begin{proof}
  Take $\sigma_f \in [0,\sigma_V]$ to be the minimal integer such that
  the conclusion of Lemma \ref{lem:g_p-regularity}, namely, $f \in
  \mtidlof{\bphi +\phi_{S_0}} \cdot \defidlof{\lcc[\sigma_f+1]} 
  =\mtidlof{\vphi_L} \cdot \defidlof{\lcc[\sigma_f+1]}$, holds true on
  $V$ (then $g_p$ in the decomposition is non-trivial on
  $\lcS|\sigma_f|$ for some $p \in \cbn[\sigma_f]$).
  The inequality \eqref{eq:logpole-xlogx-estimate} and the computation
  in \cite{Chan&Choi_ext-with-lcv-codim-1}*{\S 2.2} implies that
  $\RTF|f|(\eps)[V,\sigma] =+\infty$ for all $\eps > 0$ when $\sigma <
  \sigma_f$.
  On the other hand, given a cut-off function $\rho$ described in
  \eqref{item:thm:res-norm}, the computation in
  \cite{Chan_on-L2-ext-with-lc-measures}*{Prop.~2.2.1} shows that
  \begin{equation*}
    \RTF|f|(\eps)[V,\sigma] \leq \RTF[\rho]|f|(\eps)[V',\sigma]
    < +\infty
    \quad\text{ for all } \eps >0 \text{ when } \sigma \geq \sigma_f \; .
  \end{equation*}
  Therefore, $f \in \aidlof{\vphi_L}$ on $V$ for $\sigma \geq \sigma_f$.
  This gives \eqref{item:thm:res-fct}.

  The claims in \eqref{item:thm:res-norm} follow from the computations
  in \cite{Chan_on-L2-ext-with-lc-measures}*{Thm.~2.3.1 and
    Cor.~2.3.3}.
  The expression of the integral under the $\sigma$-lc-measure follows
  from the computation in \cite{Chan&Choi_ext-with-lcv-codim-1}*{\S
    2.2}. 
\end{proof}

\begin{remark}
  The integer $\sigma_f$ is the \emph{codimension of the mlc of
    $(V,\vphi_L,\psi)$ (or of $(V,S)$) with respect to $f$} defined in
  \cite{Chan&Choi_ext-with-lcv-codim-1}*{Def.~2.2.5}.
  The proof of Theorem \ref{thm:results-on-residue-fct} also shows
  that it is the same as the codimension of the mlc of $(V,0,\phi_S)$ with
  respect to $f$ stated in Section \ref{sec:local-expression-of-germs}.
\end{remark}

\begin{remark} \label{rem:residue-fct-result-summary}
  For the convenience of computations in practice, the essence of the
  convergence result for the integral $\RTF|f|(\eps)[V,\sigma]$ in
  Theorem \ref{thm:results-on-residue-fct} \eqref{item:thm:res-fct} is
  stated in terms of local coordinates.
  Suppose $\psi$ is given as in \eqref{eq:local-expression-bphi-psi}
  on $V$, in which $\nu_j > 0$ for $j = 1, \dots, \sigma_V$.
  For the purpose of illustration, assume that $\sigma_V < n$ while
  \alert{$a_{n-1} > 0$ and $a_n=0$} ($a_k$ is the coefficient of
  $\log\abs{z_k}^2$ in $\psi$).
  The computation in
  \cite{Chan_on-L2-ext-with-lc-measures}*{Prop.~2.2.1} as well as
  \cite{Chan&Choi_ext-with-lcv-codim-1}*{Prop.~2.2.1} (hence Theorem
  \ref{thm:results-on-residue-fct} \eqref{item:thm:res-fct}) implies
  that, for any $\eps > 0$ and $\sigma_f \leq \sigma_V$ and $\sigma_f
  < n-1$, one has
  \begin{equation*}
    \int_V \frac{\dvol_V}{\abs{z_1\dotsm z_{\sigma_f}}^2
      \abs{z_{n-1}}^{2\ell_{n-1}} \abs{z_n}^{2\ell_n} \logpole}
    \;
    \begin{dcases}
      < +\infty &
        \text{when } \sigma > \sigma_f  ,\: \ell_{n-1} \leq 1
        , \: \ell_n < 1 \; ,
      \\
      < +\infty & \text{when } \sigma = \sigma_f \text{ and }
      \ell_{n-1}, \ell_n < 1 \; , \\
      = +\infty & \text{when } \sigma =\sigma_f \text{ and }
      \ell_{n-1} \geq 1 \; , \\
      =+\infty &  \text{when } \sigma < \sigma_f  \; , \\
      = +\infty & \text{when } \ell_{n-1} > 1 \text{ or } \ell_n \geq
      1 \; .
    \end{dcases} 
  \end{equation*} 
\end{remark}

\begin{remark} \label{rem:residue-fct-result-jumping-no}
  When the family $\set{\mtidlof{\vphi_L+m\psi}}_{m \in [0,1]}$ has
  another jumping number $m_0 \in (0,1)$ besides $m=1$ such that
  $\mtidlof{\vphi_L} \supsetneq \mtidlof{\vphi_L+m_0\psi}
  =\mtidlof{\vphi_L+m\psi}$ for all $m \in [m_0, 1)$, all the
  multiplier ideal sheaves $\mtidlof{\vphi_L}$ appearing in the
  statement of Theorem \ref{thm:results-on-residue-fct} have to be
  replaced by $\mtidlof{\vphi_L+m_0\psi}$.
  In particular, Theorem \ref{thm:results-on-residue-fct}
  \eqref{item:thm:res-fct} should conclude that $\aidlof{\vphi_L}
  =\mtidlof{\vphi_L+m_0\psi} \cdot \defidlof{\lcc[\sigma+1]}$ for all
  integers $\sigma \geq 0$.
  Indeed, assuming the existence of $m_0 \in (0,1)$, one has
  $\RTF|f|(\eps)[V,\sigma] =+\infty$ for any $f \in
  \mtidlof{\vphi_L}_x \setminus
  \mtidlof{\vphi_L+m_0\psi}_x$, $\sigma \geq 0$ and $\eps >
  0$ and for any admissible open neighbourhood $V$ of $x$ in $X$ on
  which $f$ is defined.
\end{remark}

\begin{remark} \label{rem:RTF-thm-on-smooth-sections}
  The computation of $\RTF[\rho]|f|(\eps)[V,\sigma]$ invokes only
  Fubini's theorem and integration by parts on each variables in the
  admissible open set $V$ (see
  \cite{Chan_on-L2-ext-with-lc-measures}*{Prop.~2.2.1, Thm.~2.3.1 and
    Cor.~2.3.3}, also
  cf.~\cite{Chan&Choi_ext-with-lcv-codim-1}*{Prop.~2.2.1}).
  The holomorphicity of $f$ plays little role in the computation.
  Indeed, it suffices to have $f \in \mtidlof{\vphi_L} \cdot
  \smooth_X$ such that $f$ can be written in the form
  \eqref{eq:decomposition-of-f} via Taylor's theorem (in which case
  $g_p$'s are smooth local functions on $V$) and the computation of
  $\RTF[\rho]|f|(\eps)[V,\sigma]$ is still valid.
  It turns out that both Lemma \ref{lem:g_p-regularity} and Theorem
  \ref{thm:results-on-residue-fct} remain true when the ideal sheaves
  $\mtidlof{\vphi_L}$ and $\aidlof{\vphi_L}$ are extended to
  $\mtidlof{\vphi_L} \cdot \smooth_X$ and $\aidlof{\vphi_L} \cdot
  \smooth_X$ respectively (via the injection $\holo_X \hookrightarrow
  \smooth_X$) and the involving functions $f$ and $g_p$'s are only
  smooth.
\end{remark}

Theorem \ref{thm:results-on-residue-fct} \eqref{item:thm:res-fct}
shows, in particular, that, for any admissible open set $V \Subset X$,
the maps $f \mapsto \RTF|f|(\eps)[V,\sigma]$ for $\eps > 0$ define a
family of $L^2$ norms on $\aidlof{\vphi_L}\paren{\cl V}$ and that
\begin{equation*}
  \aidlof{\vphi_L}\paren{\cl V}
  =\setd{f \in \mtidlof{\vphi_L}\paren{\cl V}}{ \exists~\eps > 0
    \text{ such that }
    \RTF|f|(\eps)[V,\sigma] < +\infty } \; .
\end{equation*} 

For the sake of notational convenience, for any $f \in \smooth_X
\cdot\! \aidlof{\vphi_L}(\cl V)$ on some open set $V \Subset X$ such
that $f$ \emph{does not have compact support in $V$}, define
\begin{equation}
  \label{eq:lc-measure-for-non-cpt-supp}
  \RTF|f|(0)[V,\sigma]
  :=\int_{\lcc} \abs f^2 \lcV
  :=\lim_{\rho \descendsto \charfct_{\cl V}}
  \RTF[\rho]|f|(0)[V',\sigma] \; ,
\end{equation}
where $\rho \colon V' \to [0,1]$ is a compactly supported smooth
cut-off function on an open set $V' \Supset V$ with $\res\rho_V \equiv
1$ such that $f$ is defined on $V'$ and $\lim_{\rho \descendsto
  \charfct_{\cl V}}$ indicates taking limit by considering a sequence
of the functions $\rho$ which descends pointwisely to the
characteristic function $\charfct_{\cl V}$ of $\cl V$ on $X$.
Then, Theorem \ref{thm:results-on-residue-fct}
\eqref{item:thm:res-norm} implies, in particular, that, for any
admissible open set $V \Subset X$,
\begin{equation*}
  \RTF|f|(0)[V,\sigma]
  =\sum_{p \in\cbn/\sigma_{V}/} \frac{\pi^\sigma}{(\sigma-1)! \:
    \vect\nu_{p}}
  \int_{\lcS} \abs{g_p}^2 \:e^{-\bphi}
  \dvol_{\lcS} 
\end{equation*}
given the decomposition \eqref{eq:decomposition-of-f} of $f$, and thus
\begin{equation*}
  \aidlof|\sigma-1|{\vphi_L}\paren{\cl V}
  =\setd{f \in \aidlof{\vphi_L}\paren{\cl V}}{ \RTF|f|(0)[V,\sigma] =
    0 } \; .
\end{equation*}




\subsection{The residue morphism}
\label{sec:residue-map}


For a given integer $\sigma \geq 1$, let $\nu \colon \rs{\lcc} \to
\lcc$ be the normalisation of $\lcc$ (thus $\rs{\lcc}$ is the disjoint
union of all $\sigma$-lc centres of $(X,S)$) and $\iota
\colon \lcc \hookrightarrow X$ the natural inclusion.
By a slight abuse of notation, write temporarily that $\rs{\lcc}
=\bigsqcup_{q \in \Iset./X/} \lcS[q]$, where $\set{\lcS[q]}_{q\in \Iset./X/}$ is
the set of all $\sigma$-lc centres of $(X,S)$ indexed by
$q\in \Iset./X/$.
As $S$ is an snc divisor by assumption, in view of the adjunction
formula, there is a divisor $\Diff_{q}S$ on $\lcS[q]$ (the
\emph{general different} of $S$ on $\lcS[q]$; see
\cite{Kollar_Sing-of-MMP}*{\S 4.2}) such that (in terms of their
associated line bundles)
\begin{equation*}
  K_{\lcS[q]} \otimes \res{S^{-1}}_{\lcS[q]} 
  =\res{K_X}_{\lcS[q]} \otimes \paren{\Diff_{q} S}^{-1} \; .
\end{equation*}
Indeed, on any admissible open set $V$ such that $\lcS[q] \cap V =
\lcS$ for some $p \in \cbn$ (following the notation in Section
\ref{sec:local-expression-of-germs}), the line bundle
$\Diff_{q} S$ is isomorphic to the line bundle associated to the
divisor $\res{\sum_{k=\sigma+1}^{\sigma_V} \set{z_{p(k)}=0}}_{\lcS}$.
Again by an abuse of notation, let $\paren{\Diff_{q} S}^{-1}$ to mean
the \emph{extension to $\rs{\lcc}$ by zero} of its associated invertible
sheaf.

Now set
\begin{align*}
  \residlof{\bphi}
  :=\residlof<X>{\bphi}
  &:=S_0^{-1} \otimes \iota_*\nu_*\paren{\paren{\bigoplus_{q \in \Iset./X/}
      \paren{\Diff_{q} S}^{-1}} \otimes
    \mtidlof[\rs{\lcc}]{\nu^*\iota^*\bphi}} \\
  &=\bigoplus_{q \in \Iset./X/}
    \iota_*\nu_*\paren{
    \res{S_0^{-1}}_{\lcS[q]} \otimes 
    \paren{\Diff_{q} S}^{-1} \otimes
    \mtidlof[\lcS[q]]{\nu^*\iota^*\bphi}
    } \; ,
\end{align*}
which is a coherent sheaf on $X$ supported on $\lcc$ (note that the
supports of $\paren{\Diff_{q} S}^{-1}$ for $q \in \Iset./X/$ are
mutually disjoint in $\rs{\lcc}$).
Let $V$ be an admissible open set and let $\set{\lcS}_{p \in
\cbn}$ be the set of $\sigma$-lc centres of $(V,S \cap
V)$ as described in Section \ref{sec:local-expression-of-germs}.
One then has
\begin{equation*}
  \nu^{-1}\paren{\lcc \cap V} = \bigsqcup_{p\in\cbn} \lcS
\end{equation*}
and, by writing $\Diff_{p} S := \Diff_{q}S$ when $\lcS[q] \cap V =
\lcS$,
\begin{equation*}
  \residlof{\bphi}(V)
  =\prod_{p\in\cbn} \res{S_0^{-1}}_{\lcS} \otimes \paren{\Diff_p S}^{-1} \otimes
  \mtidlof[\lcS]{\bphi} \paren{\lcS} \; . 
\end{equation*}

It therefore follows from Theorem \ref{thm:results-on-residue-fct}
that the residue map
\begin{equation*}
  \xymatrix@C=1cm@R=0.2cm{
    {\aidlof{\vphi_L}(V)} \ar[r]^-{\Res} \ar@{}[d]|*[left]{\in}
    & {\residlof{\bphi}(V)} \ar@{}[d]|*[left]{\in} \\
    {\quad f \quad} \ar@{|->}[r] & 
    {\paren{\res{g_p}_{\lcS}}_{\mathrlap{p\in\cbn}}} 
  }
\end{equation*}
is well-defined on every admissible open set $V$, where $f$ and the
$g_p$'s are related by the decomposition
\eqref{eq:decomposition-of-f} and $\res{g_p}_{\lcS}$ is viewed as a
section of $S_0^{-1} \otimes \paren{\Diff_p S}^{-1}$ (see
\eqref{eq:uniqueness-of-g_p}).
It can also be seen that the map is independent of the choice of
the coordinate system and induces a morphism of sheaves (which depends
on the choice of the canonical sections $\sect_0$ and $z_{p(\sigma+1)}
\dotsm z_{p(\sigma_V)}$ of $S_0$ and of $\Diff_p S$ respectively).

Sometimes it may be convenient to describe the residue morphism in
terms of holomorphic $n$-forms.
With the adjunction formula in mind, one has
\begin{equation*}
  K_X \otimes \residlof{\bphi}(V)
  =\prod_{p\in\cbn} K_{\lcS} \otimes \parres{S_0^{-1} \otimes
    S^{-1}}_{\lcS} 
  \otimes \mtidlof[\lcS]{\bphi}\paren{\lcS} \; .
\end{equation*}
Let $\sect$ be a canonical section of $S$ such that $\res\sect_V
=z_1 \dotsm z_{\sigma_V}$.
Identifying $f$ with the $(n,0)$-form $f \:dz_1\wedge \dotsm \wedge
dz_n$ on $V$ and each $g_p$ 
with the 
$(n-\sigma, 0)$-form 
$g_p \:\sgn{p} \:dz_{p(\sigma+1)} \wedge
\dotsm \wedge dz_{p(\sigma_V)} \wedge dz_{\sigma_V+1} \wedge \dotsm
\wedge dz_n$ on $V$ (where $\sgn p$ is the sign of the permutation
representing $p$ described in Section
\ref{sec:local-expression-of-germs}), also writing $u_p := dz_{p(1)}
\wedge \dotsm \wedge dz_{p(\sigma)}$, one then has
\begin{equation*}
  f =\sum_{p\in\cbn} u_p \wedge g_p \:\sect_0
  \:z_{p(\sigma+1)} \dotsm z_{p(\sigma_V)}
  =\sum_{p\in\cbn} \frac{u_p}{z_{p(1)} \dotsm z_{p(\sigma)}}
  \wedge g_p \:\sect_0 \sect
\end{equation*}
(cf.~\eqref{eq:decomposition-of-f}).
With a suitable choice of orientation on the normal bundle of $\lcS$
in $X$ which determines the sign on the \textfr{Poincaré} residue map
$\PRes$ (see \cite{Kollar_Sing-of-MMP}*{\S 4.18}), it follows that
\begin{equation} \label{eq:g_p-from-Poincare-residue}
  \res{g_p}_{\lcS}
  = \PRes(\frac{f}{\sect_0 \sect}) 
\end{equation}
for all $p \in \cbn$ (cf.~\eqref{eq:uniqueness-of-g_p}).
The residue morphism
is then given by
\begin{equation*}
  \xymatrix@C=1cm@R=0.2cm{
    {K_X \otimes \aidlof{\vphi_L}(V)} \ar[r]^-{\Res} \ar@{}[d]|*[left]{\in}
    & {K_X \otimes \residlof{\bphi}(V)}
    \ar@{}[d]|*[left]{\in} \\ 
    {\quad f \quad} \ar@{|->}[r] & 
    {\quad \paren{\res{g_p}_{\lcS}}_{\mathrlap{p\in\cbn}} \quad} 
  } 
\end{equation*}
on each admissible open set $V$, which depends on the choice of
$\sect_0$ and $\sect$.
It is well-defined as assured by Theorem
\ref{thm:results-on-residue-fct}. 


\subsection{A local $L^2$ extension theorem}
\label{sec:exact-sequence}


In the versions of analytic adjoint ideal sheaves of Guenancia
(\cite{Guenancia}) and Dano Kim (\cite{KimDano-adjIdl}), surjectivity
of the residue morphism is obtained via an application of the
Ohsawa--Takegoshi--Manivel extension theorem (see \cite{Manivel} and
\cite{Demailly_on_OTM-extension}), which is not sufficient for proving
the surjectivity of the residue morphism in the present setting. 
In this section, a local extension which belongs to $\aidlof{\vphi_L}$
(i.e.~a preimage of a germ in $\residlof{\bphi}$ under $\Res$) is
constructed directly instead.

Given an admissible open set $V \subset X$, for any fixed integer
$\sigma \in [1, \sigma_V]$ and any $\sigma$-lc 
centre $\lcS \subset \lcc<V>$ where $p \in\cbn$ (see Section
\ref{sec:local-expression-of-germs}), 
consider the maps given in the diagram
\begin{equation*}
  \xymatrix{
    {\lcS \isom \set 0 \times \lcS\; } \ar@{^{(}->}[r]^-{\iota_p} &
    {U_p \times \lcS }  \ar[d]^-{\pr_p} 
    & **[l]{\quad = U_p \times W_p \isom V}  \\ 
    & {\lcS} &
  }
\end{equation*}
where $U_p \subset \fieldC^\sigma$ is an $\sigma$-disc centred at the
origin with coordinates given by $z_{p(1)}, \dots, z_{p(\sigma)}$,
and the maps $\iota_p$ and $\pr_p$ are, respectively, the natural
inclusion and the projection to the second factor.
$\lcS$ in the product $U_p \times \lcS$ is written as $W_p$ so that
the symbol $\lcS$ is reserved for the subvariety $\set{0} \times W_p$.
Set
\begin{equation*}
  \bphip := \pr_p^*\iota_p^*\bphi
  \quad\;\;\text{(and } \bphip :=\bphi \text{ when } \sigma =0 \text{)}
\end{equation*}
and notice that, from the smoothness of $\beta$ in the local
expression of $\res\bphi_V$ in \eqref{eq:local-expression-bphi-psi}
and that $\bphi^{-1}(-\infty)$ is in snc with $S$ but contains no
components of $S$, one has, for any integer $\sigma \geq 0$,
\begin{equation} \label{eq:bdds-on-bphi}
  \bphi \:\sim_{\tlog} \bphip
  \quad\text{ on }V \;\;\text{ for all } p \in\cbn \; .
\end{equation}

With the above preparation, the following key property of
adjoint ideal sheaves can be proved.
\begin{thm} \label{thm:short-exact-seq}
  Given the Snc assumption \ref{assum:snc-nas}, the sequence
  \begin{equation*}
    \xymatrix{
      0 \ar[r]
      & {\aidlof|\sigma-1|{\vphi_L}} \ar[r]
      & {\aidlof{\vphi_L}} \ar[r]^-{\Res}
      & {\residlof{\bphi}} \ar[r]
      & 0
    } \; ,
  \end{equation*}
  where the map between $\aidlof|\sigma-1|{\vphi_L}$ and
  $\aidlof{\vphi_L}$ is the natural inclusion, is exact.
\end{thm}

\begin{proof}
  To see that $\aidlof|\sigma-1|{\vphi_L}$ is the kernel
  of $\Res$, note that, according to Theorem
  \ref{thm:results-on-residue-fct}, for any $f \in 
  \aidlof{\vphi_L}$, one has $\RTF[\rho]|f|(0)[\sigma] =0$ (for some
  smooth cut-off function $\rho$ which is $\equiv 1$ on some
  admissible open set $V$ as in Theorem
  \ref{thm:results-on-residue-fct} \eqref{item:thm:res-norm}) if and
  only if $\res{g_p}_{\lcS} \equiv 0$ on $\lcS \subset \lcc<V>$ for
  all $p \in \cbn$, which is equivalent to $f \in 
  \mtidlof{\vphi_L} \cdot \defidlof{\lcc}
  =\aidlof|\sigma-1|{\vphi_L}$ according to
  \eqref{eq:decomposition-of-f} and Lemma \ref{lem:g_p-regularity}.

  For the surjectivity of the residue morphism $\Res$, first note that
  $\residlof{\bphi}_x = 0$ for any $x \in X \setminus \lcc$ and thus
  $\Res_x$ is automatically surjective for such $x$.
  Next, consider any $x \in \lcc \subset X$ and any germ $\paren{g_p}_{p\in\cbn}
  \in \residlof{\bphi}_x$ such that $V$ is an admissible open set in
  $X$ centred at $x$ and
  \begin{equation*}
    \sum_{p \in\cbn} 
    \int_{\lcS} \abs{g_p}^2 \:e^{-\bphi} \dvol_{\lcS}  < +\infty
    \quad\text{ (where $\lcS \subset \lcc<V> \subset V$)} \; .
  \end{equation*}
  Note that
  \begin{equation*}
    \abs{\psi_p} := \abs{\sum_{j=1}^\sigma
      \nu_{p(j)}\log\abs{z_{p(j)}}^2} +1
    \leq \abs{\psi} \quad\text{ on }V
  \end{equation*}
  as seen from \eqref{eq:local-expression-bphi-psi}.
  Abusing the notation $g_p$ to mean also its pullback $\pr_p^*g_p$ to
  $V$,
  it then follows that, for each $p\in\cbn$ and for
  any $\eps > 0$, 
  \begin{align*}
    &~\int_V \frac{
      \abs{g_p \:z_{p(\sigma+1)} \dotsm z_{p(\sigma_V)}}^2
      \:e^{-\bphi-\phi_S} \dvol_V
      }{\logpole} \\ 
    \overset{\mathclap{\eqref{eq:bdds-on-bphi}}}\lesssim
    &\quad
      \int_V \frac{
      \abs{g_p \:z_{p(\sigma+1)} \dotsm z_{p(\sigma_V)}}^2
      \:e^{-\alert{\bphip}-\phi_S} \dvol_V
      }{\logpole}
      \leq
      \int_V \frac{
      \abs{g_p \:z_{p(\sigma+1)} \dotsm z_{p(\sigma_V)}}^2
      \:e^{-\bphip-\phi_S} \dvol_V
      }{\logpole|\alert{\psi_p}|} \\ 
    =
    &~\int_{U_p \times W_p} \frac{
      \abs{g_p}^2 \:e^{-\bphip} \dvol_{U_p} \dvol_{W_p}
      }{\abs{z_{p(1)} \dotsm z_{p(\sigma)}}^2 \:\logpole|\psi_p|} \\
    =
    &~\int_{U_p} \frac{\dvol_{U_p}} {\abs{z_{p(1)} \dotsm
      z_{p(\sigma)}}^2 \:\logpole|\psi_p|}
      \cdot \int_{\lcS} \abs{g_p}^2 \:e^{-\bphi} \dvol_{\lcS} \quad < +\infty \; .
  \end{align*}
  Therefore, the holomorphic function
  \begin{equation*}
    f := \sum_{p\in\cbn} g_p \:\sect_0
    \:z_{p(\sigma+1)} \dotsm z_{p(\sigma_V)} \quad\text{ on } V
  \end{equation*}
  satisfies $\RTF|f|(\eps)[V,\sigma] < +\infty$ for all $\eps > 0$,
  and thus $f \in \aidlof{\vphi_L}_x$.
  One has $\Res_x(f) = \paren{g_p}_{p\in\cbn}$ by construction.
  This completes the proof.
\end{proof}

The following can be considered as a local $L^2$ extension theorem for
extending holomorphic sections on $\sigma$-lc centres locally with
estimates.
The constant of the estimate depends only ``mildly'' on the given
metric.

\begin{cor} \label{cor:local-L2-estimates}
  For any given $\paren{g_p}_{p\in\cbn} \in \residlof{\bphi}
  \paren{\cl V}$ on an admissible open set $V$, the extension $f \in
  \aidlof{\vphi_L}\paren{\cl V}$ constructed in the proof of Theorem
  \ref{thm:short-exact-seq} satisfies the estimate
  \begin{equation*}
    \RTF|f|(\eps)[V,\sigma] 
    \leq C \sum_{p\in\cbn} \frac{\pi^\sigma}{\paren{\sigma-1}! \:\vect\nu_p}
    \int_{\lcS} \abs{g_p}^2 \:e^{-\bphi}
    \dvol_{\lcS}
    =C \:\RTF|f|(0)[V,\sigma]
  \end{equation*}
  for all $\eps > 0$ and for some constant $C > 0$
  which depends only on the constants involved in
  $\lesssim_\tlog$ in the inequality $\bphip \lesssim_\tlog \bphi$ for
  every $p\in\cbn$ given in \eqref{eq:bdds-on-bphi}.\footnote{
    In Theorem \ref{thm:main-results}
    \eqref{item:thm:local-L2-estimate}, the constant $C$ is described
    as depending only ``mildly'' on $\bphi$ because such kind of
    constants is ``robust'' under the approximation of
    quasi-psh functions.
    Indeed, if $\bphi$ has more general singularities and if $\bphip
    \leq \bphi + C$ for some constant $C \geq 0$, then the Bergman kernel
    approximations $\dep[k]{\bphip}$ and $\dep[k]{\bphi}$ (with neat
    analytic singularities) of the quasi-psh potentials $\bphip$ and
    $\bphi$, respectively, satisfy $\dep[k]{\bphip} \leq \dep[k]{\bphi}
    + C$ for the same $C$.
    The variables of $\dep[k]{\bphip}$ can be separated relatively
    easily on admissible open sets in favour of the computation using
    Fubini's theorem as in the proof of Theorem
    \ref{thm:short-exact-seq}.
    These will be discussed in details in the subsequent paper.
  }
  (See \eqref{eq:lc-measure-for-non-cpt-supp} for the
  definition of $\RTF|f|(0)[V,\sigma]$ when $f$ is not compactly
  supported in $V$ and see Theorem \ref{thm:results-on-residue-fct}
  \eqref{item:thm:res-norm} for the definition of $\vect\nu_p$.)
\end{cor}

\begin{proof}
  This essentially follows from the estimates of integrals in the
  proof of Theorem \ref{thm:short-exact-seq} and the computation in
  \cite{Chan_on-L2-ext-with-lc-measures}*{Proof of Prop.~2.2.1}.
  Following the notation in the proof of Theorem
  \ref{thm:short-exact-seq}, one has
  \begin{align*}
    \RTF|f|(\eps)[V,\sigma] 
    &\leq  
      2 \sum_{p\in\cbn} \eps \int_V \frac{
      \abs{g_p \:z_{p(\sigma+1)} \dotsm z_{p(\sigma_V)}}^2
      \:e^{-\bphi-\phi_S} \dvol_V
      }{\logpole} \\ 
    &\overset{\mathclap{\text{pf.~of
      Thm.~\ref{thm:short-exact-seq}}}}\lesssim
      \qquad
      \sum_{p\in\cbn} \eps \int_{U_p} \frac{\dvol_{U_p}} {\abs{z_{p(1)} \dotsm
      z_{p(\sigma)}}^2 \:\logpole|\psi_p|}
      \cdot \int_{\lcS} \abs{g_p}^2 \:e^{-\bphi} \dvol_{\lcS} \; ,
  \end{align*}
  where the constant involved in $\lesssim$ depends only on the
  constants involved in $\lesssim_{\tlog}$ in the inequalities $\bphip
  \lesssim_{\tlog} \bphi$ for all $p \in \cbn$ given in
  \eqref{eq:bdds-on-bphi}.

  It remains to estimate, for each $p \in\cbn$, the integral
  \begin{equation*}
    \RTI[1](\eps)
    :=\eps \int_{U_p} \frac{\dvol_{U_p}} {\abs{z_{p(1)} \dotsm
        z_{p(\sigma)}}^2 \:\logpole|\psi_p|} \; .
  \end{equation*}
  (Note that the symbol $\RTI$ is chosen to match with the one in
  \cite{Chan_on-L2-ext-with-lc-measures}*{Proof of Prop.~2.2.1} so
  that \cite{Chan_on-L2-ext-with-lc-measures}*{(eq$\,$2.2.4)} can be
  applied directly.)
  For simplicity, assume that $U_p$ is a $\sigma$-disc with a uniform
  poly-radius $R$ (adjust the function $\rho$ defined below suitably
  when the poly-radius is not uniform).
  For any $\delta \in (0, R)$, let $\rs\rho :=\rs\rho_\delta \colon
  \fieldR_{\geq 0} \to [0,1]$ be a \emph{non-increasing} smooth
  function such that
  \begin{equation*}
    \rs\rho(t) =
    \begin{dcases}
      1 & \text{for } t \in [0,R-\delta] \; , \\
      0 & \text{for } t \geq R \; .
    \end{dcases} 
  \end{equation*}
  Then, set
  \begin{equation*}
    \rho :=\prod_{j=1}^{\sigma} \rho_{p(j)}
    :=\prod_{j=1}^\sigma \rs\rho \circ \abs{z_{p(j)}}^2 \; .
  \end{equation*}
  Notice that one has
  \begin{equation*}
    \RTI[\rho](\eps)
    =\eps \int_{U_p} \frac{\rho \:\dvol_{U_p}} {\abs{z_{p(1)} \dotsm
        z_{p(\sigma)}}^2 \:\logpole|\psi_p|}
    \;\;\tendsto\; \RTI[1](\eps)
    \quad\text{ as }\quad
    \delta \tendsto 0^+
  \end{equation*}
  for each $\eps > 0$ by the dominated convergence theorem.

  Let $(r_j, \theta_j)$ be the polar coordinate system of the
  $z_j$-plane and write $\fdiff{r_j}$ as $\diff_{r_j}$ for
  convenience. 
  Now, applying
  \cite{Chan_on-L2-ext-with-lc-measures}*{(eq$\,$2.2.4)} to
  $\RTI[\rho](\eps)$ and noting the deliberate choice of $\psi_p$ in
  the denominator of the integrand of $\RTI[\rho](\eps)$, one obtains
  \begin{align*}
    \RTI[\rho](\eps)
    &\leq
    \frac{(-1)^\sigma}{(\sigma -1)!}
    \int_{U_p} \frac{\diff_{r_{p(\sigma)}} \dotsm \diff_{r_{p(1)}}
      \rho }{\paren{\slog|\psi_p|}^\eps} \prod_{j=1}^\sigma
      \frac{dr_{p(j)} d\theta_{p(j)}}{2\nu_{p(j)}} \\
    &=\frac{\pi^\sigma}{(\sigma -1)! \:\vect\nu_p}
      \int_{[0,R]^\sigma}
      \frac{1}{\paren{\slog|\psi_p|}^\eps}
      \prod_{j=1}^\sigma \fdiff{r_{p(\sigma)}}[\paren{-\rho_{p(j)}}]
      \:dr_{p(j)} \\
    &\leq
      \frac{\pi^\sigma}{(\sigma -1)! \:\vect\nu_p}
      \int_{[0,R]^\sigma}
      \prod_{j=1}^\sigma \fdiff{r_{p(\sigma)}}[\paren{-\rho_{p(j)}}]
      \:dr_{p(j)}
      \;\;
      =\; \frac{\pi^\sigma}{(\sigma -1)! \:\vect\nu_p} 
  \end{align*}
  for all $\eps > 0$.
  Note that the last inequality above makes use of the facts that
  $\slog|\psi_p| \geq 1$ and $\fdiff{r_{p(j)}}[\paren{-\rho_{p(j)}}]
  \geq 0$ on $U_p$ for $j=1,\dots,\sigma$.
  Finally, letting $\delta \tendsto 0^+$ yields
  \begin{equation*}
    \RTI[1](\eps) \leq \frac{\pi^\sigma}{(\sigma -1)! \:\vect\nu_p}
  \end{equation*}
  for all $\eps > 0$.
  This results in the desired estimate for $\RTF|f|(\eps)[V,\sigma]$.
\end{proof}


\section{The non-snc scenarios}
\label{sec:non-snc}


Suppose in this section that $\vphi_L$ and $\psi$ are given as in
Section \ref{sec:setup} (thus having only neat analytic singularities
in particular) and satisfy the Snc assumption \ref{assum:snc} (but
need not satisfy \ref{assum:snc-nas}).
Let $\pi \colon \rs X \to X$ be the log-resolution of $(X,\vphi_L,\psi)$
given in Snc assumption \ref{assum:snc}.
Recall from Section \ref{sec:snc-assumption} the decomposition
$E_{d\pi} = E + R$ 
and the notation of the canonical sections $\sect_E$ and $\sect_R$ and
the potential $\pi^*_\ominus\vphi_L :=\pi^*\vphi_L -\phi_R$.
Let $\rs S$ be the lc locus of $\set{\mtidlof[\rs
  X]{\pi^*_\ominus\vphi_L +m\pi^*\psi}}_{m \in [0,1]}$ at the jumping
number $m=1$, which is a reduced snc divisor. 
Recall also that $E$ and $\rs S$ have no common irreducible
components.

Let $\rs S_0$ be the divisor on $\rs X$ corresponding to
the divisor $S_0$ described in Remark
\ref{rem:consequences-of-snc-assumptions} with $(\rs
X,\pi^*_\ominus\vphi_L,\pi^*\psi)$ in place of
$(X,\vphi_L,\psi)$.
Also let $\rs\bphi$ be the quasi-psh potential on $\pi^*L \otimes
R^{-1} \otimes \rs S_0^{-1} \otimes \rs S^{-1}$ corresponding to the
potential $\bphi$ in \eqref{eq:define-bphi}
such that
\begin{equation*}
  \rs\bphi +\phi_{\rs S_0} +\phi_{\rs S}
  := \pi^*_\ominus\vphi_L +\pi^*\psi \; ,
\end{equation*}
where $\phi_{\rs S_0} =\log\abs{\rs\sect_0}^2$ and $\phi_{\rs S}
=\log\abs{\rs\sect}^2$, in which $\rs\sect_0$ and $\rs\sect$ are
respectively fixed canonical holomorphic sections of $\rs S_0$ and
$\rs S$ on $\rs X$.
Moreover, let $m_0 \in [0,1)$ be the number provided by Proposition
\ref{prop:log-resoln-on-jump-subvar} such that
\begin{equation*}
  \mtidlof[\rs X]{\pi^*_\ominus\vphi_L +m_0\pi^*\psi}
  =\mtidlof[\rs X]{\pi^*_\ominus\vphi_L +m\pi^*\psi}
  \supsetneq
  \mtidlof[\rs X]{\pi^*_\ominus\vphi_L +\pi^*\psi}
  \quad\text{ for all } m\in [m_0,1) \; .
\end{equation*}

For definiteness, define the volume form $\dvol_{\rs X,
  \paren{\pi^*\omega}_{>0}}$ on $\rs X$ given by
\begin{equation} \label{eq:induced-volume-via-log-resoln}
  \pi^*\dvol_{X,\omega} =\frac{\pi^*\omega^{n}}{n!}
  =: \abs{\sect_E}^2 \:e^{\phi_R} \dvol_{\rs X, \paren{\pi^*\omega}_{>0}}
  \; .
\end{equation} 
In this section, for the sake of notational convenience, the
subscripts ``$\omega$'' and ``$\paren{\pi^*\omega}_{>0}$'' are made
implicit in the notation for volume forms, including those induced on
the lc centres.
(One can also avoid considering the volume form induced from $\omega$
if $(n,0)$-forms are considered in the following discussion;
see Remark \ref{rem:residue-norm-on-hol-n-forms-intrinsic}.)

\subsection{The direct image of the residue short exact sequence}
\label{sec:short-exact-seq-push-forward}

Following the arguments in the proof of Proposition
\ref{prop:log-resoln-on-jump-subvar}, since
\begin{equation*}
  \int_V \frac{\abs f^2 \:e^{-\vphi_L-\psi} \dvol_X}{\logpole}
  = \int_{\mathrlap{\pi^{-1}\paren{V}}}
  \quad\frac{\abs{\pi^*f \cdot \sect_E}^2
    \:e^{-\pi^*_\ominus\vphi_L -\pi^*\psi} \dvol_{\rs X}}{\logpole|\pi^*\psi|}
\end{equation*}
for any $f \in \holo_X\paren{V}$ on any open set $V \subset X$ and for any
$\sigma \geq 0$ and $\eps > 0$ (notice also that
$\frac{e^{-\vphi_L-\psi}}{\logpole} \geq C >0$ locally in $X$ for some
constant $C$), one immediately obtains 
\begin{gather}
  \notag
  \aidlof<X>{\vphi_L} \cdot \holo_{\rs X}
  \hookrightarrow 
    \aidlof<\rs X>{\pi^*_\ominus\vphi_L -\phi_E}[\pi^*\psi]
    \xhookrightarrow{\:\otimes \sect_E\:} 
    E \otimes \aidlof<\rs X>{\pi^*_\ominus\vphi_L}[\pi^*\psi]
    \quad\mathrlap{\text{ and}} \\
  \label{eq:log-resoln-aidl}
  \aidlof<X>{\vphi_L}
  \isom \pi_*\aidlof<\rs X>{\pi^*_\ominus\vphi_L -\phi_E}[\pi^*\psi]
    \isom \pi_*\paren{E \otimes \aidlof<\rs X>{\pi^*_\ominus\vphi_L}[\pi^*\psi]}
\end{gather}
for any integer $\sigma \geq 0$, in which both globally defined
sheaf-homomorphisms on the far right-hand-side depend on the choice of
$\sect_E$.
Recall that $\lcc<\rs X>(\rs* S)$ is defined via the definition of lc
centres of $(\rs X,\rs S)$ in \cite{Kollar_Sing-of-MMP}*{Def.~4.15}
according to Remark \ref{rem:consequences-of-snc-assumptions}.


Theorem \ref{thm:short-exact-seq} provides the residue short exact
sequences 
\begin{equation*}
  \xymatrix@R=0.5cm@C=0.65cm{
    0 \ar[r]
    & {\aidlof|\sigma-1|<\rs X>{\pi^*_\ominus\vphi_L -\phi_E}[\pi^*\psi]}
    \ar[r] \ar@{^(->}[d]^-{\otimes \sect_E}
    & {\aidlof<\rs X>{\pi^*_\ominus\vphi_L -\phi_E}[\pi^*\psi]}
    \ar[r]^-{\rs\Res} \ar@{^(->}[d]^-{\otimes \sect_E}
    & {\residlof<\rs X>{\rs \bphi -\phi_E}} \ar[r] \ar[d]^-{t_E}
    & 0
    \\
    0 \ar[r]
    & {E \otimes \aidlof|\sigma-1|<\rs X>{\pi^*_\ominus\vphi_L}[\pi^*\psi]}
    \ar[r]
    & {E \otimes \aidlof<\rs X>{\pi^*_\ominus\vphi_L}[\pi^*\psi]}
    \ar[r]^-{\rs\Res}
    & {E \otimes \residlof<\rs X>{\rs \bphi}}
    \ar[r]
    & 0
  }
\end{equation*}
for any integer $\sigma \geq 1$.
From the facts that all involving divisors are in the snc
configuration and that $E$ contains no lc centres of $(\rs X, \rs S)$
and no components of $\rs\bphi^{-1}(-\infty)$,
one sees that the map $t_E$ is injective.
Applying $\pi_*$ to the above diagram and taking
\eqref{eq:log-resoln-aidl} into account, a diagram-chasing argument
then yields the exact sequence
\begin{equation} \label{eq:isom-of-2models-of-residl}
  \begin{aligned}
    \xymatrix@C+0.3cm@R=0.01cm{ &&& {\pi_*\residlof<\rs X>{\rs \bphi
          -\phi_E}} \ar[dd]^-{\isom}
      \\
      0 \ar[r] & {\aidlof|\sigma-1|<X>{\vphi_L}} \ar[r] &
      {\aidlof<X>{\vphi_L}} \ar[ur]^-{\pi_*\rs\Res}
      \ar[dr]+<-46pt, 13pt>_-{\pi_*\rs\Res} &
      \\
      &&& {\pi_*\paren{E \otimes \residlof<\rs X>{\rs \bphi}} \; .}  }
  \end{aligned}
\end{equation}
In what follows, the sheaf $E \otimes \aidlof<\rs
X>{\pi^*_\ominus\vphi_L}[\pi^*\psi]$ is used as a model of
$\aidlof<X>{\vphi_L}$ on $\rs X$ via the log-resolution, but all
results still hold true when $\aidlof<\rs X>{\pi^*_\ominus\vphi_L
  -\phi_E}[\pi^*\psi]$ is used as the model instead.

The homomorphism $\pi_*\rs\Res$ need not be surjective (so
$R^1\pi_*\paren{E \otimes \aidlof|\sigma-1|<\rs
  X>{\pi^*_\ominus\vphi_L}[\pi^*\psi]} \neq 0$ in general), as 
can be seen in the following example.
\begin{example} \label{example:the-cross-in-2-disc}
  Let $X :=\Delta^2 \subset \fieldC^2$ be the unit $2$-disc centred at
  the origin $\vect 0 =(0,0)$ with holomorphic coordinates $z_1$ and
  $z_2$.
  Let also
  \begin{equation*}
    \vphi_L:=0 \quad\text{ and }\quad
    \psi = \log\abs{z_1}^2 +\log\abs{z_2}^2 -1 \; .
  \end{equation*} 
  Then the family $\set{\mtidlof[X]{m\psi}}_{m\in [0,1]}$ has only
  $m=1$ as a jumping number, with $\mtidlof[X]{m\psi} =\holo_X$ for
  $m\in[0,1)$ and $\mtidlof[X]{\psi} =\genby{z_1 z_2}$.
  Therefore, $S = S_1 + S_2$, where $S_j :=\set{z_j = 0}$ for
  $j=1,2$.
  It can also be seen that
  \begin{equation*}
    \lcc[1] = S \quad\text{ and }\quad
    \lcc[2] = S_1 \cap S_2 = \set{\vect 0} \; .
  \end{equation*}

  Let $\pi \colon \rs X \to X$ be the blow-up at the origin with an
  exceptional divisor $R$ (and $E = 0$).
  Write $\rs S_j :=\pi^{-1}_* S_j$ as the proper transform of $S_j$ on
  $\rs X$.
  Note that
  \begin{equation*}
    \pi^*\psi \sim_{\tlog} \phi_{\rs S_1} +\phi_{\rs S_2} +2\phi_R
  \end{equation*}
  (see Notation \ref{notation:potentials}), so the family
  $\set{\mtidlof[\rs X]{-\phi_R+m\pi^*\psi}}_{m\in [0,1]}$ jumps only
  at $m=1$ and $\mtidlof[\rs X]{-\phi_R +m\pi^*\psi} =\holo_{\rs X}$
  for all $m \in [0,1)$.
  The subvariety corresponding to
  $\mtidlof[\rs X]{-\phi_R+\pi^*\psi}$ is $\rs S_1
  +\rs S_2 +R$, so $\rs S =\rs S_1 +\rs S_2 +R$
  and thus
  \begin{equation*}
    \lcc[1]<\rs X>(\rs* S) = \rs S \quad\text{ and }\quad
    \lcc[2]<\rs X>(\rs* S) = \paren{\rs* S_1 \cap R} \sqcup
    \paren{\rs* S_2 \cap R} =\set{p_1, p_2} \; ,
  \end{equation*}
  where $\set{p_j} := \rs S_j \cap R$ for $j=1,2$ and $\lcc[2]<\rs
  X>(\rs* S)$ is the union of two distinct points.
  Note also that $\pi\paren{\set{p_1, p_2}} =\set{\vect 0}$.
  Moreover, $\rs S_0 =0$.
  From a local computation, $\rs\bphi$ can be chosen to be the
  constant $-1$. 

  The sheaf $\residlof|2|<\rs X>{\rs\bphi} \isom \fieldC_{p_1} \oplus
  \fieldC_{p_2}$ is thus a skyscraper sheaf supported at $p_1$ and
  $p_2$.
  For any $2$-disc $V =\Delta_r^2 \subset X$ of radius $r \in (0,1)$
  centred at $\vect 0$ and for any section $f \in 
  \aidlof|2|<X>{0}[\psi](V) =\pi_*\aidlof|2|<\rs
    X>{-\phi_R}[\pi^*\psi](V)$, the image $\pi_*\rs\Res\paren{f}
  =\rs\Res\paren{\pi^*f}$ has to lie inside the diagonal of
  $\pi_*\residlof|2|<\rs X>{\rs\bphi}(V) \isom\fieldC \times
  \fieldC$, so $\pi_*\rs\Res$ cannot be surjective at the origin
  $\vect 0$.
\end{example}

The following Theorem implies that the image of $\pi_*\rs\Res$ is
indeed independent of the log-resolution $\pi$ up to isomorphisms.

\begin{thm} \label{thm:direct-image-residl}
  Suppose that $(X,\vphi_L,\psi)$ itself satisfies the Snc assumption
  \ref{assum:snc-nas}, and let $\Res$ and $\residlof<X>{\bphi}$ be the
  residue morphism and the target sheaf of $\Res$ given in the residue
  short exact sequence in Theorem \ref{thm:short-exact-seq}.
  Then, there is an isomorphism $\tau$ between $\residlof<X>{\bphi}$
  and $\im\pi_*\rs\Res \subset \pi_*\paren{E \otimes \residlof<\rs
    X>{\rs\bphi}}$ such that the diagram
  \begin{equation*}
    \xymatrix{
      &
      {\aidlof<X>{\vphi_L}} \ar[dl]_-{\Res}
      \ar[dr]^-{\pi_*\rs\Res} \ar@{}[d]|*+{\circlearrowleft}
      &
      \\
      {\residlof<X>{\bphi}} \ar[rr]^-{\tau}_-{\isom}
      &
      &
      {\im\pi_*\rs\Res}
    }
  \end{equation*}
  commutes.
  Moreover, $\tau$ is locally an isometry in the sense that, for any
  admissible open set $V \Subset X$
  and for any
  $\paren{g_p}_{p\in \cbn} \in \residlof<X>{\bphi}(\cl V)$ and any $f
  \in \aidlof<X>{\vphi_L}(\cl V)$ such that $\paren{g_p}_{p\in \cbn}
  =\Res\paren f$, one has $\tau \paren{g_p}_{p\in \cbn}
  =\pi_*\rs\Res\paren{f} =\rs\Res\paren{\pi^*f \cdot \sect_E}
  =\paren{\rs g_q \cdot \sect_E}_{q\in \Iset}$ (where $\lcS/\rs
  S/[q]$ for $q \in \Iset$ are the $\sigma$-lc centres of
  $\paren{\pi^{-1}\paren{V}, \rs S \cap \pi^{-1}\paren{V}}$ such that
  $\lcc<\pi^{-1}\paren{V}>(\rs* S) = \bigcup_{q \in \Iset} \lcS/\rs
  S/[q]$) and 
  \begin{equation*}
    \int_{\lcc<V>} \abs f^2 \lcV
    =\int_{\lcc<\pi^{-1}\paren V>(\rs* S)} \abs{\pi^*f \cdot
      \sect_E}^2 \lcV/\pi^*_\ominus\vphi_L/[\pi^*\psi] 
  \end{equation*}
  (see \eqref{eq:lc-measure-for-non-cpt-supp}), or, more explicitly
  (according to Theorem \ref{thm:results-on-residue-fct}
  \eqref{item:thm:res-norm}),
  \begin{equation*}
    \sum_{p \in\cbn/\sigma_{V}/} \frac{\pi^\sigma}{(\sigma-1)! \:
      \vect\nu_{p}}
    \int_{\lcS} \abs{g_p}^2 \:e^{-\bphi}
    \dvol_{\lcS}
    =\sum_{q \in \Iset} \frac{\pi^\sigma}{(\sigma-1)! \:
      \vect\nu_{q}}
    \int_{\lcS/\rs S/[q]} \abs{\rs g_q \cdot \sect_E}^2 \:e^{-\rs\bphi}
    \dvol_{\lcS/\rs S/[q]} \; , 
  \end{equation*}
  where $\vect\nu_{p}$'s (resp.~$\vect\nu_{q}$'s) are the
  products of the generic Lelong numbers $\lelong{\psi}[D_i]$ of
  $\psi$ (resp.~$\lelong{\pi^*\psi}[\rs* D_i]$ of $\pi^*\psi$) along
  irreducible components $D_i$ of $S$ (resp.~$\rs D_i$ of $\rs S$)
  which contain $\lcS$ (resp.~$\lcS/\rs S/[q]$).
\end{thm}

\begin{proof}
  The well-definedness and bijectivity of the homomorphism $\tau$ follow
  immediately from the residue short exact sequence on $X$ (see
  Theorem \ref{thm:short-exact-seq}) and the direct image of the one
  on $\rs X$ (see the beginning of Section
  \ref{sec:short-exact-seq-push-forward}).

  To see that $\tau$ is an isometry, suppose that $V' \Supset V$ is an
  admissible open set in the same coordinate system on
  $V$ and that $\paren{g_p}_{p\in \cbn} \in \residlof<X>{\bphi}(V')$.
  Let $\rho \colon V' \to [0,1]$ be any compactly supported smooth
  cut-off function with $\res{\rho}_V \equiv 1$.
  For any $f \in \aidlof<X>{\vphi_L}(V')$ such that
  $\paren{g_p}_{p\in \cbn} =\Res\paren f$ (which exists by the proof
  of Theorem \ref{thm:short-exact-seq}, or from the fact that $V'$ is
  Stein and $\aidlof|\sigma-1|<X>{\vphi_L}$ is coherent
  such that $\cohgp 1[V']{\aidlof|\sigma-1|<X>{\vphi_L}}
  =0$) and $\pi_*\rs\Res\paren{f} =\tau\paren{g_p}_{p\in\cbn}
  =\paren{\rs g_q \cdot \sect_E}_{q\in \Iset}$, one has
  \begin{align*}
    \RTF[\rho]|f|(\eps)[V',\sigma]
    &=\eps \int_{V'}  \frac{
      \rho \abs f^2 \:e^{-\vphi_L-\psi} \dvol_X
    }{\logpole} \\
    &=\eps \int_{\mathrlap{\pi^{-1}\paren{V'}}} \;\quad \frac{
      \pi^*\rho \abs{\pi^*f \cdot \sect_E}^2
      \:e^{-\pi^*_\ominus\vphi_L-\pi^*\psi} \dvol_{\rs X}
    }{\logpole|\pi^*\psi|}
    =: \RTF{\rs*{\RTFsym}}[\pi^*\paren{\rho\abs
      f^2}](\eps)[\pi^{-1}\paren{V'},\sigma] 
  \end{align*}
  for any $\eps > 0$.
  By Theorem \ref{thm:results-on-residue-fct}
  \eqref{item:thm:res-norm} (one may need to apply an argument with a
  partition of unity on $\pi^{-1}\paren{V'}$ when handling the
  integral on the right-hand-side), the above functions in $\eps$ can
  be continued analytically across $\eps = 0$ and thus it follows that
  \begin{align*}
    \int_{\lcc<V'>} \rho \abs f^2 \lcV
    &=\RTF[\rho]|f|(0)[V',\sigma] \\
    &=\RTF{\rs*{\RTFsym}}[\pi^*\paren{\rho\abs
      f^2}](0)[\pi^{-1}\paren{V'},\sigma] 
    =\int_{\lcc<\pi^{-1}\paren{V'}>(\rs* S)} \pi^*\rho \abs{\pi^*f
      \cdot \sect_E}^2
    \lcV/\pi^*_\ominus\vphi_L/[\pi^*\psi] \; .
  \end{align*}
  The proof is completed by taking a sequence of $\rho$'s such that
  $\rho$ descends pointwisely to the characteristic function of $\cl
  V$ on $X$.
\end{proof}

\begin{remark}
  The result that $\tau$ is locally an isometry on admissible open
  sets can be extended to general relatively compact open sets $V$,
  although one can in general only obtain a smooth extension $f \in
  \aidlof<X>{\vphi_L} \cdot \smooth_X (V)$ when following
  the proof above (cover the open set $V$ by admissible open sets
  $\set{V_\gamma}_\gamma$, obtain a holomorphic extension on each $V_\gamma$,
  then glue the extensions together via a partition of unity).
  This is because Theorem \ref{thm:results-on-residue-fct}
  \eqref{item:thm:res-norm} is still valid for smooth $f$ (see Remark
  \ref{rem:RTF-thm-on-smooth-sections}).
\end{remark}

\begin{remark}
  Under the same setting and given the same notation as in Theorem
  \ref{thm:direct-image-residl}, 
  if all components of $\paren{g_p}_{p\in\cbn}$ are $0$ except for
  $g_{p'}$ for some $p' \in\cbn$, then all components of $\paren{\rs
    g_q}_{q \in \Iset}$ are $0$ possibly except for those $q\in \Iset$
  such that $\pi\paren{\lcS*[q]} \subset \lcS[p']$.
  Indeed, by considering any (admissible) open set $V' \Subset V
  \setminus \lcS[p']$, it follows from the equation in Theorem
  \ref{thm:direct-image-residl} that $\rs g_q \equiv 0$ on $\lcS*[q]
  \cap \pi^{-1}(V')$, therefore on $\lcS*[q]$, for all $q \in \Iset$
  such that $\lcS*[q] \cap \pi^{-1}(V') \neq \emptyset$.
  Theorem \ref{thm:sigma-lc-centres} below implies that
  $\pi\paren{\lcS*[q]} \subset \lcc<V>(\vphi_L;\psi) =\lcc<V>(S)$ if
  $\rs g_q \not\equiv 0$ on $\lcS*[q]$ (note that $\lcS[p'] \subset
  \lcc<V>$), hence the claim. 
\end{remark}

In view of Theorem \ref{thm:direct-image-residl}, the sheaves
$\residlof<X>{\cdot}$ and the residue norms can be defined even for
the system $(X,\vphi_L,\psi)$ which does \emph{not} satisfy the Snc
assumption \ref{assum:snc-nas}.
\begin{definition} \label{def:non-snc-residl-residue-norm}
  Without the Snc assumption \ref{assum:snc-nas} on
  $(X,\vphi_L,\psi)$, one can define the sheaf
  $\residlof<X>{\vphi_L;\psi}$ and the residue morphism $\Res$ by
  \begin{equation*}
    \residlof<X>{\vphi_L;\psi}
    := \im \pi_*\rs\Res
    \;\subset \pi_*\paren{E \otimes \residlof<\rs
      X>{\rs\bphi}}
    \quad\text{ and }\quad
    \Res :=\pi_*\rs\Res \; .
  \end{equation*}
  Moreover, an $L^2$ norm with respect to the $\sigma$-lc measure
  $\lcV$, which is referred to as the \emph{residue norm}
  $\norm\cdot_{\lcc<V>(\vphi_L;\psi)}$, on
  $\residlof<X>{\vphi_L;\psi}(\cl V)$ for any open set $V \Subset X$ can
  be defined by setting, for any $g \in \residlof<X>{\vphi_L;\psi}(\cl V)$ which
  has a smooth lifting $f \in \aidlof<X>{\vphi_L} \cdot \smooth_X(V')$
  via $\Res$ on an open set $V' \Supset V$,
  \begin{equation*}
    \norm g^2_{\lcc<V>(\vphi_L;\psi)}
    :=\lim_{\rho \descendsto \charfct_{\cl V}} \:
    \lim_{\eps \tendsto 0^{\mathrlap{+}}} \:
    \RTF[\rho]|f|(\eps)[V',\sigma]
    \paren{=: \int_{\mathrlap{\lcc(\vphi_L;\psi)}} \qquad \abs f^2 \lcV =: \RTF|f|(0)[V,\sigma]}
    \; ,
  \end{equation*}
  where the limit on the compactly supported smooth cut-off functions
  $\rho \colon V' \to [0,1]$ is defined as in
  \eqref{eq:lc-measure-for-non-cpt-supp}.
\end{definition} 
By considering a common log-resolution of two given log-resolutions of
$(X,\vphi_L,\psi)$, it follows from Theorem
\ref{thm:direct-image-residl} that $\residlof<X>{\vphi_L;\psi}$ and
$\Res$ such defined are independent of the log-resolution $\pi$ up to
$\holo_X$-isomorphisms.

The result of this section is summarised as follows.
\begin{cor} \label{cor:residue-exact-seq-for-non-snc}
  For $(X,\vphi_L,\psi)$ given as in the beginning of Section
  \ref{sec:non-snc}, 
  the adjoint ideal sheaves $\aidlof<X>{\vphi_L}$ satisfy the
  residue short exact sequences
  \begin{equation*}
    \xymatrix{
      0 \ar[r]
      & {\aidlof|\sigma-1|<X>{\vphi_L}} \ar[r]
      & {\aidlof<X>{\vphi_L}} \ar[r]^-{\Res}
      & {\residlof<X>{\vphi_L;\psi}} \ar[r]
      & 0
    }
  \end{equation*}
  for all integers $\sigma \geq 1$, even when $S$ is not an snc
  divisor.
  Moreover, for any open set $V \Subset X$ and for any log-resolution
  $\pi \colon \rs X \to X$ of $(X,\vphi_L,\psi)$, if
  $\residlof<X>{\vphi_L;\psi}\paren{\cl V}$ is equipped with the $L^2$
  residue norm $\norm\cdot_{\lcc<V>(\vphi_L;\psi)}$ with
  respect to the $\sigma$-lc measure while $E \otimes \residlof<\rs
  X>{\pi^*_\ominus\vphi_L; \pi^*\psi}\paren{\pi^{-1}(\cl V)}$ (where
  $\residlof<\rs X>{\pi^*_\ominus\vphi_L; \pi^*\psi} :=\residlof<\rs
  X>{\rs\bphi}$) is equipped with $\norm\cdot_{\lcc<\pi^{-1}(V)>(\rs*
    S)}$ with respect to the corresponding $\sigma$-lc measure (which
  is induced from the volume form on $\rs X$ described in
  \eqref{eq:induced-volume-via-log-resoln}), then the monomorphism
  \begin{equation*}
    \tau \colon \residlof<X>{\vphi_L;\psi}\paren{\cl V}
    \hookrightarrow
    E \otimes \residlof<\rs X>{\pi^*_\ominus\vphi_L;
      \pi^*\psi}\paren{\pi^{-1}(\cl V)}
  \end{equation*}
  is an isometric embedding.
\end{cor}

\begin{remark} \label{rem:residue-norm-on-hol-n-forms-intrinsic}
  Using the isomorphism
  \begin{equation*}
    K_X \otimes \aidlof<X>{\vphi_L} \isom
    \pi_*\paren{K_{\rs X} \otimes R^{-1} \otimes \aidlof<\rs
      X>{\pi^*_\ominus\vphi_L}[\pi^*\psi]}
  \end{equation*}
  given by $f \mapsto \frac{\pi^*f}{\sect_{R}}$ in place of
  \eqref{eq:log-resoln-aidl}, Theorem \ref{thm:direct-image-residl}
  and Corollary \ref{cor:residue-exact-seq-for-non-snc} can be
  reformulated for $K_X \otimes \aidlof<X>{\vphi_L}$.
  The volume form on $\rs X$ described in
  \eqref{eq:induced-volume-via-log-resoln} is not needed and the
  residue norm $\norm\cdot_{\lcc<V>(\vphi_L;\psi)}$ thus constructed on
  $K_X \otimes \residlof<X>{\vphi_L;\psi}\paren{\cl V}$ is independent
  of $\omega$.
  The isometry between $K_X \otimes \residlof<X>{\vphi_L;\psi}$ and
  its image in $\pi_*\paren{K_{\rs X} \otimes R^{-1} \otimes
    \residlof<\rs X>{\rs\bphi}}$ is therefore a more intrinsic
  property.
\end{remark}

\begin{remark}\label{rem:1-residue-sheaf}
  Using the vanishing of the higher direct image $R^1\pi_*\paren{E
    \otimes \mtidlof<\rs X>{\pi^*_\ominus \vphi_L +\pi^*\psi}} = 0$
  (see \cite{Lazarsfeld_book-II}*{Thm.~9.4.1} in the algebraic setting
  and \cite{Matsumura_injectivity-Kaehler}*{Cor.~1.5} in the analytic
  setting), one indeed always has $\residlof|1|<X>{\vphi_L;\psi} \isom
  \pi_*\paren{E\otimes \residlof|1|<\rs X>{\rs\bphi}}$.
\end{remark}

\subsection{Definition of $\sigma$-lc centres via adjoint ideal sheaves}
\label{sec:sigma-lc-centres}

Under the assumption that both $\vphi_L$ and $\psi$ have
only neat analytic singularities and according to Theorem
\ref{thm:results-on-residue-fct} \eqref{item:thm:res-fct} and Remark
\ref{rem:residue-fct-result-jumping-no}, one has
\begin{equation*}
  \begin{aligned}[b]
    \aidlof<\rs X>{\pi^*_\ominus\vphi_L}[\pi^*\psi]
    &=\mtidlof[\rs X]{\pi^*_\ominus\vphi_L +m_0\pi^*\psi} \cdot
    \defidlof{\lcc[\sigma+1]<\rs X>(\rs* S)} \\
    &=\mtidlof[\rs X]{\rs*\bphi +\phi_{\rs S_0}} \cdot
    \defidlof{\lcc[\sigma+1]<\rs X>(\rs* S)}
  \end{aligned}
  \;\;\text{ for all integers $\sigma \geq 0$} \; .
\end{equation*}
Recall again that $\lcc<\rs X>(\rs* S)$ is defined via the definition
of lc centres of $(\rs X,\rs S)$ in
\cite{Kollar_Sing-of-MMP}*{Def.~4.15} according to Remark
\ref{rem:consequences-of-snc-assumptions}, so every irreducible component
of $\lcc<\rs X>(\rs* S)$ is precisely a connected component of the
intersection of some choice of $\sigma$ irreducible components of $\rs
S$, as $\rs S$ is a reduced snc divisor.
Moreover, it is easy to see from a direct computation that
\begin{equation*}
  \Ann_{\holo_{\rs X}}\paren{\frac{
      \aidlof<\rs X>{\pi^*_\ominus\vphi_L}[\pi^*\psi]
    }{
      \aidlof|\sigma-1|<\rs X>{\pi^*_\ominus\vphi_L}[\pi^*\psi]
    }}
  =\Ann_{\holo_{\rs X}}\paren{\frac{
      \mtidlof[\rs X]{\rs*\bphi +\phi_{\rs S_0}}
      \cdot \defidlof{\lcc[\sigma+1]<\rs X>(\rs* S)}
    }{
      \mtidlof[\rs X]{\rs*\bphi +\phi_{\rs S_0}}
      \cdot \defidlof{\lcc<\rs X>(\rs* S)}
    }}
  =\defidlof{\lcc<\rs X>(\rs* S)} \; .
\end{equation*}
This indeed makes sense as the residue short exact sequence implies
that $\frac{ 
  \aidlof<\rs X>{\pi^*_\ominus\vphi_L}[\pi^*\psi]
}{
  \aidlof|\sigma-1|<\rs X>{\pi^*_\ominus\vphi_L}[\pi^*\psi]
} \isom \residlof<\rs X>{\pi^*_\ominus\vphi_L;\pi^*\psi}$ and
$\Ann_{\holo_{\rs X}} \paren{\residlof<\rs
  X>{\pi^*_\ominus\vphi_L;\pi^*\psi}}$ defines the support of
$\residlof<\rs X>{\pi^*_\ominus\vphi_L;\pi^*\psi}$ (at least
set-theoretically), which is exactly $\lcc<\rs X>(\rs* S)$ (note that
the zero locus of $\Ann \sheaf F$ of a coherent sheaf $\sheaf F$ is
exactly the support of $\sheaf F$; see, for example,
\cite{Grauert&Remmert-CAS}*{A.4.5}).
Notice also that $\Ann_{\holo_{\rs X}}\paren{\frac{
    \aidlof<\rs X>{\pi^*_\ominus\vphi_L}[\pi^*\psi]
  }{
    \aidlof|\sigma-1|<\rs X>{\pi^*_\ominus\vphi_L}[\pi^*\psi]
  }} =\Ann_{\holo_{\rs X}} \paren{\residlof<\rs
  X>{\pi^*_\ominus\vphi_L;\pi^*\psi}}$ is a radical ideal sheaf since
so is $\defidlof{\lcc<\rs X>(\rs* S)}$.

The following proposition essentially shows that the union of
$\sigma$-lc centres of $(X,\vphi_L,\psi)$ can be defined independent
of log-resolutions.

\begin{thm} \label{thm:sigma-lc-centres}
  Under the assumption that $\vphi_L$ has only neat analytic
  singularities such that the Snc assumption \ref{assum:snc} is
  satisfied (but $(X,\vphi_L,\psi)$ may not satisfy the Snc assumption
  \ref{assum:snc-nas}),
  for any log-resolution $\pi \colon \rs X \to X$ of
  $(X,\vphi_L,\psi)$, one has
  \begin{equation*}
    \Ann_{\holo_X}\paren{\frac{
        \aidlof<X>{\vphi_L} 
      }{
        \aidlof|\sigma-1|<X>{\vphi_L}
      }}
    =\Ann_{\holo_X}\paren{\pi_*\paren{ E \otimes
        \frac{
          \aidlof<\rs X>{\pi^*_\ominus\vphi_L}[\pi^*\psi]
        }{
          \aidlof|\sigma-1|<\rs X>{\pi^*_\ominus\vphi_L}[\pi^*\psi]
        }
      }} 
  \end{equation*}
  and the above ideal sheaf is a radical ideal sheaf.
  This implies that the sheaves $\residlof<X>{\vphi_L;\psi}$ and
  $\pi_*\paren{E \otimes \residlof<\rs X>{\pi^*_\ominus\vphi_L;
      \pi^*\psi}}$ have the same support in particular.
  Let $\lcc(\vphi_L;\psi)$ be the analytic subset defined by the above
  ideal sheaf.
  Then one has, in general,
  \begin{equation*}
    \lcc(\vphi_L;\psi) \subset \pi\paren{\lcc<\rs X>(\rs* S)} \; .
  \end{equation*}
  If $\lcS*[q]$ is a $\sigma$-lc centre in $\lcc<\rs X>(\rs* S)$ such
  that $\pi(\lcS*[q]) \not\subset \lcc(\vphi_L;\psi)$, then
  $\res{\frac{\pi^*f}{\rs \sect_0}}_{\lcS*[q]} \equiv 0$ on $\lcS*[q] \cap
  \pi^{-1}(V)$ (or, equivalently, the component of $\rs\Res(\pi^*f)$
  on $\lcS*[q] \cap \pi^{-1}(V)$ vanishes identically) for all $f \in
  \aidlof<X>{\vphi_L}(V)$ and any open set $V \Subset X$ such that
  $\lcS*[q] \cap \pi^{-1}(V) \neq \emptyset$.
  
  Note that if $(X,\vphi_L,\psi)$ already satisfies the Snc assumption
  \ref{assum:snc-nas}, the above ideal sheaf is indeed
  $\defidlof{\lcc}$ and $\lcc(\vphi_L;\psi) = \lcc$.
\end{thm}

\begin{proof}
  It follows from the residue short exact sequence (see Theorem
  \ref{thm:short-exact-seq} or Corollary
  \ref{cor:residue-exact-seq-for-non-snc}) that the question under
  consideration is reduced to proving the claim
  \begin{equation*}
    \Ann_{\holo_X}\paren{
      \residlof<X>{\vphi_L ; \psi}
    }
    =\Ann_{\holo_X}\paren{\pi_*\paren{ E \otimes
        \residlof<\rs X>{\rs\bphi}
      }}
    \overset{\text{~\eqref{eq:isom-of-2models-of-residl}}}=
    \Ann_{\holo_X}\paren{\pi_*
        \residlof<\rs X>{\rs\bphi -\phi_E}
      } \; .
  \end{equation*}
  It follows immediately from Remark \ref{rem:1-residue-sheaf} (a
  consequence of the vanishing of the higher direct images) that
  the above equality (on the left) always holds true when $\sigma =1$.
  The proof below deals with the cases where $\sigma \geq 2$ (although
  it is also applicable to the case $\sigma =1$).
  
  Recall that the sheaf $\residlof<\rs X>{\rs\bphi -\phi_E}$ is given as a
  direct sum $\bigoplus_{q \in \Iset/X/} \sheaf S_{\lcS*[q]}$,
  where each $\sheaf S_{\lcS*[q]} = \paren{\rs*\iota_q}_*
  \paren{\res{\rs* S_0^{-1}}_{\lcS*[q]} \otimes 
  \paren{\Diff_{q} \rs* S}^{-1} \otimes \mtidlof<\lcS*[q]>{\rs\bphi -\phi_E}}$
  is, before taking the direct image $\paren{\rs*\iota_q}_*$ via the
  inclusion $\rs \iota_q \colon \lcS*[q] \hookrightarrow \rs X$, a
  torsion-free $\holo_{\lcS*[q]}$-sheaf (so its support is 
  precisely $\lcS*[q]$). 
  By setting
  \begin{equation} \label{eq:def-of-lcc'}
    \lcc<\rs X>(\rs* S)'
    :=\smashoperator[r]{\bigcup_{q \in \Iset/X/ \colon 
        \pi_*\sheaf S_{\lcS*[q]} \neq 0}} \;\;\;\lcS*[q]
    \qquad\paren{
      \subset
      \bigcup_{q \in \Iset/X/} \lcS*[q] =\lcc<\rs X>(\rs* S)
    }\; ,
  \end{equation}
  which is the union of $\sigma$-lc centres in $\lcc<\rs X>(\rs* S)$
  which contribute to the zero locus of the annihilator $\Ann_{\holo_{X}}
  \paren{\pi_*\residlof<\rs X>{\rs\bphi -\phi_E}}$.
  Take the annihilator as the defining ideal sheaf
  $\defidlof{\pi\paren{\lcc<\rs X>(\rs* S)'}}$ of $\pi\paren{\lcc<\rs
    X>(\rs* S)'}$ in $X$ (it is shown below that the annihilator is
  indeed radical).
  If the equality between the annihilators in the claim holds true, it
  then implies that $\lcc(\vphi_L;\psi) = \pi\paren{\lcc<\rs X>(\rs*
    S)'} \subset \pi\paren{\lcc<\rs X>(\rs* S)}$.
  This also implies that, if $\lcS*[q]$ is a $\sigma$-lc centre in
  $\lcc<\rs X>(\rs* S)$ such that $\pi\paren{\lcS*[q]} \not\subset
  \lcc(\vphi_L;\psi)$, one then has $\lcS*[q] \not\subset \lcc<\rs
  X>(\rs* S)'$ and thus $\pi_* \sheaf S_{\lcS*[q]} = 0$.
  The claim on the component of $\rs\Res(\pi^*f)$ on $\lcS*[q] \cap
  \pi^{-1}(V)$ for any $f \in \aidlof<X>{\vphi_L}(V)$ then follows.

  Back to the proof of the equality between the annihilators.
  Recall from Definition \ref{def:non-snc-residl-residue-norm}
  (together with the isomorphism in \eqref{eq:isom-of-2models-of-residl}) that
  $\residlof<X>{\vphi_L;\psi} :=\im \pi_*\rs\Res \subset \pi_*
  \residlof<\rs X>{\rs\bphi -\phi_E}$.
  Therefore, one immediately has
  \begin{equation*}
    \Ann_{\holo_X}\paren{
      \residlof<X>{\vphi_L ; \psi}
    }
    \supset \Ann_{\holo_X}\paren{\pi_*
      \residlof<\rs X>{\rs\bphi -\phi_E}
    } \; . 
  \end{equation*}
  If $\Ann_{\holo_X}\paren{\pi_* \residlof<\rs X>{\rs\bphi -\phi_E}}
  =\holo_X$, the reverse inclusion follows automatically.
  It remains to prove the reverse inclusion under the assumption that
  $\Ann_{\holo_X}\paren{\pi_* \residlof<\rs X>{\rs\bphi -\phi_E}} \neq
  \holo_X$ (so $\lcc<\rs X>(\rs* S)' \neq \emptyset$).
  For that, take any open polydisc $V \Subset X$ (in some coordinate
  chart) and any $h \in
  \Ann_{\holo_X}\paren{\residlof<X>{\vphi_L;\psi}}\paren{\cl V}$.
  Suppose $\lcc<\rs X>(\rs* S) \cap \pi^{-1}(V) =\bigcup_{q \in \Iset}
  \lcS/\rs S/[V,q]$ is the union of the $\sigma$-lc centres in
  $\pi^{-1}(V)$ and let $g
  =\tau^{-1} \rs g = \tau^{-1}(\rs g_q)_{q \in \Iset}
  \in \residlof<X>{\vphi_L ; \psi}\paren{\cl V}
  =\im \pi_*\rs\Res \paren{\cl V}
  \subset \residlof<\rs X>{\rs\bphi -\phi_E}\paren{\pi^{-1}(\cl V)}$ (in
  the notation in Theorem \ref{thm:direct-image-residl} and Corollary
  \ref{cor:residue-exact-seq-for-non-snc}) be an element such that each $\rs
  g_q$ is defined on the $\sigma$-lc centre $\lcS/\rs S/[V,q]$.
  Considering the residue norm (given in Corollary
  \ref{cor:residue-exact-seq-for-non-snc}) of $h \cdot g = \tau^{-1}(\pi^*h
  \:\rs g_q)_{q \in \Iset}$, it yields
  \begin{equation*} \tag{$*$} \label{eq:pf:res-norm-h.g}
    \begin{multlined}[c][0.92\textwidth]
      0 =\norm{h \cdot g}_{\lcc<V>(\vphi_L;\psi)}^2
      =\smashoperator{\sum_{q \in \Iset}}
      \frac{\pi^\sigma}{(\sigma-1)! \: \vect\nu_{q}} \int_{\lcS/\rs
        S/[V,q]} \abs{\pi^*h \: \rs g_q \cdot \sect_E}^2 \:e^{-\rs\bphi}
      \dvol_{\lcS/\rs S/[V,q]} \\
      =\smashoperator{\sum_{q \in \Iset}}
      \frac{\pi^\sigma}{(\sigma-1)! \: \vect\nu_{q}} \int_{\lcS/\rs
        S/[V,q]} \abs{\pi^*h \: \rs g_q}^2 \:e^{-\rs\bphi +\phi_E}
      \dvol_{\lcS/\rs S/[V,q]} 
      = \norm{\pi^*h \cdot \rs g}_{\lcc<\pi^{-1}(V)>(\rs* S)}^2 \; .
    \end{multlined}
  \end{equation*}
  It thus follows that, if $\rs g_q \not\equiv 0$ on $\lcS/\rs
  S/[V,q]$, then $\pi^*h \equiv 0$ on $\lcS/\rs S/[V,q]$; and if this
  holds true for all $\sigma$-lc centres $\lcS*[V,q] \subset \lcc<\rs
  X>(\rs* S)' \cap \pi^{-1}(V)$, then $h \in \Ann_{\holo_X}
  \paren{\pi_* \residlof<\rs X>{\rs\bphi -\phi_E}} \paren{V}$ and the
  desired inclusion follows.

  (The formula \eqref{eq:pf:res-norm-h.g} also shows that both of the
  annihilators $\Ann_{\holo_X}\paren{\pi_* \residlof<\rs X>{\rs\bphi
      -\phi_E}}$ and
  $\Ann_{\holo_X}\paren{\residlof<X>{\vphi_L;\psi}}$ are radical ideal
  sheaves.
  One sees this by replacing $h$ by $h^r$ for some integer $r > 0$ in
  \eqref{eq:pf:res-norm-h.g} and assuming $\paren{\rs
    g_q}_{q\in\Iset}$ being an element in $\pi_* \residlof<\rs
  X>{\rs\bphi -\phi_E}$ or $\residlof<X>{\vphi_L;\psi}$ (and $h^r$ being
  an element in the corresponding annihilator). 
  Since $\pi^*h^r \: \rs g_q \equiv 0$ on $\lcS/\rs S/[V,q]$ implies
  that $\pi^*h \: \rs g_q \equiv 0$ on $\lcS/\rs S/[V,q]$, the
  equality $\norm{h^r \cdot g}_{\lcc<V>(\vphi_L;\psi)}^2 =0$ then
  implies that $\norm{h \cdot g}_{\lcc<V>(\vphi_L;\psi)}^2 =0$, and
  thus the corresponding annihilator is a radical ideal sheaf.)

  Write $\aidlof{}*$ for $\aidlof<X>{\vphi_L}$ and
  $\rs{\aidlof{}*}$ for $\aidlof<\rs X>{\pi_\ominus^*\vphi_L
    -\phi_E}[\pi^*\psi]$ for convenience, and
  suppose that $\xi_1, \dots, \xi_r
  \in \aidlof<X>{\vphi_L}(\cl V)$ are generators of
  $\aidlof{}*$ on a neighbourhood of $\cl V$ for the given
  open set $V \Subset X$.
  Note that $\pi^*\xi_1 ,\dots, \pi^*\xi_r$ are then generators of
  $\rs{\aidlof{}*}$ on a neighbourhood of $\pi^{-1}(\cl V)$.
  To complete the proof, it suffices to show that, for any 
  $q' \in \Iset$ such that $\lcS*[V,q'] \subset \lcc<\rs X>(\rs* S)'
  \cap \pi^{-1}(V)$, there exists $\xi_i$ among the generators $\xi_1,
  \dots, \xi_r$ of $\aidlof{}*$ on a neighbourhood of $\cl V$ such
  that, by setting $g_i :=\Res\paren{\xi_i}$ (thus
  $\rs\Res\paren{\pi^*\xi_i} =\rs g_i =\paren{\rs g_{i, q}}_{q \in
    \Iset}$), one has $\rs g_{i,q'} \not\equiv 0$ on $\lcS*[V,q']$.

  Suppose, on the contrary, that $\rs g_{i,q'} \equiv 0$ on
  $\lcS*[V,q']$ for all $i =1, \dots, r$.
  Recall the direct sum decomposition
  $\pi_*\residlof<\rs X>{\rs\bphi -\phi_E}\paren{\cl V} = \bigoplus_{q
    \in \Iset} \pi_*\sheaf S_{\lcS*[V,q]} \paren{\cl V}$ (such that
  each $\pi_*\sheaf S_{\lcS*[V,q]}$ is supported on $\lcS*[V,q]$).
  Note that $\pi_*\sheaf S_{\lcS*[V,q']} (\cl V)$ is non-trivial as
  $\lcS*[V,q'] \subset \lcc<\rs X>(\rs* S)' \cap \pi^{-1}(V)$.
  Take any \emph{non-trivial} holomorphic section $\rs g =\paren{\rs
    g_q}_{q \in\Iset}$ in $\bigoplus_{q \in \Iset} \pi_*\sheaf
  S_{\lcS*[V,q]} \paren{\cl V}$ \emph{whose support contains
    $\lcS*[V,q']$}
  (one may conveniently take $\rs g_q \equiv 0$ on $\lcS*[V,q]$ for
  all $q \in \Iset \setminus \set{q'}$ and pick any non-trivial $\rs
  g_{q'} \in \sheaf S_{\lcS*[V,q']}\paren{\pi^{-1}(\cl V)}$).
  Let $\set{U_\alpha}_{\alpha \in \Gamma}$ be an open cover of
  $\pi^{-1}(\cl V)$ and $\set{\rho_\alpha}_{\alpha \in \Gamma}$ be a
  partition of unity subordinate to this cover.
  The cover is assumed to be sufficiently fine such that every
  section of $\frac{\rs{\aidlof{}*}}{\rs{\aidlof|\sigma-1|{}*}}$ on
  each $U_\alpha$ has a lifting to $\rs{\aidlof{}*}$ on $U_\alpha$.
  The section $\rs g$
  can then be expressed as a \v Cech cocycle $\set{\eqcls{\rs
      f_\alpha}}_{\alpha \in \Gamma}$, where $\rs f_\alpha \in
  \rs{\aidlof{}*}(U_\alpha)$ and $\eqcls{\rs f_\alpha} :=\rs f_\alpha
  \bmod \rs{\aidlof|\sigma-1|{}*}(U_\alpha)$ for each $\alpha \in
  \Gamma$ such that $\rs f_\beta -\rs f_\alpha \in
  \rs{\aidlof|\sigma-1|{}*}(U_\alpha \cap U_\beta)$ for any $\alpha,
  \beta \in \Gamma$ with $U_\alpha \cap U_\beta \neq \emptyset$.
  Set also $\rs F :=\sum_{\alpha \in \Gamma} \rho_\alpha \rs
  f_\alpha \in \smooth_{\rs X} \cdot \rs{\aidlof{}*}
  \paren{\pi^{-1}(\cl V)}$.




  Note that $\rs g =\paren{\rs g_q}_{q \in\Iset} =\rs\Res\paren{\rs*
    F}$ (after extending $\rs\Res$ to a $\smooth_{\rs
    X}$-homomorphism).
  Pick any admissible open set $U \Subset \pi^{-1}(V)$ with respect to
  $\paren{\pi^*_\ominus\vphi_L,\pi^*\psi}$ such that $U \cap \lcc<\rs
  X>(\rs* S) = U \cap \lcS*[V,q']$ (i.e.~$U$ intersects no $\sigma$-lc
  centres other than $\lcS*[V,q']$).
  The computation of residue norm (see Theorem
  \ref{thm:results-on-residue-fct} \eqref{item:thm:res-norm}, Remark
  \ref{rem:RTF-thm-on-smooth-sections} and
  \eqref{eq:lc-measure-for-non-cpt-supp}) yields
  \begin{equation*}
    \RTF{\rs*{\RTFsym}}|\rs* F|(0)[U,\sigma]
    = \lim_{\rho \descendsto \charfct_{\cl U}}
    \lim_{\eps \tendsto 0^+}
    \RTF{\rs*{\RTFsym}}[\rho]|\rs* F|(\eps)[U',\sigma]
    =
    \frac{\pi^\sigma}{(\sigma-1)! \:\vect\nu_{q'}}
    \int_{\lcS*[V,q'] \cap U} \abs{\rs g_{q'}}^2 \:e^{-\rs\bphi +\phi_E}
    \dvol_{\lcS*[V,q']} \; ,
  \end{equation*}
  where $U' \Supset U$ is some relatively compact subset in
  $\pi^{-1}(V)$ and $\rho$ runs through a decreasing sequence of
  compactly supported, non-negative, smooth cut-off functions on $U'$
  which are $\equiv 1$ on $U$.
  It suffices to show that $\rs g_{q'} \equiv 0$ on $U \cap
  \lcS*[V,q']$, which implies that $\rs g_{q'} \equiv 0$ on the whole
  of $\lcS*[V,q']$ by the identity theorem, to obtain a contradiction.

  The function $\rs F$
  can be written as a sum $\sum_{i=1}^r \rs a_i \pi^*\xi_i$ with $\rs
  a_i$'s being smooth functions on a neighbourhood of $\pi^{-1}(\cl
  V)$ thanks to the facts that $\pi^*\xi_1, \dots, \pi^*\xi_r$
  generate $\rs{\aidlof{}*}$ (which induce epimorphisms $\holo_{\rs
    X}^{\oplus r} \to \rs{\aidlof{}*} \to 0$ and ${\smooth_{\rs
      X}}^{\oplus r} \to \smooth_{\rs X} \cdot \rs{\aidlof{}*} \to 0$)
  on the neighbourhood and all $\smooth_{\rs X}$-modules are soft
  (see, for example, \cite{Grauert&Remmert}*{A.4.1 and A.4.4}; the
  neighbourhood of $\pi^{-1}(\cl V)$ can, of course, be taken to be
  paracompact).
  As all coefficients $\rs a_i$ in the sum $\rs F = \sum_{i=1}^r \rs
  a_i \pi^*\xi_i$ are bounded on a neighbourhood $U'$ of $\cl U$, it
  follows that
  \begin{equation*}
    \RTF{\rs*{\RTFsym}}[\rho]|\rs F|(\eps)[U',\sigma]
    \lesssim~
    \sum_{i=1}^r
    \RTF{\rs*{\RTFsym}}[\rho]|\pi^*\xi_i|(\eps)[U',\sigma]
    \xrightarrow{\eps \tendsto 0^+} 0
  \end{equation*}
  by the assumption on $\xi_i$'s that $\rs g_{i,q'}
  =\res{\rs\Res\paren{\pi^*\xi_i}}_{\lcS*[V,q']} \equiv 0$.
  This gives the contradiction and thus concludes the proof.
\end{proof}

The sets $\lcc<\rs X>(\vphi_L;\psi)$ and $\pi\paren{\lcc<\rs X>(\rs*
  S)}$ are different in general, as is illustrated in the following
examples.
\begin{example} \label{example:lcc-neq-pi_lcc-w-S-neq-pi_S}
  Example \ref{example:S-neq-pi-rs-S} (which is an example having $S
  \subsetneq \pi\paren{\rs* S}$) also provides an example with
  $\lcc<X>(\vphi_L;\psi) \subsetneq \pi\paren{\lcc<\rs X>(\rs*
    S)}$ for $\sigma =1,2$.
  Following the notation and the computation there, one already has
  $\lcc[1]<\rs X>(\rs* S) =\rs S = \rs S_1 + R$, $\lcc[2]<\rs X>(\rs*
  S) =\rs S_1 \cap R$, $\aidlof|0|<X>{0} =\mtidlof<X>{\psi} =
  \genby{z_1 z_2}$ and $\aidlof|2|<X>{0} =\mtidlof<X>{m_0 \psi} =
  \genby{z_2}$. 
  Furthermore, it can also be seen that
  \begin{gather*}
    \divsr{\pi^*z_1} = \rs S_1 \; , \quad
    \divsr{\pi^*z_2} = \rs S_2 + R \; , \quad
    \divsr{\pi^*z_3} = \rs S_3 +R \quad\text{ and } \\
    \aidlof|1|<X>{0} =\pi_*\aidlof|1|<\rs X>{-\phi_R}[\pi^*\psi]
    \qquad
    \begin{aligned}[t]
      &\overset{\mathclap{\text{Thm.~\ref{thm:results-on-residue-fct}
          \eqref{item:thm:res-fct}}}}= \qquad
      \pi_*\paren{\mtidlof<\rs X>{\frac 32 \phi_{\rs S_2}} \cdot
        \defidlof{\rs S_1 \cap R}} \\
      &=\pi_*\paren{\holo_{\rs X} \paren{-\rs* S_2} \cdot
        \defidlof{\rs S_1 \cap R}} =\genby{z_1 z_2 , z_2} =\genby{z_2}
      \; .
    \end{aligned}
  \end{gather*}
  As a result, $\defidlof{\lcc[2]<X>(0;\psi)} =\Ann_{\holo_X}
  \paren{\frac{\aidlof|2|<X>{0}}{\aidlof|1|<X>{0}}} =\holo_X$ and
  $\defidlof{\lcc[1]<X>(0;\psi)} =\Ann_{\holo_X}
  \paren{\frac{\aidlof|1|<X>{0}}{\aidlof|0|<X>{0}}} =\genby{z_1}$,
  i.e.~$\lcc[2]<X>(0;\psi) = \emptyset \subsetneq \pi\paren{\rs* S_1
    \cap R} =\pi\paren{\lcc[2]<\rs X>(\rs* S)}$ and 
  $\lcc[1]<X>(0;\psi) = S_1 \subsetneq \pi\paren{\rs* S_1
    \cup R} =\pi\paren{\lcc[1]<\rs X>(\rs* S)}$.
  Note that $\pi(R) =\set{z_2 =z_3 =0}$ is lying in the ``base locus'' of
  $\aidlof|1|<X>{0}$, i.e.~$\res{\pi^*f}_R \equiv 0$ for all $f \in
  \aidlof|1|<X>{0}$. 
\end{example}

\begin{example} \label{example:lcc-neq-pi_lcc-w-S=pi_S}
  This example provides a simple instance which has $S =\pi(\rs* S)$
  but $\lcc[2]<X>(\vphi_L;\psi) \subsetneq \pi\paren{\lcc[2]<\rs
    X>(\rs* S)}$.
  Let $X := \Delta^2 \subset \fieldC^2$ be the unit $2$-disc centred
  at the origin under the holomorphic coordinate system $(z_1,z_2)$.
  Take
  \begin{equation*}
    \vphi_L :=\log\abs{z_1}^2 +\log\abs{z_2}^2
    \quad\text{ and }\quad
    \psi := \log\paren{\abs{z_1}^2 +\abs{z_2}^2} \; ,
  \end{equation*}
  and consider the modification $\pi \colon \rs X \to X$ which is the
  blow-up at the origin followed by another blow-up at a general point
  on the exceptional divisor (i.e.~a point away from the proper
  transforms of $S_1 :=\set{z_1=0}$ and $S_2 :=\set{z_2=0}$).
  Let $R_1$ be the proper transform of the exceptional divisor of the
  first blow-up and $R_2$ be the exceptional divisor of the second
  blow-up, which gives
  \begin{equation*}
    K_{\rs X} \sim \pi^*K_X +R_1 +2R_2 \; .
  \end{equation*}
  Also let $\rs S_1$ and $\rs S_2$ be the proper transforms of
  $S_1$ and $S_2$ respectively.
  One then has
  \begin{equation*}
    \divsr{\pi^*z_1} =\rs S_1 +R_1 +R_2
    \quad\text{ and }\quad
    \divsr{\pi^*z_2} =\rs S_2 +R_1 +R_2
  \end{equation*}
  and
  \begin{equation*}
    \pi^*_{\ominus}\vphi_L \sim_{\tlog} \phi_{\rs S_1} +\phi_{\rs S_2}
    +\phi_{R_1}
    \quad\text{ and }\quad
    \pi^*\psi \sim_{\tlog} \phi_{R_1} +\phi_{R_2}
  \end{equation*}
  in the notation in Notation \ref{notation:potentials}.
  Note also that $E = 0$.
  This shows that the family $\set{\mtidlof<\rs
    X>{\pi^*_\ominus\vphi_L+m\pi^*\psi}}_{m \in \fieldR_{\geq 0}}$ has
  jumping numbers $m = \mu \in \Nnum$ (which need not be the jumping
  numbers of $\set{\mtidlof<X>{\vphi_L+m\psi}}_{m \in \fieldR_{\geq
      0}}$) with lc locus
  \begin{equation*}
    \rs S =R_1 +R_2
  \end{equation*}
  for every jumping number $\mu$.
  Using Theorem \ref{thm:results-on-residue-fct}
  \eqref{item:thm:res-fct} together with the equality
  $\aidlof<X>{\vphi_L}[\mu\cdot\psi] =\pi_*\aidlof<\rs
  X>{\pi^*_\ominus\vphi_L}[\mu\cdot\pi^*\psi]$ and the fact that the
  ideal sheaves in question are toric (thus generated by monomials),
  the adjoint ideal sheaves for each jumping number $m=\mu$ are
  computed as follows:
  \begin{align*}
    \aidlof|2|<X>{\vphi_L}[\mu\cdot\psi] &
    \begin{aligned}[t]
      &=\pi_*\holo_{\rs X}\paren{-\rs S_1 -\rs S_2 -\mu R_1
        -(\mu-1)R_2} \\
      &=\genby{z_1^p z_2^q \: | \: p \geq 1 \; , \; q \geq 1 \; , \;
        p+q \geq \mu } \; ,
    \end{aligned}
    \\
    \aidlof|1|<X>{\vphi_L}[\mu\cdot\psi] &
    \begin{aligned}[t]
      &=\pi_*\paren{\holo_{\rs X}\paren{-\rs S_1 -\rs S_2 -\mu R_1
        -(\mu-1)R_2} \cdot \defidlof{R_1 \cap R_2}} \\
      &=\genby{z_1^p z_2^q \: | \: p \geq 1 \; , \; q \geq 1 \; , \;
        p+q \geq \mu } \; ,
    \end{aligned}
    \\
    \aidlof|0|<X>{\vphi_L}[\mu\cdot\psi] &
    \begin{aligned}[t]
      &=\pi_*\holo_{\rs X}\paren{-\rs S_1 -\rs S_2 -(\mu+1) R_1
        -\mu R_2} \\
      &=\genby{z_1^p z_2^q \: | \: p \geq 1 \; , \; q \geq 1 \; , \;
        p+q \geq \mu +1 } \; .
    \end{aligned}
  \end{align*}
  It follows that $m=\mu$ is a jumping number of
  $\set{\mtidlof<X>{\vphi_L+m\psi}}_{m \in \fieldR_{\geq 0}}$ if and
  only if $\mu \geq 2$ (note that $\mtidlof<X>{\vphi_L+\psi}
  =\aidlof|0|<X>{\vphi_L} =\aidlof|2|<X>{\vphi_L}
  =\mtidlof<X>{\vphi_L}$).
  Moreover, for $\mu \geq 2$, one has $\defidlof{S}
  =\Ann_{\holo_X}\paren{\frac{
      \aidlof|2|<X>{\vphi_L}[\mu \cdot\psi]
    }{
      \aidlof|0|<X>{\vphi_L}[\mu \cdot\psi]
    }} =\genby{z_1, z_2}$ (which is also equal to
  $\defidlof{\lcc[1]<X>(\vphi_L; \mu\cdot\psi)} =\Ann_{\holo_X}\paren{\frac{
      \aidlof|1|<X>{\vphi_L}[\mu \cdot\psi]
    }{
      \aidlof|0|<X>{\vphi_L}[\mu \cdot\psi]
    }}$) and $\defidlof{\lcc[2]<X>(\vphi_L; \mu\cdot\psi)} =\holo_X$,
  thus $S =\set{(0,0)} =\pi\paren{R_1 \cup R_2} =\pi\paren{\rs* S}$ and
  $\lcc[2]<X>(\vphi_L; \mu\cdot\psi) =\emptyset \subsetneq \set{(0,0)}
  =\pi\paren{R_1 \cap R_2} =\pi\paren{\lcc[2]<\rs X>(\rs* S)}$.
  Note also that $\rs S_0 =\rs S_0^\mu = \mu R_1 +(\mu-1)R_2$ for
  every jumping number $\mu$.
  Writing $\rs\sect_0^\mu$ as a canonical section of $\rs S_0^\mu$,
  one has $\res{\frac{\pi^*f}{\rs\sect_0^\mu}}_{R_1 \cap R_2} \equiv
  0$ for all $f \in \aidlof|2|<X>{\vphi_L}[\mu\cdot\psi]$.
\end{example}

With the above understanding, now it makes sense to give the following definition.

\begin{definition}[$\sigma$-lc centres] \label{def:sigma-lc-centres}
  Suppose that $\vphi_L$ has only neat analytic singularities and that
  $(X,\vphi_L,\psi)$ satisfies the Snc assumption \ref{assum:snc} (but
  need not satisfy \ref{assum:snc-nas}).
  A \emph{$\sigma$-lc centre} $\lcS$ of $(X,\vphi_L,\psi)$
  is an irreducible component of the (reduced) closed
  analytic subset $\lcc(\vphi_L;\psi)$
  of $X$ defined by the ideal sheaf 
  \begin{equation*}
    \defidlof{\lcc(\vphi_L;\psi)}
    :=\Ann_{\holo_X}\paren{
      \frac{\aidlof<X>{\vphi_L}}{\aidlof|\sigma-1|<X>{\vphi_L}}
    } \; .
  \end{equation*}
  For any log-resolution $\pi \colon \rs X \to X$ of
  $(X,\vphi_L,\psi)$ with $\rs S$ being the corresponding snc lc
  locus on $\rs X$, a $\sigma$-lc centre $\lcS*[q] \subset \lcc<\rs
  X>(\rs* S)$ of $(\rs X, \rs S)$ is said to be
  \emph{$\pi$-supportive}\footnote{Such $\sigma$-lc centre in
    $\lcc<\rs X>(\rs* S)$ is ``supporting''
    $\residlof<X>{\vphi_L;\psi}$ via $\pi$.} if it is a subset of
  the closed analytic subset $\lcc<\rs X>(\rs* S)'$ defined as in
  \eqref{eq:def-of-lcc'}, i.e.~there exists $f \in
  \aidlof<X>{\vphi_L}$ such that $\res{\rs\Res(\pi^*f)}_{\lcS*[q]}
  \not\equiv 0$, or equivalently, $\frac{\pi^*f}{\rs\sect_0}
  \not\equiv 0$ on $\lcS*[q]$.
  Theorem \ref{thm:sigma-lc-centres} asserts that 
  \begin{equation*}
    \lcc(\vphi_L;\psi) =\pi\paren{\lcc<\rs X>(\rs* S)'}
    \subset \pi\paren{\lcc<\rs X>(\rs* S)} 
  \end{equation*}
  and $\frac{\pi^* f}{\rs\sect_0} \equiv 0$ on the closure of
  $\lcc<\rs X>(\rs* S) \setminus \lcc<\rs X>(\rs* S)'$ for all $f \in
  \aidlof<X>{\vphi_L}$.
  Note that each $\sigma$-lc centre $\lcS \subset
  \lcc<X>(\vphi_L;\psi)$ must have $\codim_X \lcS \geq \sigma$.
  Define also the \emph{index of the mlc of $(X,\vphi_L,\psi)$},
  denoted by $\sigma_{\mlc} := \sigma_{\mlc}\paren{X,\vphi_L,\psi}$,
  to be the smallest non-negative integer such that $\lcc[\sigma
  +1]<X>(\vphi_L;\psi) =\emptyset$ for all $\sigma \geq
  \sigma_{\mlc}$, which is upper-bounded by, but can possibly be
  different from, the codimension of the mlc of $(\rs X, \rs S)$.
  (The latter is denoted by $\rs\sigma_{\mlc}$ or
  $\rs\sigma_{\mlc}\paren{\rs* X, \rs* S}$ in the rest of this section.)
\end{definition}

\begin{remark} \label{rem:sigma-lcc-explain}
  It is deliberate to call it ``$\sigma$-lc centre'' instead of ``lc
  centre of codimension $\sigma$'', as a $\sigma$-lc centre in $X$ may
  not necessarily be of codimension $\sigma$ in $X$, as indicated in
  \cite{Chan&Choi_ext-with-lcv-codim-1}*{Remark 1.4.2}.
  See also \cite{Chan&Choi_ext-with-lcv-codim-1}*{Example 3.5.1}.
  Note also that the above definition is \emph{different} from (yet a
  refined version of) the ad hoc definition given in
  \cite{Chan&Choi_ext-with-lcv-codim-1}*{Def.~1.4.1} as well as
  \cite{Chan_on-L2-ext-with-lc-measures}*{Def.~1.3.3}.
\end{remark}

Theorem \ref{thm:sigma-lc-centres} shows that the sequence
\begin{equation*}
  \mtidlof[X]{\vphi_L+\psi} 
  =\aidlof|0|<X>{\vphi_L} 
  \subset \aidlof|1|<X>{\vphi_L} 
  \subset \dots \subset \aidlof|\sigma_{\mlc}|<X>{\vphi_L} 
  =\mtidlof[X]{\vphi_L} \; ,
\end{equation*}
contains bimeromorphic information of the $\sigma$-lc centres of
$(X,\vphi_L,\psi)$ for various integers $\sigma \geq 0$.
It is tempting to use this definition to define $\sigma$-lc centres in
the more general situations (for example, when the singularities of
$\vphi_L$ are worse than neat analytic singularities).
However, the coherence of $\aidlof{\vphi_L}$ and the existence of the
residue short exact sequence will then be in question.
This will be studied in subsequent papers.

With such definition of $\sigma$-lc centres, their relation with the
lc locus $S$ of $(X,\vphi_L,\psi)$ is the expected one.
\begin{prop} \label{prop:relation-of-S-and-lcc}
  Under the notation and assumptions in Definition
  \ref{def:sigma-lc-centres} and given the lc locus $S$ of
  $(X,\vphi_L,\psi)$ defined by the ideal sheaf $\defidlof{S}
  =\Ann_{\holo_X} \paren{\frac{
      \aidlof|\sigma_{\mlc}|<X>{\vphi_L}
    }{
      \aidlof|0|<X>{\vphi_L}
    }}$ as in item \eqref{item:def-S} in Section \ref{sec:setup}, one
  has
  \begin{equation*}
    \defidlof{S} = \defidlof{\lcc[1]<X>(\vphi_L;\psi)}
    \cap \defidlof{\lcc[2]<X>(\vphi_L;\psi)} \cap \dotsm
    \cap \defidlof{\lcc[\sigma_{\mlc}]<X>(\vphi_L;\psi)} \; ,
  \end{equation*}
  which is translated to $S =\lcc[1]<X>(\vphi_L;\psi) \cup
  \lcc[2]<X>(\vphi_L;\psi) \cup \dotsm \cup
  \lcc[\sigma_{\mlc}]<X>(\vphi_L;\psi) \;$. 
\end{prop}

\begin{proof}
  \NewDocumentCommand{\fAnn}{
    D<>{X}
    m
    m
  }{\Ann_{\holo_{#1}}\paren{\frac{#2}{#3}}}
  
  Write $\aidlof{}*$ for $\aidlof<X>{\vphi_L}$ and $\lc^{\sigma}$ for
  $\lcc<X>(\vphi_L;\psi)$ for convenience.
  Consider the short exact sequence
  \begin{equation*}
    \renewcommand{\objectstyle}{\displaystyle}
    \xymatrix{
      {0} \ar[r]
      &{\frac{\aidlof|\sigma|{}*}{\aidlof|\sigma-1|{}*}} \ar[r]
      &{\frac{\aidlof|\sigma'|{}*}{\aidlof|\sigma-1|{}*}} \ar[r]
      &{\frac{\aidlof|\sigma'|{}*}{\aidlof|\sigma|{}*}} \ar[r]
      &{0}
    }
  \end{equation*}
  for any $\sigma' \geq \sigma \geq 1$, which is induced from the
  inclusions of the adjoint ideal sheaves.
  As all the maps in the exact sequence are $\holo_X$-homomorphisms,
  it can be checked readily that
  \begin{align*}
    \fAnn{\aidlof|\sigma'|{}*}{\aidlof|\sigma-1|{}*}
    \subset \fAnn{\aidlof|\sigma|{}*}{\aidlof|\sigma -1|{}*}
      \cap \fAnn{\aidlof|\sigma'|{}*}{\aidlof|\sigma|{}*}
    =\defidlof{\lc^{\sigma}}
      \cap \fAnn{\aidlof|\sigma'|{}*}{\aidlof|\sigma|{}*} \; .
  \end{align*}
  By induction, one also sees that
  \begin{equation*}
    \fAnn{\aidlof|\sigma'|{}*}{\aidlof|\sigma-1|{}*}
    \subset \defidlof{\lc^{\sigma}}
    \cap \defidlof{\lc^{\sigma+1}}
    \cap \dotsm \cap \defidlof{\lc^{\sigma'}} \; .
  \end{equation*}
  Putting $\sigma' = \sigma_{\mlc}$ and $\sigma =1$ yields
  $\defidlof{S} \subset \defidlof{\lc^{1}}
  \cap \defidlof{\lc^{2}}
  \cap \dotsm \cap \defidlof{\lc^{\sigma_{\mlc}}}$.

  To obtain the reverse inclusion, take any $h \in \defidlof{\lc^{1}}
  \cap \dotsm \cap \defidlof{\lc^{\sigma_{\mlc}}}$, which means $h \in
  \defidlof{\lc^\sigma} =\fAnn{\aidlof{}*}{\aidlof|\sigma-1|{}*}$ for
  each $\sigma =1, \dots , \sigma_{\mlc}$.
  Therefore, $h^{\sigma_{\mlc}} \in
  \fAnn{\aidlof|\sigma_{\mlc}|{}*}{\aidlof|0|{}*} = \defidlof{S}$.
  As $\defidlof{S}$ is radical by \cite{Demailly_extension}*{Lemma
    4.2}, it follows that $h \in \defidlof{S}$.
  This completes the proof.
\end{proof}

Here is a sufficient condition for the equalities $\sigma_{\mlc} =
\rs\sigma_{\mlc}$ and $\lcc[\sigma_{\mlc}]<X>(\vphi_L;\psi) =
\pi\paren{\lcc[\sigma_{\mlc}]<\rs X>(\rs* S)}$ to hold true, which may
be useful in practice.

\begin{prop} \label{prop:mlc=rs_mlc-in-lc-system}
  Under the notation and assumptions in Definition
  \ref{def:sigma-lc-centres}, suppose that
  $\aidlof<X>{\vphi_L} =\holo_X$ for some $\sigma \geq 0$, or
  equivalently, $\aidlof|\sigma_{\mlc}|<X>{\vphi_L} =\holo_X$.
  Then, for any log-resolution $\pi \colon \rs X \to X$ of
  $(X,\vphi_L,\psi)$, one has $\sigma_{\mlc}\paren{X,\vphi_L,\psi} =
  \rs\sigma_{\mlc}(\rs X, \rs S)$.
  Moreover, one also has $\lcc<X>(\vphi_L;\psi) \setminus \pi\paren{
    \lcc[\sigma+1]<\rs X>(\rs* S)
  }
  =\pi\paren{
    \lcc[\sigma]<\rs X>(\rs* S)
  } \setminus \pi\paren{
    \lcc[\sigma+1]<\rs X>(\rs* S)
  }$, and thus
  $\pi\paren{\lcc<\rs X>(\rs* S)} = \bigcup_{s =
    \sigma}^{\sigma_{\mlc}} \lcc[s]<X>(\vphi_L;\psi)$
  for all integers $\sigma \geq 1$.
\end{prop}

\begin{proof}
  Write $\aidlof{}*$ for $\aidlof<X>{\vphi_L}$, 
  $\rs{\aidlof{}*}$ for $\aidlof<\rs X>{\pi_\ominus^*\vphi_L
    -\phi_E}[\pi^*\psi]$, $\lc^{\sigma}$ for $\lcc<X>(\vphi_L;\psi)$
  and $\rs*{\lc^{\sigma}}$ for $\lcc<\rs X>(\rs* S)$ for convenience.
  Since $\holo_X = \aidlof|\sigma_{\mlc}|{}*
  =\pi_*\rs{\aidlof|\sigma_{\mlc}|{}*}$ implies that $\holo_{\rs X}
  =\rs{\aidlof|\sigma_{\mlc}|{}*} =\rs{\aidlof|\rs\sigma_{\mlc}|{}*}$,
  this forces the equality $\sigma_{\mlc} =\rs\sigma_{\mlc}$ to hold
  true (notice that one must have $\rs{\aidlof|\sigma-1|{}*}
  \subsetneq \rs{\aidlof{}*}$ for all $\sigma =1,\dots,
  \rs\sigma_{\mlc}$). 
  Moreover, it follows that
  \begin{align*}
    \defidlof{\lc^{\sigma_{\mlc}}}
    =\Ann_{\holo_X} \paren{\frac{
        \aidlof|\sigma_{\mlc}|{}*
      }{
        \aidlof|\sigma_{\mlc} -1|{}*
      }}
    &=\aidlof|\sigma_{\mlc}-1|{}* \\
    &=\pi_*\rs{\aidlof|\sigma_{\mlc}-1|{}*}
    \overset{\text{Thm.~\ref{thm:results-on-residue-fct} \eqref{item:thm:res-fct}}}=
    \pi_* \defidlof{\rs*{\lc^{\sigma_{\mlc}}}}
    =\defidlof{\pi\paren{\rs*{\lc^{\sigma_{\mlc}}}}} \; ,
  \end{align*}
  and therefore $\lc^{\sigma_{\mlc}}
  =\pi\paren{\rs*{\lc^{\sigma_{\mlc}}}}$.


  The rest is proceeded via induction.
  Suppose it is known that
  \begin{equation*}
    \pi\paren{\rs{\lc^{\sigma'}}}
    =\lc^{\sigma'} \cup \lc^{\sigma'+1} \cup \dotsm \cup
    \lc^{\sigma_{\mlc}}
    \quad\text{ for } \sigma' =\sigma+1, \sigma+2, \dots,
    \sigma_{\mlc} \; .
  \end{equation*}
  To prove the equality $\lc^{\sigma} \setminus
      \pi\paren{\rs*{\lc^{\sigma+1}}}
  =\pi\paren{\rs*{\lc^{\sigma}}} \setminus
      \pi\paren{\rs*{\lc^{\sigma+1}}}$ (which will then imply
  $\pi\paren{\rs*{\lc^\sigma}} =\bigcup_{s=\sigma}^{\sigma_{\mlc}} \lc^s$),
  consider the short exact sequence
  \begin{equation*}
    \renewcommand{\objectstyle}{\displaystyle}
    \xymatrix{
      {0} \ar[r]
      &{\frac{\aidlof|\sigma'|{}*}{\aidlof|\sigma-1|{}*}} \ar[r]
      &{\frac{\aidlof|\sigma'+1|{}*}{\aidlof|\sigma-1|{}*}} \ar[r]
      &{\frac{\aidlof|\sigma'+1|{}*}{\aidlof|\sigma'|{}*}} \ar[r]
      &{0}
    }
  \end{equation*}
  for any $\sigma' \geq \sigma$ (which is induced from the inclusions
  of the adjoint ideal sheaves).
  Take any point $x \in X \setminus \pi\paren{\rs*{\lc^{\sigma+1}}}$, which implies
  that $x \not\in \lc^{\sigma' +1}$ for all $\sigma' \geq \sigma$ by
  the inductive assumption.
  Since $\paren{\frac{\aidlof|\sigma'+1|{}*}{\aidlof|\sigma'|{}*}}_x =
  0$ for all $\sigma' \geq \sigma$ by the choice of $x$, the short
  exact sequence 
  yields
  $\paren{\frac{\aidlof|\sigma'|{}*}{\aidlof|\sigma-1|{}*}}_x =
  \paren{\frac{\aidlof|\sigma'+1|{}*}{\aidlof|\sigma-1|{}*}}_x$ for
  all $\sigma' \geq \sigma$, which results in
  $\paren{\frac{\aidlof|\sigma|{}*}{\aidlof|\sigma-1|{}*}}_x = 
  \paren{\frac{\aidlof|\sigma_{\mlc}|{}*}{\aidlof|\sigma-1|{}*}}_x$.
  Therefore, following the computation for
  $\defidlof{\lc^{\sigma_{\mlc}}}$ above (which makes use of the
  assumption $\aidlof|\sigma_{\mlc}|{}* =\holo_X$), one has
  \begin{equation*}
    \res{\defidlof{\lc^{\sigma}}}_{X \setminus \pi\paren{\rs*{\lc^{\sigma +1}}}}
    =\res{\defidlof{\pi\paren{\rs*{\lc^{\sigma}}}}}_{X \setminus
      \pi\paren{\rs*{\lc^{\sigma +1}}}} \; ,
  \end{equation*}
  which is translated to $\lc^\sigma \setminus \pi(\rs{\lc^{\sigma+1}})
  =\pi(\rs{\lc^{\sigma}}) \setminus
  \pi(\rs{\lc^{\sigma+1}})$, as desired.
  The claim is thus proved by induction.
\end{proof}

With the terminology of $\sigma$-lc centres, the following theorem can
be stated, which is needed in Theorem
\ref{thm:comparison-alg-and-analytic-aidl}.
Set $\rs S' := \lcc[1]<\rs X>(\rs* S)'$, defined as in
\eqref{eq:def-of-lcc'}, to be the union of all $\pi$-supportive $1$-lc
centres in $\lcc[1]<\rs X>(\rs* S)$ in what follows.

\begin{thm} \label{thm:pi-supportive-1-lc-centres}
  Under the notation and assumptions in Definition
  \ref{def:sigma-lc-centres}, given any $1$-lc centre $\lcS|1|$ of
  $(X,\vphi_L,\psi)$ (which may not be a divisor), the divisor
  $\pi^{-1}\paren{\lcS|1|} \cap \rs S'$ is irreducible, that is, there
  is one and only one $\pi$-supportive $1$-lc centre $\rs D$
  of $(\rs X, \rs S)$ such that $\rs D$ is mapped into, and thus onto,
  $\lcS|1|$ by $\pi$.
\end{thm}

\begin{proof}
  According to Remark \ref{rem:1-residue-sheaf} and
  \eqref{eq:isom-of-2models-of-residl}, one has 
  $\residlof|1|<X>{\vphi_L;\psi} \isom \pi_* \residlof|1|<\rs
  X>{\rs\bphi -\phi_E}$ (which is a consequence of the
  local vanishing theorem (\cite{Lazarsfeld_book-II}*{Thm.~9.4.1}) in
  the algebraic situation or the generalisation of the
  Grauert--Riemenschneider vanishing theorem by Matsumura
  (\cite{Matsumura_injectivity-Kaehler}*{Cor.~1.5}) in the analytic
  situation).
  Let $\rs D \subset \rs S'$ be a $\pi$-supportive $1$-lc centre of
  $(\rs X, \rs S)$ such that $\pi(\rs D) =\lcS|1|$.
  Suppose that $\rs D_1 \subset \rs S'$ is another $\pi$-supportive
  $1$-lc centre such that $\pi(\rs D_1) \subset \lcS|1|$.
  Since $\residlof|1|<\rs X>{\rs\bphi -\phi_E}$ is a direct sum where each of
  its summands is supported on an irreducible component of $\rs S'$, there
  exists a decomposition
  \begin{equation*}
    \residlof|1|<X>{\vphi_L;\psi}
    \isom \pi_* \residlof|1|<\rs X>{\rs\bphi -\phi_E}
    =\pi_* \sheaf S_1 \oplus \pi_* \sheaf S \; ,
  \end{equation*}
  where $\sheaf S_1$ is supported on $\rs D_1$ and $\sheaf S$ is
  supported on the components of $\rs S' -D_1$ (which include $\rs
  D$).
  Note that both summands $\pi_*\sheaf S_1$ and $\pi_*\sheaf S$ are
  non-trivial as both $\rs D_1$ and $\rs D$ are $\pi$-supportive.
  By restricting attention to a local open set in $X$, one can assume that
  $X$ is Stein.
  Let $\rs g \in \sheaf S_1\paren{\pi^{-1}(X)}$ be a
  \emph{non-trivial} holomorphic section and $\rs 0$ be the zero section
  in $\sheaf S\paren{\pi^{-1}(X)}$. Take also $\rs h \in \sheaf
  S\paren{\pi^{-1}(X)}$ such that $\rs h$ is non-trivial on
  \emph{every} irreducible component of $\rs S' -\rs D_1$.

  It follows from the surjectivity of the residue morphism on $X$ (see
  Corollary \ref{cor:residue-exact-seq-for-non-snc}) that there exist
  $f_{\rs 0}, f_{\rs h} \in \aidlof|1|<X>{\vphi_L}\paren{X}$ such
  that $\rs\Res(\pi^*f_{\rs 0}) = (\rs g, \rs 0)$ and $\rs\Res(\pi^*f_{\rs h}) =(\rs
  g, \rs h)$.
  As seen from the formula for $\rs\Res$, these functions satisfy
  \begin{equation*}
    \res{\frac{\pi^*f_{\rs 0}}{\rs\sect_0}}_{\rs
      D_1} =\res{\frac{\pi^*f_{\rs h}}{\rs\sect_0}}_{\rs D_1} \not\equiv
    0 \;\; (\text{as } \rs g \not\equiv 0)
    \; , \quad
    \res{\frac{\pi^*f_{\rs h}}{\rs\sect_0}}_{\rs D} \not\equiv 0
    \;\; (\text{as } \rs h \not\equiv 0)
    \quad\text{ and }\quad
    \res{\frac{\pi^*f_{\rs 0}}{\rs\sect_0}}_{\rs D} \equiv 0
  \end{equation*}
  simultaneously, which in turn yield
  \begin{equation*}
    \res{\pi^*\paren{\frac{f_{\rs 0}}{f_{\rs h}}}}_{\rs D_1} \equiv 1
    \quad\text{ and }\quad
    \res{\pi^*\paren{\frac{f_{\rs 0}}{f_{\rs h}}}}_{\rs D} \equiv 0 \; .
  \end{equation*}
  As $\frac{f_{\rs 0}}{f_{\rs h}}$ is a meromorphic function on $X$,
  the equalities above contradict the inclusion $\pi(\rs D_1) \subset
  \pi(\rs D) =\lcS|1|$.
  This completes the proof.
\end{proof}


\section{Comparison between various algebraic and analytic adjoint ideal
  sheaves}
\label{sec:algebraic-and-analytic-comparison}



\subsection{Ein--Lazarsfeld and Hacon--\McKernan\ adjoint ideal
  sheaves}
\label{sec:def-of-alg-adjoint-ideal-sheaves}

Let $S$ be a reduced divisor (which need not be snc) on a complex manifold $X$.
Let $\pi \colon \rs X \to X$ be any log-resolution of $(X,S)$ such that
$\pi^*S +\Exc(\pi)$ is snc.
There exist a \emph{reduced} snc divisor $\Gamma$ and an exceptional
snc divisor $E^0$, which have no common irreducible
components, such that $E^0$ has no coefficients $=-1$ and
\begin{equation} \label{eq:canonical-div-log-resolution-formula}
  K_{\rs X} +\Gamma
  = K_{\rs X} +\pi^{-1}_*S +\Gamma_{\Exc}
  \sim \pi^*\paren{K_X +S} +E^0 \; ,
\end{equation}
where $\sim$ stands for the linear equivalence relation
between divisors and $\Gamma_{\Exc}$ is an effective reduced
exceptional divisor such that it shares no irreducible components with
the proper transform $\pi^{-1}_*S$ of $S$.
Readers are referred to \cite{Kollar_Sing-of-MMP}*{Def.~2.8} for the
definitions of $(X,S)$ being klt, plt, dlt and lc.
Note that $(X,S)$ being klt (resp.~lc) is equivalent to $\psi_S :=
\phi_S -\sm\vphi_S$ (see Notation \ref{notation:potentials}) having
klt (resp.~lc) singularities as described in Section \ref{sec:setup},
and $(X,S)$ being plt implies that $\Gamma =\pi^{-1}_*S$
(i.e.~$\Gamma_{\Exc} = 0$).
Note also that, \emph{when $(X,S)$ is lc}, since all the divisors
involved in \eqref{eq:canonical-div-log-resolution-formula} are
integral, one has $E^0 \geq 0$.


Let $\ideal a$ be a coherent ideal sheaf on $X$ and assume that $\pi$
is chosen such that $\ideal a \cdot \holo_{\rs X} = \holo_{\rs
  X}\paren{-F}$ for some effective divisors $F$ where $F +\pi^*S
+\Exc(\pi)$ is snc.
The \emph{Hacon--\McKernan\ adjoint ideal sheaf}
$\HMAdjidlof{\ideal a^c}$ for 
any given number $c \geq 0$ (following the definition in
\cite{Hacon&Mckernan}*{Def.-Lemma 4.2} as well as
\cite{Ein-Popa}*{Def.~2.4}) is given by
\begin{equation} \label{eq:definition-alg-HMAdjidl}
  \HMAdjidlof{\ideal a^c}
  \begin{aligned}[t]
    &:=\pi_*\holo_{\rs X} \paren{K_{\rs X} -\pi^*\paren{K_X +S}
      +\Gamma -\floor{c F}} \\
    &=\pi_*\holo_{\rs X} \paren{E^0 -\floor{c F}} \; ,
  \end{aligned}
\end{equation}
where $\floor\cdot$ is the floor function.
When $\ideal a$ is already a principal ideal $\holo_X\paren{-D}$, the
adjoint ideal sheaf $\HMAdjidlof{\ideal a^c}$ is written as
$\HMAdjidlof{c D}$.
It is independent of the choice of log-resolution $\pi$ of
$(X,S,\ideal a)$ as shown in \cite{Hacon&Mckernan}*{Def.-Lemma 4.2}.

Assume further that \emph{$\pi$ is chosen such that $\pi^{-1}_*S$ is
  smooth} (so all irreducible components are disjoint).
The \emph{Ein--Lazarsfeld adjoint ideal sheaf} $\ELAdjidlof{\ideal
  a^c}$ for any given $c \geq 0$ (see
\cite{Ein&Lazarsfeld_adjIdl}*{Prop.~3.1},
\cite{Lazarsfeld_book-II}*{Def.~9.3.47} as well as
\cite{Takagi_alg-adjoint-ideal}*{Def.~1.2 (ii)}) is given by
\begin{equation} \label{eq:definition-alg-ELAdjidl}
  \ELAdjidlof{\ideal a^c}
  \begin{aligned}[t]
    &:=\pi_*\holo_{\rs X} \paren{K_{\rs X} -\pi^*\paren{K_X +S}
      +\pi^{-1}_*S -\floor{c F}} \\
    &=\pi_*\holo_{\rs X} \paren{E^0 -\Gamma_{\Exc} -\floor{c F}} \; .
  \end{aligned}
\end{equation}
It is independent of the choice of log-resolution $\pi$ of $(X,S,\ideal a)$ such
that $\pi^{-1}_*S$ is smooth (see, for example,
\cite{Lazarsfeld_book-II}*{Thm.~9.2.18} or
\cite{Takagi_alg-adjoint-ideal}*{Lemma 1.7 (1)} for a proof).

It can be seen that $\ELAdjidlof{\ideal a^c} \subset
\HMAdjidlof{\ideal a^c}$ in general.
The two adjoint ideal sheaves coincide when $(X,S)$ is plt as seen
from the definitions.
One can construct simple examples to see that the two adjoint ideal
sheaves are different in more general setting.
For example, take $X =\fieldC^2$, $\ideal a = \holo_X$ and $S =
\set{z_2^2 = z_1^2 (z_1 +1)}$ a nodal curve with a node at the
origin (note that $S$ has only $1$ component, yet $(X,S)$ is lc but
not plt).
Take $\pi \colon \rs X \to X$ to be the blow-up of $X$ at the origin
and let $E$ be the exceptional divisor.
One can check that $K_{\rs X} = \pi^*K_X +E$ and $\pi^*S =\pi^{-1}_*S
+2E$, thus $\ELAdjidlof{\holo_X} = \pi_*\holo_{\rs X}\paren{-E} =
\maxidl_{X,0}$ (the defining ideal sheaf of the origin in $X$) while
$\HMAdjidlof{\holo_X} = \pi_*\holo_{\rs X} = \holo_{X}$.

Indeed, for any effective $\fieldQ$-divisor $D$ whose support does not
contain any lc centre of $(X,S)$, $\ELAdjidlof{D}$ is trivial if and
only if $(X, S+D)$ is plt with $\floor D = 0$
(cf.~\cite{Takagi_alg-adjoint-ideal}*{Remark 1.3 (2)} and see
\cite{Kollar_Sing-of-MMP}*{Def.~2.8} instead of
\cite{Takagi_alg-adjoint-ideal}*{Def.~1.1 (ii)} for the definition of
plt).
When $S$ is snc (thus $(X,S)$ is lc and log-smooth), $\HMAdjidlof{D}$
is trivial if and only if $(X, S+D)$ is dlt with $\floor D = 0$ (see
\cite{Hacon&Mckernan}*{Lemma 4.3 (1)}).

\subsection{Analytic and algebraic adjoint ideal sheaves in the lc
  case}
\label{sec:analytic-algebraic-adjIdl-in-lc-case}

To compare the algebraic adjoint ideal sheaves with the current
version of analytic adjoint ideal sheaves, now suppose that 
\emph{$(X,S)$ is lc} (but need not be log-smooth, i.e.~$S$ may not
have snc) and that the \emph{zero locus of $\ideal a$ contains no lc
centres of $(X,S)$ in the sense of
\cite{Kollar_Sing-of-MMP}*{Def.~4.15}}.
Take $L=\holo_X$ and let $\vphi_L := \vphi_{\ideal a^c}$ be a
quasi-psh function on $X$ associated to $\ideal a^c$ such that, given
an open cover $\set{V_\gamma}_\gamma$ of $X$, if $\ideal a$ is
generated on an open set $V_\gamma$ by $g_1, \dots, g_N \in
\holo_{X}\paren{V_\gamma}$, one then has
\begin{equation*}
  \res{\vphi_{\ideal a^c}}_{V_\gamma}
  =\res{c \:\vphi_{\ideal a}}_{V_\gamma}
  =c\log\paren{\abs{g_1}^2 + \dotsm
    +\abs{g_N}^2} +c\beta
\end{equation*}
for some $\beta \in \smooth_X\paren{V_\gamma}$ (so $\vphi_{\ideal a^c}$ has neat
analytic singularities for all $c \geq 0$).
Set also
\begin{equation*}
  \psi := \psi_S := \phi_S -\sm\vphi_S \leq -1
\end{equation*}
(see Notation \ref{notation:potentials}).
As $\vphi_{\ideal a}^{-1}(-\infty)$ contains no lc centres of $(X,S)$ by the
assumption on $\ideal a$, it follows that, for every given constant $c
\geq 0$, the number $m=1$ is a jumping number of the family
$\set{\mtidlof[X]{\vphi_{\ideal a^c} +m\psi}}_{m\in [0,1]}$.
Moreover, there is a number $m^c \in [0,1)$ for each $c \geq 0$ such
that
\begin{equation*}
  \mtidlof[X]{\vphi_{\ideal a^c} +m^c\psi}
  =\mtidlof[X]{\vphi_{\ideal a^c} +m\psi}
  \supsetneq \mtidlof[X]{\vphi_{\ideal a^c} +\psi}
  \quad\text{ for all } m \in [m^c, 1) 
\end{equation*}
and that
\begin{equation*}
  \Ann_{\holo_X}\paren{
    \frac{
      \mtidlof[X]{\vphi_{\ideal a^c} +m^c\psi}
    }{
      \mtidlof[X]{\vphi_{\ideal a^c} +\psi}
    }
  } =\defidlof{S} \; .
\end{equation*}
To see the inclusion $\Ann_{\holo_X}(\dotsm) \subset \defidlof{S}$
(the reverse inclusion is easy to see), on any sufficiently small open
set $V \Subset X$, take any $h \in \Ann_{\holo_X}(\dotsm)\paren{\cl
  V}$ and take for each irreducible component $D$ of $S \cap V$ a
section $f_D \in \mtidlof[X]{\vphi_{\ideal a^c}+m^c\psi}\paren{\cl V}$ such
that $f_D$ does not vanish on $D$ (the existence follows from
assumption on $\ideal a$ and the fact $m^c < 1$).
Since $\vphi_{\ideal a^c}$ is locally bounded from above, one then has
\begin{equation*}
  \int_V \abs{hf_D}^2 e^{-\psi_S} \dvol_X
  \lesssim \int_V \abs{hf_D}^2 e^{-\vphi_{\ideal a^c}-\psi} \dvol_X < +\infty \; ,
\end{equation*}
which implies $h \in \defidlof{S}$ after letting $D$ run through all
components in $S \cap V$.

Given the log-resolution $\pi \colon \rs X \to X$ of $(X,\vphi_{\ideal
  a^c},\psi)$ (also of $(X,S,\ideal a)$), which comes with the
exceptional ($\Znum$-)divisors $E^c$ \footnote{Coincidentally, $E^c$
  with $c=0$ is exactly $E^0$ defined in
  \eqref{eq:canonical-div-log-resolution-formula}.} and $R^c$ defined
as in Section \ref{sec:snc-assumption} (with $\vphi_{\ideal a^c}$ in
place of $\vphi_L$), and assuming that $m^c$ is sufficiently close to
$1$ in view of Proposition \ref{prop:log-resoln-on-jump-subvar}, let
$\rs S^c$ be the \emph{reduced} subvariety defined by
\begin{equation*}
  \defidlof{\rs S^c} := \Ann_{\holo_{\rs X}}\paren{
    \frac{
      \mtidlof[\rs X]{\pi^*_\ominus\vphi_{\ideal a^c} +m^c\pi^*\psi}
    }{
      \mtidlof[\rs X]{\pi^*_\ominus\vphi_{\ideal a^c} +\pi^*\psi}
    }
  }
  \quad\text{ (where $\pi^*_\ominus\vphi_{\ideal a^c}
    :=\pi^*\vphi_{\ideal a^c} -\phi_{R^c}$)}\; ,
\end{equation*}
which is an snc divisor as
\begin{equation}\label{eq:alg-mtidl-formula}
  \mtidlof[\rs X]{\pi^*_\ominus\vphi_{\ideal a^c} +\pi^*\psi}
  =\holo_{\rs X}\paren{-\floor{cF -R^c +\pi^*S}}
\end{equation}
is a principal ideal sheaf with an snc generator and Lemma
\ref{lem:principal-annihilator-criterion} applies.

The divisors $\Gamma$ and $\rs S^c$ are equal for all $c \geq 0$
on the complement of a discrete subset of $\fieldR_{> 0}$ (so $\Gamma
= \rs S^0$ in particular), which can be illustrated in the following
example.

\begin{example} \label{example:compare-alg-ana-adj}
  Let $X :=\Delta^3 \subset \fieldC^3$ be the unit $3$-disc centred at
  the origin in the holomorphic coordinate system $(z_1,z_2,z_3)$.
  Take
  \begin{equation*}
    \ideal a := \genby{z_1, z_2, z_3} \; , \quad
    \vphi_{\ideal a^c} := c \log\paren{\abs{z_1}^2 +\abs{z_2}^2 +\abs{z_3}^2}
    \quad\text{ and }\quad
    \psi :=\log\abs{z_1}^2 +\log\abs{z_2}^2 -1
  \end{equation*}
  such that $m=1$ is a jumping number for the family
  $\set{\mtidlof[X]{\vphi_{\ideal a^c} +m\psi}}_{m\in [0,1]}$ for any $c \geq 0$
  and the corresponding lc locus $S$ is given by $\set{z_1z_2 = 0} =:
  S_1 + S_2$.
  Let $\pi \colon \rs X \to X$ be a log-resolution of $(X,S,\ideal a)$
  obtained by first blowing up the line $\set{z_1=z_2=0}$ in $X$ (with
  exceptional divisor $E_1$) and then blowing up the intersection of
  $E_1$ and the proper transform of $\set{z_3=0} =: S_3$ (with exceptional
  divisor $E_2$).
  Abusing $E_1$ to mean also its proper transform via the second
  blow-up, it follows from a simple computation that
  \begin{gather*}
    K_{\rs X} \sim \pi^*K_X +E_1 +2E_2 \; , \quad
    \pi^*S = \pi^{-1}_*S +2E_1 +2E_2 \; , \\
    \pi^*\vphi_{\ideal a^c} \sim_{\tlog} c\:\phi_{E_2} 
    \quad\text{ and }\quad
    \pi^*\psi \sim_{\tlog} \phi_{\pi^{-1}_*S} +2\phi_{E_1}
    +2\phi_{E_2} \; ,
  \end{gather*}
  in which the notations $\phi_{\pi^{-1}_*S}$, $\phi_{E_1}$ and
  $\phi_{E_2}$ are explained in Notation \ref{notation:potentials},
  and thus
  \begin{gather*}
    R^c = R^0 = E_1+2E_2 \; , \quad
    E^0 = E^c = 0 \quad
    \text{ and } \\
    \Gamma = \pi^{-1}_*S +E_1  \quad(\text{thus }
    \Gamma_{\Exc} =E_1) \; , 
  \end{gather*}
  where $\Gamma$, $\Gamma_{\Exc}$, $E^0$, $E^c$ and $R^c$ are the
  divisors described above.
  Note also that the two irreducible components of $\pi^{-1}_*S = S_1'
  +S_2'$ (where $S_i'$ is the proper transform of $S_i =\set{z_i=0}$
  for $i=1,2,3$) are disjoint in $\rs X$.
  The weight in the $L^2$ norm associated to $\mtidlof[\rs
  X]{\pi^*_\ominus\vphi_{\ideal a^c}+m\pi^*\psi}$ is then of the form
  \begin{equation*}
    e^{-\pi^*_\ominus \vphi_{\ideal a^c} -m\pi^*\psi} \dvol_{\rs X}
    \sim e^{-(c+2m-2)\phi_{E_2} -(2m-1)\phi_{E_1} -m\phi_{\pi^{-1}_*S}} \dvol_{\rs X} \; .
  \end{equation*}
  It can then be seen that, for $c \geq 0$, by setting
  \begin{equation*}
    m^c :=
    \begin{dcases}
      1-\frac{c-\floor{c}}2 & \text{if } c \not\in \Nnum \cup [0,1] \; , \\
      \frac 12 & \text{if } c \in \Nnum \cup [0,1]\; , 
    \end{dcases} 
  \end{equation*}
  one has
  \begin{equation*}
    \mtidlof[\rs X]{\pi^*_\ominus\vphi_{\ideal a^c}+m^c\pi^*\psi}
    =\mtidlof[\rs X]{\pi^*_\ominus\vphi_{\ideal a^c}+m\pi^*\psi}
    \supsetneq \mtidlof[\rs X]{\pi^*_\ominus\vphi_{\ideal a^c}+\pi^*\psi}
    \quad\text{ for all } m \in [m^c,1) \; .
  \end{equation*} 
  Recall that $\defidlof{\rs S^c} =\Ann_{\holo_{\rs X}} \paren{\frac{
      \mtidlof[\rs X]{\pi^*_\ominus\vphi_{\ideal a^c}+m^c\pi^*\psi}
    }{
      \mtidlof[\rs X]{\pi^*_\ominus\vphi_{\ideal a^c}+\pi^*\psi}
    }}$.
  Therefore, one also sees that
  \begin{equation*}
    \rs S^c =
    \begin{dcases}
      \Gamma & \text{if } c \not\in\Nnum=\Nnum_{\geq 1} \; , \\
      \Gamma +E_2 & \text{if } c \in \Nnum \; .
    \end{dcases} 
  \end{equation*} 

  The various adjoint ideal sheaves can now be computed.
  Notice that the ideal sheaves in question
  are toric, so they are generated by monomials.
  A direct use of
  \eqref{eq:definition-alg-HMAdjidl}, \eqref{eq:definition-alg-ELAdjidl}
  and \eqref{eq:alg-mtidl-formula}, together with the equalities
  \begin{equation*}
    \divsr{\pi^*z_1} = S_1' +E_1 +E_2 \; , \quad
    \divsr{\pi^*z_2} = S_2' +E_1 +E_2 \quad\text{ and }\quad
    \divsr{\pi^*z_3} = S_3' +E_2 \; , 
  \end{equation*}
  shows that
  \begin{gather*}
    \begin{aligned}
      \HMAdjidlof{\ideal a^c} &=\pi_*\holo_{\rs X}\paren{-\floor c
        E_2} =\genbyd{z_1^p z_2^q z_3^r}{p,q,r \geq 0 \; , \; p+q+r
        \geq
        \floor c} \; , \\
      \ELAdjidlof{\ideal a^c} &=\pi_*\holo_{\rs X}\paren{-E_1 -\floor
        c E_2} =\genbyd{z_1^p z_2^q z_3^r}{p,q,r \geq 0 \; , \; p+q
        \geq 1 \; , \; p+q+r \geq \floor c} \; ,
    \end{aligned}
    \\
    \begin{aligned}
      \mtidlof<X>{\vphi_{\ideal a^c} +m^c\psi}
      &=\pi_*\paren{\mtidlof<\rs X>{\pi^*_\ominus \vphi_{\ideal a^c}
          +m^c \pi^*\psi}} =
      \begin{dcases}
        \pi_*\holo_{\rs X}\paren{-\floor c E_2}
        & \text{if } c \not\in\Nnum \; , \\
        \pi_*\holo_{\rs X}\paren{-(c-1) E_2} & \text{if } c \in\Nnum
        \; ,
      \end{dcases}
      \\
      \mtidlof<X>{\vphi_{\ideal a^c} +\psi} &=\pi_*\paren{\mtidlof<\rs
        X>{\pi^*_\ominus \vphi_{\ideal a^c} + \pi^*\psi}}
      =\pi_*\holo_{\rs X}\paren{-\Gamma -\floor c E_2} \; .
    \end{aligned}
  \end{gather*}
  As both $\Gamma$ and $\Gamma +E_2$ are reduced snc divisors in $\rs
  X$, their lc centres (in the sense of
  \cite{Kollar_Sing-of-MMP}*{Def.~4.15}, also in the sense of Definition
  \ref{def:sigma-lc-centres} when $(\rs X,0,\phi_\Gamma)$ or
  $(\rs X,0,\phi_{\Gamma+E_2})$ is considered) can be obtained directly from the
  intersections of their irreducible components.
  Recall that $\pi^{-1}_*S =S_1' + S_2'$.
  One sees that $\lcc[3]<\rs X>(\Gamma) =\emptyset$ while
  $\lcc[3]<\rs X>(\Gamma+E_2)$ consists of two distinct points ($S_1'
  \cap E_1 \cap E_2$ and $S_2' \cap E_1 \cap E_2$), and $\lcc[2]<\rs
  X>(\Gamma)$ consists of two disjoint lines ($S_1' \cap E_1$ and
  $S_2' \cap E_1$) while $\lcc[2]<\rs X>(\Gamma +E_2)$ is the union of
  five lines (the extra three lines are obtained from the intersections
  of $E_2$ with $S_1'$, $S_2'$ and $E_1$).
  Using Theorem \ref{thm:results-on-residue-fct} and
  \eqref{eq:log-resoln-aidl} (which says $\aidlof<X>{\vphi_{\ideal
      a^c}} =\pi_*\paren{\mtidlof<\rs X>{\pi^*_\ominus\vphi_{\ideal
        a^c} +m^c \pi^*\psi} \cdot \defidlof{\lcc<\rs X>(\rs*
      S^c)}}$), one obtains 
  \begin{align*}
    \aidlof|3|<X>{\vphi_{\ideal a^c}}
    &
      =
      \begin{dcases}
        \pi_*\holo_{\rs X}\paren{-\floor c E_2}
        =\genbyd{z_1^p z_2^q z_3^r }{ p,q,r \geq 0 \; , \; p+q+r
          \geq \floor c}
        &\text{if } c \not\in \Nnum \; , \\
        \pi_*\holo_{\rs X}\paren{-(c-1) E_2}
        =\genbyd{z_1^p z_2^q z_3^r }{ p,q,r \geq 0 \; , \; p+q+r \geq c-1}
        &\text{if } c \in \Nnum \; ,
      \end{dcases}
    \\
    \aidlof|2|<X>{\vphi_{\ideal a^c}}
    &
      =
      \begin{dcases}
        \pi_*\holo_{\rs X}\paren{-\floor c E_2}
        & \text{if } c \not\in
        \Nnum \; , \\
        \pi_*\paren{\holo_{\rs X}\paren{-(c-1) E_2} \cdot
          \defidlof{\lcc[3]<\rs X>(\Gamma+E_2)}}
        & \text{if } c \in \Nnum \; ,
      \end{dcases}
    \\
    &=
      \begin{dcases}
        \genbyd{z_1^p z_2^q z_3^r }{ p,q,r \geq 0 \; , \; p+q+r
          \geq \floor c}
        & \text{if } c \not\in
        \Nnum \; , \\
        \genbyd{z_1^p z_2^q z_3^r}{
          \begin{aligned}
            &~p,q,r \geq 0 \; , \; p+q \geq 1 \; , \;  p+q+r \geq c-1 \\
            \text{or }\;
            &(p,q,r) = (0,0,c)
          \end{aligned}
        }
        & \text{if } c \in \Nnum \; ,
      \end{dcases}
    \\
    \aidlof|1|<X>{\vphi_{\ideal a^c}}
    &
      =
      \begin{dcases}
        \pi_*\paren{\holo_{\rs X}\paren{-\floor c E_2} \cdot
        \defidlof{\lcc[2]<\rs X>(\Gamma)}}
        & \text{if } c \not\in
        \Nnum \; , \\
        \pi_*\paren{\holo_{\rs X}\paren{-(c-1) E_2} \cdot
          \defidlof{\lcc[2]<\rs X>(\Gamma+E_2)}}
        & \text{if } c \in \Nnum \; ,
      \end{dcases}
    \\
    &=
      \begin{dcases}
        \genbyd{z_1^p z_2^q z_3^r }{ p,q,r \geq 0 \; , \; p+q \geq
          1 \; , \; p+q+r \geq \floor c}
        & \text{if } c \not\in
        \Nnum \; , \\
        \genbyd{z_1^p z_2^q z_3^r}{
          \begin{aligned}
            &p \geq 1 \; , \; q \geq 1 \; , \; r \geq 0 \; , \;  p+q+r \geq c-1 \\
            \text{or }\;
            &p,q,r \geq 0 \; , \; p+q \geq 1 \; , \;  p+q+r \geq c
          \end{aligned}
        }
        & \text{if } c \in \Nnum \; ,
      \end{dcases}
    \\
    \aidlof|0|<X>{\vphi_{\ideal a^c}}
    &
      =\pi_*\holo_{\rs X}\paren{-\Gamma -\floor c E_2}
      =\genbyd{z_1^p z_2^q z_3^r }{ p \geq 1 \; , \; q \geq 1 \; ,
      \; r \geq 0 \; , \; p+q+r \geq \floor c} \; .
  \end{align*}
  It follows directly from Definition \ref{def:sigma-lc-centres} that
  \begin{equation*}
    \defidlof{\lcc[3](\vphi_{\ideal a^c} ; \psi)} =
    \begin{dcases}
      \holo_{\rs X} & \text{if } c \not\in\Nnum \; , \\
      \genby{z_1 , z_2 , z_3} & \text{if } c \in\Nnum \; ,
    \end{dcases}
  \end{equation*}
  and
  \begin{equation*}
    \defidlof{\lcc[2](\vphi_{\ideal a^c} ; \psi)} =\genby{z_1, z_2} 
    \quad\text{ and }\quad
    \defidlof{\lcc[1](\vphi_{\ideal a^c} ; \psi)} =\genby{z_1 z_2}
    \quad\mathrlap{\text{ for any } c \geq 0 \; .} 
  \end{equation*}
  Computing directly the images $\pi\paren{\lcc<\rs X>(\rs* S^c)}$, one
  sees that
  \begin{align*}
    &\left.
      \begin{aligned}
        &\lcc[1](\vphi_{\ideal a^c}; \psi) =\pi\paren{\lcc[1]<\rs
          X>(\rs* S^c)} =S
        \\
        &\lcc[2](\vphi_{\ideal a^c}; \psi) =\pi\paren{\lcc[2]<\rs
          X>(\rs* S^c)} =\set{z_1 =z_2 = 0}
      \end{aligned}
      \quad
    \right\}
          \;\text{ for any } c\geq 0
          \;\;\text{ and } \\
    &\lcc[3](\vphi_{\ideal a^c}; \psi)
      =\pi\paren{\lcc[3]<\rs X>(\rs* S^c)} =
      \begin{dcases}
        \emptyset & \text{if } c \not\in\Nnum \; , \\
        \set{(0,0,0)} & \text{if } c \in\Nnum \; .
      \end{dcases}
  \end{align*}
  This shows the difference between lc centres of $(X, \vphi_{\ideal
    a^c}, \psi)$ in the sense of Definition \ref{def:sigma-lc-centres}
  and those of $(X,S)$ in the sense of
  \cite{Kollar_Sing-of-MMP}*{Def.~4.15}.
  Notice that the zero locus of $\ideal a$ contains the $3$-lc centre
  of $(X,\vphi_{\ideal a^c}, \psi)$ when $c \in \Nnum$.

  Observe that
  \begin{gather*}
    \ELAdjidlof{\ideal a^c}
    \begin{dcases}
      = \aidlof|1|<X>{\vphi_{\ideal a^c}}
      & \text{if } c \not\in \Nnum \text{ or } c=1,2 \; , \\
      \subsetneq \aidlof|1|<X>{\vphi_{\ideal a^c}}
      & \text{if } c \in \Nnum \setminus \set{1,2} \; ,
    \end{dcases}
    \intertext{and}
    \HMAdjidlof{\ideal a^c}
    \begin{dcases}
      =\aidlof|2|<X>{\vphi_{\ideal a^c}}
      =\aidlof|3|<X>{\vphi_{\ideal a^c}}
      & \text{if } c \not\in \Nnum \; , \\ 
      =\aidlof|2|<X>{\vphi_{\ideal a^c}}
      \subsetneq \aidlof|3|<X>{\vphi_{\ideal a^c}}
      & \text{if } c =1 \; , \\
      \subsetneq \aidlof|2|<X>{\vphi_{\ideal a^c}}
      \subsetneq \aidlof|3|<X>{\vphi_{\ideal a^c}}
      & \text{if } c \in \Nnum \setminus \set{1} \; .
    \end{dcases}
  \end{gather*}
  Notice that $\sigma_{\mlc}(X,\vphi_{\ideal a^c},\psi)
  =\rs\sigma_{\mlc}(\rs X, \rs S^c) =
  \begin{dcases}
    2 & \text{if } c \in \fieldR_{\geq 0} \setminus \Nnum \\
    3 & \text{if } c \in \Nnum 
  \end{dcases}
  $ in this example (see Definition \ref{def:sigma-lc-centres} for the notation).
\end{example}


\begin{remark} \label{rem:why-no-lambda}
  In the version of analytic adjoint ideal sheaves given by
  Guenancia in \cite{Guenancia} (see \eqref{eq:Guenancia-adj-ideal}),
  the parameter $\lambda > 1$ (in \eqref{eq:Guenancia-adj-ideal}) is
  introduced precisely to force both the analytic and algebraic adjoint
  ideal sheaves to coincide for all $c \geq 0$.
  However, the description of the germs of sections of the current version of
  adjoint ideal sheaves relies on the jumping numbers (see Sections
  \ref{sec:snc-assumption} and \ref{sec:local-expression-of-germs}),
  while including $\lambda$ in the definition would complicate the
  description.
  Complication from the parameter $\lambda$ also arises when one wants
  to involve log-resolutions in order to handle the non-snc case as in
  Section \ref{sec:non-snc}, especially in cases where the polar set
  of $\pi^*_\ominus\vphi_L=\pi^*_\ominus\vphi_{\ideal a^c}$ contains
  some components of the divisor $\pi^* S$.
  As the residue exact sequence (see Theorem \ref{thm:short-exact-seq}
  and Corollary \ref{cor:residue-exact-seq-for-non-snc}, also cf.~the
  exact sequence in Section \ref{sec:brief-account-on-background}) is
  originally the defining property of adjoint ideal sheaves, and the
  above example (as well as Theorem
  \ref{thm:comparison-alg-and-analytic-aidl} below) shows that the
  algebraic and current versions of adjoint ideal sheaves already
  coincide in almost all cases under consideration, the author takes
  the liberty to remove the parameter $\lambda$ from the definition of
  Guenancia in order to simplify the analysis.
\end{remark}

It turns out that the example is reflecting the relation between the
analytic and algebraic versions of adjoint ideal sheaves in general.

\begin{thm} \label{thm:comparison-alg-and-analytic-aidl}
  Given an lc pair $(X,S)$ and a coherent ideal sheaf $\ideal a$ such
  that the zero locus of $\ideal a$ contains no lc centres of $(X,S)$
  in the sense of \cite{Kollar_Sing-of-MMP}*{Def.~4.15} and
  considering $\vphi_{\ideal a^c}$ and $\psi =\psi_S$ as described at
  the beginning of Section
  \ref{sec:analytic-algebraic-adjIdl-in-lc-case}, one has, for any $c
  \geq 0$,
  \begin{equation*}
    \ELAdjidlof{\ideal a^c} \subset
    \aidlof|1|<X>{\vphi_{\ideal a^c}}
    \quad\text{ and }\quad
    \HMAdjidlof{\ideal a^c} \;
    \subset
    \aidlof|\sigma_{\mlc}|<X>{\vphi_{\ideal a^c}} \; ,
  \end{equation*}
  and there exists a countable discrete subset $N := N(\ideal a, S)
  \subset \fieldR_{> 0}$ (excluding $0$) such that equalities hold for
  both inclusions for any $c \in \fieldR_{\geq 0} \setminus N$.
  Note that the integer $\sigma_{\mlc}$ depends on $c$ and is the
  smallest integer such that $\aidlof|\sigma_{\mlc}|<X>{\vphi_{\ideal
      a^c}} =\mtidlof<X>{\vphi_{\ideal a^c} +m^c \psi} $.

  More specifically, if, for a given $c \geq 0$, the zero locus of
  $\ideal a$ contains no lc centres of $(X,\vphi_{\ideal a^c},\psi)$
  in the sense of Definition \ref{def:sigma-lc-centres}, then $c
  \not\in N$ and therefore, in such case,
  \begin{equation*}
    \ELAdjidlof{\ideal a^c}
    =\aidlof|1|<X>{\vphi_{\ideal a^c}}
    \quad\text{ and }\quad
    \HMAdjidlof{\ideal a^c} \;
    =\aidlof|\sigma_{\mlc}|<X>{\vphi_{\ideal a^c}} \; .
  \end{equation*}
\end{thm}

\begin{proof}
  Taking the log-resolution $\pi \colon \rs X \to X$ of $(X,S,\ideal
  a)$ described at the beginning of this section, it follows from
  the equalities \eqref{eq:canonical-div-log-resolution-formula}
  and $\ideal a \cdot \holo_{\rs X} = \holo_{\rs X}\paren{-F}$ and
  the linear equivalence $K_{\rs X / X} \sim E_{d\pi} = E^c + R^c$ (for any $c
  \geq 0$) that 
  \begin{equation*}
    \pi^*\vphi_{\ideal a^c} \sim_{\tlog} c \:\phi_F
    \quad\text{ and }\quad
    \pi^*\psi =\pi^*\psi_S \sim_{\tlog} \phi_\Gamma -\phi_{E^0} +\phi_{E^c}
    +\phi_{R^c} 
  \end{equation*}
  (see Notation \ref{notation:potentials}).
  Therefore, for any $m \in \fieldR$,
  \begin{equation*} \tag{$*$} \label{eq:pf:potential-at-m-leq-1}
    \pi^*_\ominus \vphi_{\ideal a^c} +m \pi^*\psi
    \sim_{\tlog}
    c \:\phi_F -(1-m)\phi_{R^c} +m\paren{\phi_\Gamma -\phi_{E^0}
      +\phi_{E^c}} \; .
  \end{equation*}
  In particular, one has
  \begin{equation*}
    \pi^*_\ominus \vphi_{\ideal a^c} +\pi^*\psi
    \sim_{\tlog}
    c \:\phi_F +\phi_\Gamma -\phi_{E^0} +\phi_{E^c} \; .
  \end{equation*}

  To determine the divisor $\rs S^c$, note that $R^0 =E^c + R^c -E^0 \geq
  0$ and that $\rs S^c$ is contained inside the polar set of
  $\pi^*\psi$.
  Therefore, $\rs S^c \subset \Gamma \cup \supp\paren{E^c + R^c
    -E^0}$.
  Recall that $R^c$ is, by definition, the maximal sub-divisor of $E_{d\pi}
  =E^c +R^c$ such that $\pi^*_\ominus\vphi_{\ideal a^c} +\pi^*\psi$ is
  quasi-psh,
  so $\rs S^c$ contains no components of $E^c$ (see footnote
  \ref{fn:E-and-S-no-common-comp} on page
  \pageref{fn:E-and-S-no-common-comp}).
  Recall also that the zero locus of $\ideal a$ contains no lc centres of
  $(X,S)$ by assumption, so $F$ and $\Gamma$ have no common
  irreducible components.
  Moreover, $E^0$ also contains no irreducible components of $\Gamma$
  by definition.
  Since a prime divisor $D_i$ in $\rs X$ is contained in $\rs S^c$ if
  and only if $D_i \subset \Gamma \cup \supp\paren{E^c+R^c-E^0}$
  and $\lelong{\pi^*_\ominus\vphi_{\ideal a^c} +\pi^*\psi}[D_i] \in
  \Nnum = \Nnum_{\geq 1}$ (as can be seen from
  \eqref{eq:pf:potential-at-m-leq-1}), it follows that $\Gamma \leq 
  \rs S^c$ and every prime divisor $D_i$ contained in $\rs S^c
  -\Gamma$ must satisfy
  \begin{equation*}
    \lelong{\phi_{E^c}}[D_i] = 0
    \quad\text{ and }\quad
    \begin{aligned}[t]
      &~\lelong{c \:\phi_F +\phi_\Gamma -\phi_{E^0} +\phi_{E^c}}[D_i]
      \\
      =&~\coef_{D_i}\paren{cF +\Gamma -E^0 +E^c} 
      =\coef_{D_i} \paren{c F -E^0}
      &\in \Nnum \; .
    \end{aligned}
  \end{equation*}
  As $\coef_{D_i} \paren{c F -E^0}$, the coefficient
  of $D_i$ in the divisor $c F -E^0$, is continuous in $c$, the
  above conditions hold only for those $c$ sitting inside some
  countable discrete subset $N \subset \fieldR_{> 0}$.
  Such $N$ is chosen such that $\rs S^c =\Gamma$ for $c \in
  \fieldR_{\geq 0} \setminus N$ and $\rs S^c \gneq \Gamma$ for $c \in N$.
  In either case, one can check at this point that
  \begin{equation*}
    \pi(\rs S^c) = S
  \end{equation*}
  as $\pi(\rs S^c) \subset \psi^{-1}(-\infty) =S \subset \pi(\Gamma)$.

  If the zero locus of $\ideal a$ contains no lc centres of
  $(X,\vphi_{\ideal a^c},\psi)$, then $F$ and $\rs S^c$ have no common
  components.
  In this case, $\coef_{D_i}\paren{c F} = 0$ for any component $D_i$
  of $\rs S^c$.
  Thus $c \not\in N$ according to the above criteria (and $\rs S^c
  =\Gamma$ in this case). 

  Recall that all the relevant divisors on $\rs X$ are in snc.
  In view of the isomorphism $\aidlof<X>{\vphi_{\ideal a^c}} \isom
  \pi_*\paren{E^c \otimes \aidlof<\rs X>{\pi^*_\ominus\vphi_{\ideal
        a^c}}[\pi^*\psi]}$, for any $f \in \holo_X$, one is led to
  consider
  \begin{equation*} \tag{$\dagger$} \label{eq:pf:ptwise-RTF-pullback}
    \frac{\abs{\pi^*f \cdot \sect_{E^c}}^2
      \:e^{-\pi^*_\ominus\vphi_{\ideal a^c}
        -\pi^*\psi}}{\logpole|\pi^*\psi|}
    \sim
    \frac{\abs{\pi^*f \cdot \sect_{E^c}}^2
      \:e^{-c\phi_F -\phi_\Gamma +\phi_{E^0}
        -\phi_{E^c}}}{\logpole|\pi^*\psi|} 
    =\frac{\abs{\pi^*f \cdot \sect_{\alert{E^0}}}^2
      \:e^{-c\phi_F -\phi_\Gamma}}{\logpole|\pi^*\psi|} \; .
  \end{equation*}
  This expression is in $\Lloc(\rs X)$ for all $\eps > 0$ if and
  only if $f \in \aidlof<X>{\vphi_{\ideal a^c}}$.

  For the claim concerning the Hacon--\McKernan\ adjoint ideal sheaf,
  let $\rs\sigma_{\mlc} :=\rs\sigma_{\mlc}(\rs X, \rs S^c)$ and
  $\sigma_{\mlc} :=\sigma_{\mlc}\paren{X,\vphi_{\ideal a^c},\psi}$
  (see Definition \ref{def:sigma-lc-centres} for the notation; note
  also that $\rs\sigma_{\mlc} \geq \sigma_{\mlc}$ and
  $\sigma_{\mlc}(\rs X, \rs S^c) \geq \sigma_{\mlc}(\rs X, \Gamma)$
  for general $c \geq 0$).
  Take $\sigma :=\rs\sigma_{\mlc}$.
  As $F$ and $\Gamma$ have no common irreducible components and
  $\frac{e^{-\phi_\Gamma}}{\logpole|\pi^*\psi|/\rs\sigma_{\mlc}/}$ is in
  $\Lloc(\rs X)$ for all $\eps > 0$, it follows from the definition of
  $\HMAdjidlof{\ideal a^c}$ in \eqref{eq:definition-alg-HMAdjidl} and
  the computation in
  \cite{Chan_on-L2-ext-with-lc-measures}*{Prop.~2.2.1} (see also 
  Remark \ref{rem:residue-fct-result-summary}) that $f \in
  \HMAdjidlof{\ideal a^c}$ implies that the expression
  \eqref{eq:pf:ptwise-RTF-pullback} is in $\Lloc(\rs X)$ for all $\eps
  > 0$, thus
  \begin{equation*}
    \HMAdjidlof{\ideal a^c}
    \subset \aidlof|\rs\sigma_{\mlc}|<X>{\vphi_{\ideal a^c}}
    = \aidlof|\sigma_{\mlc}|<X>{\vphi_{\ideal a^c}}
    \quad\text{ for any } c \geq 0 \; .
  \end{equation*}

  Suppose $c \in \fieldR_{\geq 0} \setminus N$ and $f \in
  \aidlof|\sigma_{\mlc}|<X>{\vphi_{\ideal a^c}}
  =\aidlof|\rs\sigma_{\mlc}|<X>{\vphi_{\ideal a^c}}$.
  Since $\Gamma = \rs S^c$ in this case,
  the ideal sheaf $\Ann_{\holo_{\rs X}}
  \paren{\frac{\aidlof|\rs\sigma_{\mlc}|<\rs
      X>{\pi^*_{\ominus}\vphi_{\ideal a^c} -\phi_{E^c}}[\pi^*\psi]
    }{\mtidlof<\rs X>{\pi^*_\ominus\vphi_{\ideal a^c} -\phi_{E^c}
        +\pi^*\psi}}}$ defines $\Gamma$ (Remark
  \ref{rem:defidlof-rs_S-w-phi_E}) and thus
  \begin{equation*}
    \abs{\pi^*f \cdot \sect_{E^0}}^2 \:e^{-c \phi_F}
    \sim \abs{\sect_{\Gamma} \cdot \pi^*f}^2
    \:e^{-\pi^*_\ominus \vphi_{\ideal a^c} +\phi_{E^c} -\pi^*\psi}
    \in \Lloc(\rs X) \; ,
  \end{equation*}
  where $\sect_\Gamma$ is a canonical section of $\Gamma$ and
  $\phi_\Gamma =\log\abs{\sect_\Gamma}^2$. 
  Therefore, one obtains $f \in \HMAdjidlof{\ideal a^c}$ according to
  \eqref{eq:definition-alg-HMAdjidl}.
  Combining with the previous result, it follows that
  \begin{equation*}
    \HMAdjidlof{\ideal a^c} = \aidlof|\sigma_{\mlc}|<X>{\vphi_{\ideal a^c}}
    \quad\text{ for any } c \in \fieldR_{\geq 0} \setminus N \; .
  \end{equation*}

  For the claim concerning the Ein--Lazarsfeld adjoint ideal sheaf,
  take $\sigma = 1$ and assume further that $\pi^{-1}_*S$ is smooth
  (thus having disjoint irreducible components).
  Following the previous arguments, it can be checked directly from
  \eqref{eq:definition-alg-ELAdjidl}, the expression in
  \eqref{eq:pf:ptwise-RTF-pullback} and the computation in
  \cite{Chan_on-L2-ext-with-lc-measures}*{Prop.~2.2.1} that 
  \begin{equation*}
    \ELAdjidlof{\ideal a^c} \subset \aidlof|1|<X>{\vphi_{\ideal a^c}}
    \quad\text{ for any } c \geq 0 \; .
  \end{equation*}
  It remains to show the equality for any $c \in \fieldR_{\geq 0}
  \setminus N$.
  Take any $f \in \aidlof|1|<X>{\vphi_{\ideal a^c}} \subset
  \aidlof|\sigma_{\mlc}|<X>{\vphi_{\ideal a^c}}$.
  As $f \in \HMAdjidlof{\ideal a^c}$ when $c \in \fieldR_{\geq 0}
  \setminus N$ by the previous result, it suffices to show that
  $\abs{\pi^*f \cdot \sect_{E^0}}^2 \:e^{-c\phi_F}$ (and hence
  $\pi^*f$) vanishes along $\Gamma_{\Exc} =\Gamma -\pi^{-1}_*S$
  according to \eqref{eq:definition-alg-ELAdjidl}
  (note that $\Gamma =\rs S^c$ has no common components with $F$ and
  $E^0$).
  Since $S$ is a divisor, Proposition
  \ref{prop:relation-of-S-and-lcc}, together with the fact that
  $\codim_X \lcc<X>(\vphi_{\ideal a^c};\psi) \geq \sigma$ for all
  $\sigma \geq 1$, implies that the irreducible components of $S$ are
  all $1$-lc centres of $(X,\vphi_{\ideal a^c}, \psi)$ in the sense of
  Definition \ref{def:sigma-lc-centres}.
  Being exceptional, each component of $\Gamma_{\Exc}$ must have its
  image under $\pi$ \emph{properly} sitting inside a component of
  $S$.
  According to Theorem \ref{thm:pi-supportive-1-lc-centres},
  components of $\Gamma_{\Exc}$ cannot be $\pi$-supportive $1$-lc
  centres (see Definition \ref{def:sigma-lc-centres}), which implies
  that $\pi^*f$ vanishes on $\Gamma_{\Exc}$, as desired. \qedhere

\end{proof}

\begin{remark}
  When $S$ is of higher pure codimension in $X$, one can still
  construct the corresponding global function $\psi$ such that $S
  \subset \psi^{-1}(-\infty)$ and such that $S$ is the lc locus of the
  family $\set{\mtidlof<X>{m\psi}}_{m\in [0,1]}$
  at jumping number $m=1$.
  Assuming the zero locus of the coherent ideal sheaf $\ideal a$
  contains no lc centres of $(X,0,\psi)$ (in the sense of Definition
  \ref{def:sigma-lc-centres}), one can carry out the same analysis to
  compare $\aidlof<X>{\vphi_{\ideal a^c}}$ with the adjoint ideal
  sheaf introduced by Takagi in \cite{Takagi_alg-adjoint-ideal}, which
  is a generalisation of the Ein--Lazarsfeld adjoint ideal sheaf.
\end{remark}

\begin{cor} \label{cor:aidl-1-and-plt}
  Given an lc pair $(X,S)$ and a $\fieldQ$-divisor $D$ in $X$ such
  that $\supp D$ contains no lc centres of $(X,S)$ in the sense of
  \cite{Kollar_Sing-of-MMP}*{Def.~4.15} and considering the potential
  $\phi_D$ (see Notation \ref{notation:potentials}) and the function
  $\psi = \psi_S$ as described at the beginning of Section
  \ref{sec:analytic-algebraic-adjIdl-in-lc-case}, the pair $(X,S+D)$
  is a plt pair with $\floor D =0$ if and only if $\aidlof|1|<X>{c
    \phi_D} =\holo_X$ for some $c > 1$.
  
  If $\supp D$ contains no lc centres of $(X,\phi_D,\psi)$ in the
  sense of Definition \ref{def:sigma-lc-centres}, then  $(X,S+D)$
  is a plt pair with $\floor D =0$ if and only if $\aidlof|1|<X>{\phi_D}
  =\holo_X$.
\end{cor}

\begin{proof}
  It is known that $(X,S+D)$ is a plt pair with $\floor D =0$ if and
  only if $\ELAdjidlof{D} =\holo_X$
  (cf.~\cite{Takagi_alg-adjoint-ideal}*{Remark 1.3 (2)} and notice the
  difference between the definitions of ``plt'' in
  \cite{Kollar_Sing-of-MMP}*{Def.~2.8} and in
  \cite{Takagi_alg-adjoint-ideal}*{Def.~1.1 (ii)}).
  When $\supp D$ contains no lc centres of $(X,\phi_D,\psi)$ (in the
  sense of Definition \ref{def:sigma-lc-centres}), Theorem
  \ref{thm:comparison-alg-and-analytic-aidl} implies that
  $\ELAdjidlof{D} =\aidlof|1|<X>{\phi_D}$ and thus the last claim
  follows.

  Now consider the case when $\supp D$ is only assumed to contain no
  lc centres of $(X,S)$ in the sense of
  \cite{Kollar_Sing-of-MMP}*{Def.~4.15}.
  Notice that $(X, S+D)$ is plt with $\floor D =0$ if and only if
  there exists some constant $c > 1$ such that $(X, S+c'D)$ is plt with
  $\floor{c'D} =0$ for all $c' \in [1,c]$.
  By Theorem \ref{thm:comparison-alg-and-analytic-aidl}, one can also
  choose $c > 1$ sufficiently close to $1$ such that $\ELAdjidlof{c'D}
  =\aidlof|1|<X>{c'\phi_D}$ for all $c' \in (1,c]$.
  Note also that $\aidlof|1|<X>{c\phi_D} \subset
  \aidlof|1|<X>{c'\phi_D}$ for any $c' \in (1, c]$ (in particular,
  $\aidlof|1|<X>{c'\phi_D} =\holo_X$ when $\aidlof|1|<X>{c\phi_D}
  =\holo_X$).
  The remaining claim then follows immediately.
\end{proof}

By altering the formulation a little bit, one can obtain a more direct
result for determining whether a pair is plt or lc.
\begin{thm} \label{thm:aidl-and-singularities-in-MMP}
  Let $\Delta$ be an effective $\fieldQ$-divisor on $X$ and consider
  the function $\psi_\Delta$ as described in Notation
  \ref{notation:potentials}.
  Then,
  \begin{enumerate}
  \item the pair $(X,\Delta)$ is klt if and only if
    $\aidlof|0|<X>{0}[\psi_\Delta] =\holo_X$;

  \item the pair $(X,\Delta)$ is plt if and only if
    $\aidlof|1|<X>{0}[\psi_\Delta] =\holo_X$ and
    $\aidlof|0|<X>{0}[\psi_\Delta] =\defidlof{\floor{\Delta}}$;
    
  \item the pair $(X,\Delta)$ is lc if and only if
    $\aidlof<X>{0}[\psi_\Delta] =\holo_X$ for some (sufficiently
    large) integer $\sigma \geq 0$ (and it suffices to consider only
    $\sigma \in [0,n]$).
  \end{enumerate}
\end{thm}

\begin{proof}
  The first claim is trivial as $\aidlof|0|<X>{0}[\psi_\Delta]
  =\mtidlof<X>{\psi_\Delta}$.

  For the rest of the claims, take any log-resolution $\pi \colon \rs
  X \to X$ of $(X,\Delta)$ and the corresponding effective divisors
  $E$ and $R$ as in Section \ref{sec:snc-assumption} such that $K_{\rs
    X / X} \sim E+R$ and $R$ is the maximal divisor with
  $\pi^*\psi_\Delta -\phi_R$ being quasi-psh.
  Following the notation at the beginning of Section
  \ref{sec:non-snc} and also Notation \ref{notation:potentials}, one
  has
  \begin{equation*}
    \pi^*\psi_\Delta -\phi_R =\rs\bphi +\phi_{\rs S_0} +\phi_{\rs S}
    \; ,
  \end{equation*}
  where $\rs S$ is the lc locus of the family $\set{\mtidlof<\rs
    X>{m\pi^*\psi_\Delta -\phi_R}}_{m\in [0,1]}$ at the jumping number
  $m=1$ (which is a reduced snc divisor, and $\rs S = 0$ if $m=1$ is
  not a jumping number), and $\rs S_0$ is a divisor with $\supp \rs
  S_0 \subset \rs S$ such that $e^{-\rs\bphi}$ is locally integrable at
  general points of $\rs S$.
  Indeed, it can be seen that, if $\pi^*\Delta -R =\sum_{i \in I} c_i
  D_i$ where $D_i$'s are distinct prime divisors and $c_i \in
  \fieldQ_{> 0}$ for all $i \in I$, one has
  \begin{equation} \label{eq:pf:rs_S-from-Delta}
    \rs S_0 + \rs S =\sum_{\substack{i \in I \colon \\ \alert{c_i \in \Nnum} }}
    c_i D_i \; .
  \end{equation}
  Writing $\rs B := \pi^*\Delta -R -\rs S_0 -\rs S$ (which is an
  effective $\fieldQ$-divisor whose coefficients are all
  non-integral), it follows that 
  \begin{equation*}
    \tag{$*$} \label{eq:pf:canonical-div-formula-general-bdry}
    K_{\rs X} +\rs S_0 +\rs S +\rs B \sim_\fieldQ \pi^*\paren{K_X
      +\Delta} + E
    \quad\text{ and }\quad \rs\bphi \sim_{\tlog} \phi_{\rs B} \; .
  \end{equation*}
  Note that $\rs S$ and $E$ have no common irreducible components.
  If $\rs B$ and $E$ have a common component $D_i$, then $\coef_{D_i}
  (\rs B) \in (0,1)$ by the definition of $E$ (and hence $R$).
  Let $m_0 \in [0,1)$ be a number such that $\mtidlof<\rs X>{m
    \pi^*\psi_\Delta -\phi_R} =\mtidlof<\rs X>{m_0
    \pi^*\psi_\Delta -\phi_R} \supsetneq \mtidlof<\rs
  X>{\pi^*\psi_\Delta -\phi_R}$ for all $m \in [m_0, 1)$. 

  Suppose $(X,\Delta)$ is plt (but not klt).
  From \eqref{eq:pf:canonical-div-formula-general-bdry} and the
  definition of plt in \cite{Kollar_Sing-of-MMP}*{Def.~2.8}, it
  follows that $\rs S_0 = 0$ and $\floor{\rs* B} =0$, which also
  implies that $S :=\floor\Delta$ is a reduced divisor and $\rs S
  =\pi^{-1}_*S$ (which is $\neq 0$, as $(X,\Delta)$ would then be klt
  otherwise).
  Assume that $\pi^{-1}_*S$ is smooth (possibly a union of disjoint
  components) by choosing $\pi$ suitably.
  It follows easily from Theorem \ref{thm:results-on-residue-fct} (and
  Remark \ref{rem:residue-fct-result-jumping-no}) that
  \begin{equation*}
    \aidlof|1|<\rs X>{-\phi_R}[\pi^*\psi_\Delta]
    =\mtidlof<\rs X>{m_0\pi^*\psi_\Delta -\phi_R} \cdot
    \underbrace{\defidlof{\lcc[2]<\rs X>(\rs* S)}}_{
      = \holo_{\rs X} 
      \mathrlap{\text{ as } \lcc[2]<\rs X>(\rs* S) =\emptyset}
    }
    =\mtidlof<\rs X>{m_0 \paren{\phi_{\rs B} +\phi_{\rs S}} -(1-m_0) \phi_R}
    =\holo_{\rs X} \; ,
  \end{equation*}
  which implies that $\aidlof|1|<X>{0}[\psi_\Delta]
  \isom \pi_*\paren{E \otimes \aidlof|1|<\rs X>{-\phi_R}[\pi^*\psi_\Delta]}
  \isom \holo_X$.
  Since $\pi(\rs S) =S$, Proposition
  \ref{prop:mlc=rs_mlc-in-lc-system} guarantees that $S
  =\lcc[1]<X>(0;\psi_\Delta)$ and, therefore,
  $\aidlof|0|<X>{0}[\psi_\Delta] =\Ann_{\holo_X}\paren{\frac{
    \aidlof|1|<X>{0}[\psi_\Delta]
  }{
    \aidlof|0|<X>{0}[\psi_\Delta]
  }} =\defidlof{S} =\defidlof{\floor\Delta}$.

  Conversely, suppose $\aidlof|1|<X>{0}[\psi_\Delta] =\holo_X$ and
  $\aidlof|0|<X>{0}[\psi_\Delta] =\defidlof{\floor\Delta}$, which
  implies immediately that $\sigma_{\mlc}
  =\sigma_{\mlc}(X,0,\psi_\Delta) \leq 1$ (and $=1$ if and only if
  $\defidlof{\floor\Delta} \neq \holo_X$).
  Furthermore, one has
  \begin{equation*}
    \holo_{\rs X} = \aidlof|1|<\rs X>{-\phi_R}[\pi^*\psi_\Delta]
    =\mtidlof<\rs X>{m\paren{\phi_{\rs B} +\phi_{\rs S_0} +\phi_{\rs
          S}} -(1-m) \phi_R} \cdot \defidlof{\lcc[2]<\rs X>(\rs* S)}
  \end{equation*}
  for any $m \in [m_0, 1)$, thus $\mtidlof<\rs X>{m\paren{\phi_{\rs B} +\phi_{\rs
        S_0} +\phi_{\rs S}} -(1-m) \phi_R} =\holo_{\rs X}$ (and
  $\defidlof{\lcc[2]<\rs X>(\rs* S)} =\holo_{\rs X}$).
  Choosing $m$ sufficiently close to $1$ from the left-hand-side, the
  latter equality implies that $\rs S_0 =0$ and $\floor{\rs* B}=0$
  (recall that $\supp\rs S_0 \subset \rs S$ and all coefficients in
  $\rs B$ are non-integral).
  Set again $S := \floor\Delta$, which is reduced.
  One then has $\pi(\rs S) \supset S$ as seen from the construction (see
  \eqref{eq:pf:rs_S-from-Delta}), and $\pi(\rs S) = S$ thanks to the
  assumption $\aidlof|0|<X>{0}[\psi_\Delta] =\defidlof{S}$.
  Proposition \ref{prop:mlc=rs_mlc-in-lc-system} guarantees that
  $S =\lcc[1]<X>(0;\psi_{\Delta})$.
  Since the constant function $1$, which is nowhere vanishing, belongs
  to $\aidlof|1|<X>{0}[\psi_\Delta] =\holo_X$, every irreducible
  component of $\rs S$ are $\pi$-supportive (see Definition
  \ref{def:sigma-lc-centres}).
  Theorem \ref{thm:pi-supportive-1-lc-centres} then guarantees
  that $\rs S =\pi^{-1}_*S$.
  Collecting these results, one observes from
  \eqref{eq:pf:canonical-div-formula-general-bdry} that $(X,\Delta)$
  is plt.

  When $(X,\Delta)$ is lc, one also has $\rs S_0 =0$ and $\floor{\rs*
    B}=0$.
  Choosing an integer $\sigma \in [0,n]$ such that $\lcc[\sigma+1]<\rs
  X>(\rs* S) =\emptyset$, the above computation yields
  $\aidlof<X>{0}[\psi_\Delta] =\holo_X$.
  Conversely, when $\aidlof<X>{0}[\psi_\Delta] =\holo_X$ for some
  $\sigma \geq 0$, the above argument shows again that $\rs S_0 =0$
  and $\floor{\rs* B}\vphantom{\rs B} =0$, which implies that
  $(X,\Delta)$ is lc as observed from
  \eqref{eq:pf:canonical-div-formula-general-bdry}.
  This completes the proof.
\end{proof}

Another proof of \textfr{Kollár}'s theorem on the inversion of
adjunction for the case when the base space $X$ is smooth via the use
of the residue exact sequences can now be given.
\begin{thm}[Inversion of adjunction (Theorem
  \ref{thm:inversion-of-adjunction-intro});
  see \cite{Kollar_AST_1992}*{Thm.~17.6} and
  cf.~\cite{KimDano&Seo_adj-idl}*{Thm.~1.7}]
  \label{thm:inversion-of-adjunction}
  On a complex manifold $X$, let $\Delta$ be an effective
  $\fieldQ$-divisor on $X$ such that $S :=\floor\Delta$ is a reduced
  divisor.
  Also let $\nu \colon S^\nu \to S$ be the normalisation of $S$ and
  $\Diff_{S^\nu}\Delta$ the general different such that $K_{S^\nu}
  +\Diff_{S^\nu}\Delta \sim_\fieldQ \nu^*\parres{K_X +\Delta}_{S}$.
  Then, $(X, \Delta)$ is plt near $S$ if and only if $(S^\nu,
  \Diff_{S^\nu}\Delta)$ is klt.
\end{thm}

\begin{proof}
  Write $S = \sum_{i \in I} D_i$ and $S^\nu =\sum_{i \in I} D_i^\nu$
  such that each $D_i$ is an irreducible divisor and
  $\res\nu_{D_i^\nu} \colon D_i^\nu \to D_i$ is the normalisation of
  $D_i$.
  One then has $\Diff_{D_i^\nu} \Delta =\parres{\Diff_{S^\nu}
    \Delta}_{D_i^\nu}$ for each $i \in I$.
  Take any log-resolution $\pi \colon \rs X \to X$ of $(X,\psi_\Delta)$.
  Following the notation in the proof of Theorem
  \ref{thm:aidl-and-singularities-in-MMP}, one has
  \begin{gather*}
    \pi^*\psi_\Delta -\phi_R =\rs\bphi +\phi_{\rs S_0} +\phi_{\rs S}
    \sim_{\tlog} \phi_{\rs B} +\phi_{\rs S_0} +\phi_{\rs S}
    \mathrlap{\quad\text{ and}} \\
    K_{\rs X} +\rs S_0 +\rs S +\rs B \sim_\fieldQ \pi^*\paren{K_X
      +\Delta} + E \; .
  \end{gather*}
  Recall that $\rs B$ is effective and all of its coefficients are
  non-integral, and $\supp\rs S_0 \subset \rs S$.
  Since $S =\floor\Delta$ is reduced, the divisor $\rs S_0$ is
  $\pi$-exceptional, and $\rs S \supset \pi^{-1}_* S$ as can be
  observed from \eqref{eq:pf:rs_S-from-Delta}.
  Therefore, by writing $\pi^{-1}_*S =\sum_{i \in I} \rs D_i$, one has
  \begin{gather*}
    \Diff_{D_i^\nu} \Delta
    =\pi'_*\paren{\Diff_{\rs D_i}\paren{\rs* S_0 +\rs* S +\rs* B}}
    =\pi'_*\paren{\Diff_{\rs D_i}\paren{\rs* S} +\res{\rs* B}_{\rs
        D_i}}
    \mathrlap{\quad\text{ and}} \\
    \label{eq:pf:canonical-formula-on-rs-D_i} \tag{$\dagger$}
    K_{\rs D_i} +\Diff_{\rs D_i} \rs S +\res{\rs* S_0}_{\rs D_i}
    +\res{\rs* B}_{\rs D_i}
    \:\sim_{\fieldQ}\:
    \pi'^*\paren{K_{D_i^\nu} +\Diff_{D_i^\nu} \Delta}
    +\res E_{\rs D_i} \; ,
  \end{gather*}
  where $\pi'$ is a log-resolution of $(S^\nu,
  \Diff_{S^\nu} \Delta)$ obtained from the factorisations $\res\pi_{\rs
    D_i} = \res\nu_{D_i^\nu} \circ \res{\pi'}_{\rs D_i}$.
  Moreover, consider the residue short exact sequence in Corollary
  \ref{cor:residue-exact-seq-for-non-snc} with $\sigma := 1$, $\vphi_L
  := 0$ and $\psi :=\psi_\Delta$ given as in Notation
  \ref{notation:potentials}.
  Write $\rs S =\lcc[1]<\rs X>(\rs* S) =\sum_{p \in \rs I} \lcS*|1|$.
  According to Remark \ref{rem:1-residue-sheaf},
  \eqref{eq:isom-of-2models-of-residl} and discussion in Section
  \ref{sec:residue-map}, one has
  \begin{align*}
    \residlof|1|<X>{0;\psi_\Delta} 
    &=\pi_* \residlof|1|<\rs X>{\rs\bphi -\phi_E}
    =\bigoplus_{p \in \rs I} \pi_*\sheaf S_{\lcS*|1|} \\
    &=\bigoplus_{p\in \rs I} \paren{\res\pi_{\lcS*|1|}}_*\paren{
      \res{\rs* S_0^{-1}}_{\lcS*|1|} \otimes
      \paren{\Diff_{\lcS*|1|} \rs S}^{-1} \otimes
      \mtidlof<\lcS*|1|>{\rs*\bphi -\phi_E}
      } \\
    &=\bigoplus_{p\in \rs I} \paren{\res\pi_{\lcS*|1|}}_*\paren{
      \res{\rs* S_0^{-1}}_{\lcS*|1|} \otimes
      \paren{\Diff_{\lcS*|1|} \rs S}^{-1} \otimes
      \mtidlof<\lcS*|1|>{\phi_{\rs B} -\phi_E}
      } \; .
  \end{align*}

  Suppose that $(X,\Delta)$ is plt near $S$.
  By shrinking $X$ to a neighbourhood of $S$, one can assume that
  $(X,\Delta)$ is plt everywhere. 
  The proof of Theorem \ref{thm:aidl-and-singularities-in-MMP} shows
  that $\rs S =\pi^{-1}_*S$ (thus $\setd{\rs D_i}{i \in I} =
  \setd{\lcS*|1|}{p \in \rs I}$), $\rs S_0 = 0$, $\floor{\rs* B} = 0$
  and $\Diff_{\lcS*|1|} \rs S = 0$ (as $\lcc[2]<\rs X>(\rs* S)
  =\emptyset$, i.e.~irreducible components of $\rs S$ are mutually
  disjoint).
  Therefore, $(S^\nu , \Diff_{S^\nu} \Delta)$ is klt, as can be seen
  from \eqref{eq:pf:canonical-formula-on-rs-D_i}.

  Now suppose that $(S^\nu , \Diff_{S^\nu} \Delta)$ is klt.
  It follows from \eqref{eq:pf:canonical-formula-on-rs-D_i} that
  $\res{\rs* S_0}_{\rs D_i} = 0$, $\Diff_{\rs D_i}\rs S = 0$ (note that
  $\rs S_0$ and $\rs S$ are effective $\Znum$-divisors) and
  $\floor{\res{\rs* B}_{\rs D_i}} = 0$ for all $i \in I$.
  The general different $\Diff_{\rs D_i}\rs S$ being trivial implies
  that $\rs D_i$ is disjoint from any other components of $\rs S$.
  Recall that $\rs D_i \in \setd{\lcS*|1|}{p \in \rs I}$ for all $i
  \in I$.
  It follows that the summand of $\residlof|1|<X>{0;\psi_\Delta}$
  supported on $\rs D_i$ is given by $\pi_*\sheaf S_{\rs D_i} =
  \paren{\res\pi_{\rs D_i}}_* \holo_{\rs D_i} \neq 0$, so $\rs D_i$ is
  $\pi$-supportive for each $i \in I$ (see Definition
  \ref{def:sigma-lc-centres}) and therefore $\lcc[1]<X>(0;\psi_\Delta)
  \supset S$. 
  On any open polydisc $V \Subset X$ (in some coordinate chart) such
  that $V \cap S \neq \emptyset$, take any $\rs g =\paren{\rs g_p}_{p
    \in \rs I} \in \residlof|1|<\rs X>{\rs\bphi -\phi_E}(\pi^{-1}(V))$
  such that the component $\rs g_{\rs D_i}$ of $\rs g$ on $\rs D_i$
  for each $i \in I$ (such that $D_i \cap V \neq \emptyset$) is a
  nowhere vanishing function (for example, a constant function), while
  the other components $\rs g_p$ are chosen to be $0$ on $\lcS*[p]
  \subset \rs S -\pi^{-1}_*S$.
  Surjectivity of the residue morphism guarantees there exists $f \in
  \aidlof|1|<X>{0}[\psi_\Delta](V)$ such that $\rs\Res(\pi^*f) =\rs g$.
  Note that, given $\res{\rs* S_0}_{\rs D_i} =0$ and $\Diff_{\rs
    D_i}\rs S = 0$, the image of the residue morphism $\rs\Res$ on
  $\rs D_i$ is given simply by restriction $\pi^*f \mapsto
  \res{\pi^*f}_{\rs D_i} =u_i \rs g_{\rs D_i} \not\equiv 0$ up to a multiple factor
  of a nowhere vanishing function $u_i$ on $\rs D_i$.
  This indeed implies that the components $D_i =\pi(\rs D_i)$ of $S$
  has to be disjoint from one another as well as from each $\pi(\lcS*|1|)$
  for $\lcS*|1| \subset \rs S -\pi^{-1}_*S$ (as one has
  $\res{\pi^*f}_{\lcS*|1|} \equiv 0$ on $1$-lc centres $\lcS*|1|
  \subset \rs S -\pi^{-1}_*S$ by the choice of $\rs g$, regardless of
  whether $\res{\rs* S_0}_{\lcS*|1|}$ and $\Diff_{\lcS*|1|} \rs S$ are
  trivial or not).
  This implies that there exists an open neighbourhood $W \subset X$ of $S$
  disjoint from $\pi(\rs S) \setminus S$ such that $f$ is
  nowhere vanishing on $W \cap V$, and hence
  $\aidlof|1|<X>{0}[\psi_\Delta] = \holo_X$ on such neighbourhood
  $W$.
  
  With such choice of $W$, one has $\res{\pi(\rs* S)}_W =S$ and that
  every $1$-lc centre in $\res{\rs* S}_{\pi^{-1}(W)}$ is
  $\pi$-supportive (since $\aidlof|1|<X>{0}[\psi_\Delta] =\holo_X$ on
  $W$).
  In view of Theorem \ref{thm:pi-supportive-1-lc-centres}, one has
  $\res{\rs* S}_{\pi^{-1}(W)} =\pi^{-1}_*S$. 
  It thus follows that $\aidlof|0|<X>{0}[\psi_\Delta] =\defidlof{S}$ on $W$.
  Theorem \ref{thm:aidl-and-singularities-in-MMP} then assures that
  $(X,\Delta)$ is plt on that neighbourhood.
\end{proof}

\subsection{Further examples}
\label{sec:aidl-Adjidl-different}

The first example is to show that the adjoint ideal sheaf introduced
by Guenancia is different from the Ein--Lazarsfeld adjoint ideal sheaf
in general.

\begin{example}[Guenancia's adjoint ideal sheaf] \label{example:Guenancia-neq-EL}
  Guenancia's adjoint ideal sheaf in its original form (see
  \cite{Guenancia}*{Def.~2.7 and Def.~2.10} or footnote
  \ref{fn:Guenancia-Adjidl} on page \pageref{fn:Guenancia-Adjidl})
  may not coincide with the Ein--Lazarsfeld adjoint ideal sheaf as
  claimed (note that \cite{Guenancia}*{Def.~2.1}, which is corrected
  in \cite{Guenancia_AdjIdl-Erratum}, is the definition of
  $\ELAdjidlof{\ideal a^c}$) when the reduced snc divisor $S$ has more than one
  component, as can be illustrated from the example in Example
  \ref{example:the-cross-in-2-disc}.
  Recall that $X = \Delta^2 \subset \fieldC^2$ is the unit $2$-disc
  centred at the origin and $S =\set{z_1 z_2 = 0}$ (thus one can
  set $\psi := \log\abs{z_1}^2 +\log\abs{z_2}^2 -1$).
  Take $\ideal a =\holo_X$ and $\vphi_{\ideal a^c} = 0$ for any
  $c \geq 0$.
  Taking $\pi$ to be the blow-up at the origin of $X$ with exceptional
  divisor $E$, it follows easily that $\ELAdjidlof{\holo_X} =\pi_*\holo_{\rs
    X} \paren{-E} = \maxidl_{X,\vect 0}$, the defining ideal sheaf of
  the origin $\vect 0$ in $X$.
  However, the integral
  \begin{equation*}
    \int_{\Delta^2} \frac{\dvol_{\Delta^2}}{\abs{z_1z_2}^2
      \paren{\log\frac{\abs{z_1}^{\mathrlap 2}}e
        \:\:\log\frac{\abs{z_2}^{\mathrlap 2}}e \:}^2}
    =\paren{\int_{\Delta} \frac{\pi\ibar dz\wedge d\conj z}{\abs z^2
        \paren{\log\frac{\abs z^2}e}^2} }^2
  \end{equation*}
  being convergent implies that Guenancia's adjoint ideal sheaf is
  simply $\holo_X$ in this case.
\end{example}

\begin{remark} \label{rem:flaw-in-Guenancia}
  The flaw in the claim \cite{Guenancia}*{Prop.~2.11} comes from the
  proof of \cite{Guenancia}*{Lemma 2.12}.
  The argument that there exists $0 < \epsilon' \leq \epsilon$ such
  that $\lambda_k(\epsilon') >-1$ for all $k > p$ (in the notation in
  \cite{Guenancia}*{Lemma 2.12}) can be seen false easily by
  considering the situation in Example
  \ref{example:Guenancia-neq-EL}.
  See the Erratum \cite{Guenancia_AdjIdl-Erratum} by Guenancia
  for the corrected proof of the claim when $S$ is smooth.
  While Guenancia's adjoint ideal sheaf does not coincide with the
  Ein--Lazarsfeld adjoint ideal sheaf in general when $S$ has more
  than one intersecting components, the two ideal sheaves coincide
  when $S$ has only one irreducible component, even if it is not
  smooth (a setup studied in \cite{Li_adj-idl-II} and
  \cite{KimDano&Seo_adj-idl}).
  A proof can be obtained by following the argument in the last part
  of the proof of Theorem \ref{thm:comparison-alg-and-analytic-aidl}.
\end{remark}

The second example gives an instance of the algebraic and analytic
adjoint ideal sheaves when $(X,S)$ is no longer lc.

\begin{example}[Non-lc case] \label{example:adjidl-neq-in-non-lc-case}
  Let $X = \Delta^2 \subset \fieldC^2$ again be the unit $2$-disc and
  $S = \set{z_2^2 =z_1^3}$ be the cuspidal curve.
  Let $\pi \colon \rs X \to X$ be the standard resolution of the singularity
  at the origin $\vect 0 = (0,0)$ such that
  \begin{equation*}
    K_{\rs X} \sim \pi^*K_X +E_1 +2E_2 +4E_3
    \quad\text{ and }\quad
    \pi^*S = \pi^{-1}_*S +2E_1 +3E_2 +6E_3 \; ,
  \end{equation*}
  where $E_1$, $E_2$ and $E_3$ are exceptional (prime) divisors.
  The pair $(X,S)$ is apparently not lc.
  It follows immediately from \eqref{eq:definition-alg-HMAdjidl} and
  \eqref{eq:definition-alg-ELAdjidl} that 
  \begin{equation*}
    \ELAdjidlof{\holo_X} =\pi_*\holo_{\rs X} \paren{-E_1 -E_2 -2E_3}
    \quad\text{ and }\quad
    \HMAdjidlof{\holo_X} =\pi_*\holo_{\rs X} \paren{-2E_3} \; .
  \end{equation*}

  Let $\psi =\psi_S =\log\abs{z_1^3 -z_2^2}^2 -1$.
  Notice that $S$ is the lc locus of the family
  $\set{\mtidlof<X>{m\psi}}_{m\in[0,1]}$ at jumping number $m=1$.
  Since, for $m \in [0,1]$,
  \begin{equation*}
    m\pi^*_\ominus\psi
    \sim_{\tlog} m\phi_{\pi^{-1}_*S} +(2m-1) \phi_{E_1}
    +(3m-2) \phi_{E_2} +(6m-4) \phi_{E_3} \; ,
  \end{equation*}
  one sees that $\rs S = \pi^{-1}_*S +E_1 +E_2 +E_3$
  and $m=\frac 56$ is the jumping number preceding $m=1$.
  Note that $\pi^{-1}_*S$, $E_1$ and $E_2$ are mutually disjoint,
  while $E_3$ intersects all 3 other lines exactly once.
  Therefore, $\lcc[1]<\rs X>(\rs* S) =\rs S$ and $\lcc[2]<\rs X>(\rs*
  S) =\set{p_0, p_1, p_2}$, a set of 3 distinct points.
  A computation as before shows that 
  \begin{align*}
    \mtidlof<X>{\frac 56 \psi}
    =&~\aidlof|2|<X>{0} =\pi_*\holo_{\rs X}\paren{-E_3} \; , \\
    &~\aidlof|1|<X>{0}
    =\pi_*\paren{\holo_{\rs X}\paren{-E_3} \cdot \defidlof{\set{p_0, p_1, p_2}}} \; , \\
    \mtidlof<X>{\psi}
    =&~\aidlof|0|<X>{0} =\pi_*\holo_{\rs X}\paren{-\pi^{-1}_*S -E_1
      -E_2 -2E_3} \; ,
  \end{align*}
  where $\defidlof{\set{p_0, p_1, p_2}}$ is the defining ideal sheaf
  of the 3 points in $\rs X$.
  Notice that one has $\HMAdjidlof{\holo_{X}} \subset \aidlof|1|{0}$ a
  priori
  (although the two sheaves are indeed equal by the computation below).
  Since
  \begin{equation*}
    \divsr{\pi^*z_1} = S_1' +E_1 +E_2 +2E_3
    \quad\text{ and }\quad
    \divsr{\pi^*z_2} = S_2' +E_1 +2E_2 +3E_3 \; ,
  \end{equation*}
  where $S_i'$ is the proper transform of $\set{z_i=0}$, it can be
  checked directly that
  \begin{gather*}
    \ELAdjidlof{\holo_X}
    =\HMAdjidlof{\holo_X}
    =\aidlof|1|<X>{0} 
    =\aidlof|2|<X>{0}
    =\genby{z_1, z_2}
    \\
    \text{and }\quad
    \aidlof|0|<X>{0} = \mtidlof<X>{\psi} = \defidlof{S} \; .
  \end{gather*}

  When the family $\set{\mtidlof<X>{m\psi}}_{m\geq 0}$ at jumping
  number $m=\frac 56$ is considered, one can find that, using the
  same calculation, $\rs S = E_3$ and the corresponding adjoint ideal
  sheaves are given by  
  \begin{equation*}
    \aidlof|0|<X>{0}[\frac 56 \psi] =\mtidlof<X>{\frac 56 \psi}
    =\genby{z_1, z_2}
    \quad\text{ and }\quad
    \aidlof|1|<X>{0}[\frac 56 \psi] =\mtidlof<X>{\frac 23 \psi}
    =\holo_X \; .
  \end{equation*}
\end{example}






\end{document}